\documentclass[preprint,hidelinks,onefignum,onetabnum]{siamart220329}

\usepackage{lipsum}
\usepackage{amsfonts}
\usepackage{graphicx}
\usepackage{epstopdf}
\usepackage{soul,cancel}
\usepackage{algorithmic}
\ifpdf
\DeclareGraphicsExtensions{.eps,.pdf,.png,.jpg}
\else
\DeclareGraphicsExtensions{.eps}
\fi

\newsiamremark{remark}{Remark}
\newsiamremark{hypothesis}{Hypothesis}
\crefname{hypothesis}{Hypothesis}{Hypotheses}
\newsiamthm{claim}{Claim}
\newcommand\revisions[1]{{#1}}

\headers{Explicit symmetric low-regularity methods for NLSE}{Y. Feng, G. Maierhofer, and C. Wang}

\title{Explicit symmetric low-regularity integrators for the nonlinear Schr\"odinger equation\thanks{Submitted to the editors DATE.
		\funding{{The first author gratefully acknowledges funding from the National Natural Science Foundation of China (grant no. 12401539). The second author gratefully acknowledges funding from Oxford Mathematical Institute. The third author was partially supported by the Ministry of Education of Singapore under its AcRF Tier 2 funding MOE-T2EP20122-0002 (A-8000962-00-00).}}
}}

\author{Yue Feng\thanks{School of Mathematics and Statistics, Xi’an Jiaotong University (\email{yue.feng@xjtu.edu.cn})} \and 
	Georg Maierhofer\thanks{Mathematical Institute, University of Oxford (\email{gam37@cam.ac.uk})} \and Chushan Wang\thanks{Department of Mathematics, National University of Singapore (\email{chushanwang@u.nus.edu})}}

\usepackage{amsopn}

\ifpdf
\hypersetup{
	pdftitle={An Example Article},
	pdfauthor={D. Doe, P. T. Frank, and J. E. Smith}
}
\fi

\usepackage{amsmath,amssymb,amsfonts}
\newcommand{\vphis}{\varphi_\text{s}}
\newcommand{\R}{\mathbb{R}}
\newcommand{\rmd}{\mathrm{d}}
\newcommand{\vep}{\varepsilon}
\newcommand{\vphi}{\varphi}

\newcommand{\bx}{\textbf{x}}
\newcommand{\vx}{\mathbf{x}}

\newtheorem{assumption}[theorem]{Assumption}

\usepackage{xcolor}

\usepackage{enumerate}

\usepackage{bm}

\usepackage{mathdots}

\renewcommand{\Re}{\mathrm{Re}}

\usepackage{mathrsfs}

\usepackage{tikz}
\usetikzlibrary{calc}
\usetikzlibrary{decorations}
\usetikzlibrary{positioning}
\usetikzlibrary{shapes}
\usetikzlibrary{external}

\newif\ifdark
\darkfalse

\ifdark
\definecolor{darkred}{rgb}{0.9,0.2,0.2}
\definecolor{darkblue}{rgb}{0.7,0.3,1}
\definecolor{darkgreen}{rgb}{0.1,0.9,0.1}
\definecolor{franck}{rgb}{0,0.8,1}
\definecolor{pagebackground}{rgb}{.15,.21,.18}
\definecolor{pageforeground}{rgb}{.84,.84,.85}
\pagecolor{pagebackground}
\AtBeginDocument{\globalcolor{pageforeground}}
\definecolor{symbols}{rgb}{0,0.7,1}
\colorlet{connection}{red!80!black}
\colorlet{boxcolor}{blue!50}

\else

\definecolor{darkred}{rgb}{0.7,0.1,0.1}
\definecolor{darkblue}{rgb}{0.4,0.1,0.8}
\definecolor{darkgreen}{rgb}{0.1,0.7,0.1}
\definecolor{franck}{rgb}{0,0,1}
\definecolor{pagebackground}{rgb}{1,1,1}
\definecolor{pageforeground}{rgb}{0,0,0}
\colorlet{symbols}{blue!90!black}
\colorlet{connection}{red!30!black}
\colorlet{boxcolor}{blue!50!black}

\fi

\tikzstyle{tinydots}=[dash pattern=on \pgflinewidth off \pgflinewidth]
\tikzstyle{superdense}=[dash pattern=on 4pt off 1pt]

\tikzset{
	cross/.style={path picture={ 
			\draw[symbols]
			(path picture bounding box.south east) -- (path picture bounding box.north west) (path picture bounding box.south west) -- (path picture bounding box.north east);
	}},
	root/.style={circle,fill=green!50!black,inner sep=0pt, minimum size=1.2mm},
	dot/.style={circle,fill=pageforeground,inner sep=0pt, minimum size=1mm},
	dotred/.style={circle,fill=pageforeground!50!pagebackground,inner sep=0pt, minimum size=2mm},
	var/.style={circle,fill=pageforeground!10!pagebackground,draw=pageforeground,inner sep=0pt, minimum size=3mm},
	kernel/.style={semithick,shorten >=2pt,shorten <=2pt},
	kernels/.style={snake=zigzag,shorten >=2pt,shorten <=2pt,segment amplitude=1pt,segment length=4pt,line before snake=2pt,line after snake=5pt,},
	rho/.style={densely dashed,semithick,shorten >=2pt,shorten <=2pt},
	testfcn/.style={dotted,semithick,shorten >=2pt,shorten <=2pt},
	renorm/.style={shape=circle,fill=pagebackground,inner sep=1pt},
	labl/.style={shape=rectangle,fill=pagebackground,inner sep=1pt},
	xic/.style={very thin,circle,draw=symbols,fill=symbols,inner sep=0pt,minimum size=1.2mm},
	g/.style={very thin,rectangle,draw=symbols,fill=symbols!10!pagebackground,inner sep=0pt,minimum width=2.5mm,minimum height=1.2mm},
	xi/.style={very thin,circle,draw=symbols,fill=symbols!10!pagebackground,inner sep=0pt,minimum size=1.2mm},
	xies/.style={very thin,rectangle,fill=green!50!black!25,draw=symbols,inner sep=0pt,minimum size=1.1mm},
	xiesf/.style={very thin,rectangle,fill=green!50!black,draw=symbols,inner sep=0pt,minimum size=1.1mm},
	xix/.style={very thin,crosscircle,fill=symbols!10!pagebackground,draw=symbols,inner sep=0pt,minimum size=1.2mm},
	X/.style={very thin,cross,rectangle,fill=pagebackground,draw=symbols,inner sep=0pt,minimum size=1.2mm},
	xib/.style={thin,circle,fill=symbols!10!pagebackground,draw=symbols,inner sep=0pt,minimum size=1.6mm},
	xie/.style={thin,circle,fill=green!50!black,draw=symbols,inner sep=0pt,minimum size=1.6mm},
	xid/.style={thin,circle,fill=symbols,draw=symbols,inner sep=0pt,minimum size=1.6mm},
	xibx/.style={thin,crosscircle,fill=symbols!10!pagebackground,draw=symbols,inner sep=0pt,minimum size=1.6mm},
	kernels2/.style={very thick,draw=connection,segment length=12pt},
	keps/.style={thin,draw=symbols,->},
	kepspr/.style={thick,draw=connection,->},
	krho/.style={thin,draw=symbols,superdense,->},
	krhopr/.style={thick,draw=connection,superdense},
	triangle/.style = { regular polygon, regular polygon sides=3},
	not/.style={thin,circle,draw=connection,fill=connection,inner sep=0pt,minimum size=0.5mm},
	diff/.style = {very thin,draw=symbols,triangle,fill=red!50!black,inner sep=0pt,minimum size=1.6mm},
	diff1/.style = {very thin,dectriangle={1}{0},fill=red!50!black,draw=symbols,inner sep=0pt,minimum size=1.6mm},
	diff2/.style = {very thin,dectriangle={1}{1},fill=red!50!black,draw=symbols,inner sep=0pt,minimum size=1.6mm},
	diffmini/.style = {very thin,rectangle,fill=black,draw=black,inner sep=0pt,minimum size=0.75mm},
	kernelsmod/.style={very thick,draw=connection,segment length=12pt},
	rec/.style = {very thin,rectangle,fill=black,draw=black,inner sep=0pt,minimum size=2mm},
	cerc/.style={very thin,circle,draw=black,fill=symbols,inner sep=0pt,minimum size=2mm},
	stars/.style={very thin,star,star points=6,star point ratio=0.5, draw=black,fill=red,inner sep=0pt,minimum size=0.7mm},
	>=stealth,
}

\def\DeclareSymbol#1#2#3{%
	\expandafter\gdef\csname MH@symb@#1\endcsname{\tikzsetnextfilename{symbol#1}%
		\tikz[baseline=#2,scale=0.15,draw=symbols,line join=round]{#3}}%
	\expandafter\gdef\csname MH@symb@#1s\endcsname{\scalebox{0.75}{\tikzsetnextfilename{symbol#1}%
			\tikz[baseline=#2,scale=0.15,draw=symbols,line join=round]{#3}}}%
	\expandafter\gdef\csname MH@symb@#1ss\endcsname{\scalebox{0.65}{\tikzsetnextfilename{symbol#1}%
			\tikz[baseline=#2,scale=0.15,draw=symbols,line join=round]{#3}}}%
}

\DeclareSymbol{thin}{1.4}{\draw[pagebackground] (-0.3,0) -- (0.3,0); \draw  (0,0) -- (0,2);}
\DeclareSymbol{thinn}{1.4}{\draw[pagebackground] (-0.3,0) -- (0.3,0); \draw[tinydots]  (0,0) -- (0,2);}

\DeclareSymbol{thick}{1.4}{\draw[pagebackground] (-0.3,0) -- (0.3,0); \draw[kernels2]  (0,0) -- (0,2);}
\DeclareSymbol{thickk}{1.4}{\draw[pagebackground] (-0.3,0) -- (0.3,0); \draw[kernels2]  (0,0) -- (0,2);}

\DeclareSymbol{thick2}{1.4}{\draw[pagebackground] (-0.3,0) -- (0.3,0); \draw[kernels2,tinydots]  (0,0) -- (0,2);}

\newcommand{\thick}{\csname MH@symb@thick\endcsname}
\newcommand{\thickk}{\csname MH@symb@thick2\endcsname}
\newcommand{\thin}{\csname MH@symb@thin\endcsname}
\newcommand{\thinn}{\csname MH@symb@thinn\endcsname}
\begin{document}
	
	\maketitle
	\begin{abstract}
		The numerical approximation of low-regularity solutions to the nonlinear Schr\"odinger equation is notoriously difficult and even more so if structure-preserving schemes are sought. Recent {works have} been successful in establishing symmetric low-regularity integrators for this equation. However, so far, all prior symmetric low-regularity algorithms are {fully} implicit, and therefore require the solution of a nonlinear equation at {each time step}, leading to significant numerical cost in the iteration. In this work, we introduce the first fully explicit (multi-step) symmetric low-regularity integrators for the nonlinear Schr\"odinger equation. We demonstrate the construction of an entire class of such schemes which notably can be used to symmetrise (in explicit form) a large amount of existing low-regularity integrators. We provide rigorous convergence analysis of our schemes and numerical examples demonstrating both the favourable structure preservation properties obtained with our novel schemes, and the significant reduction in computational cost over implicit methods.
	\end{abstract}
	
	\begin{keywords}
		Nonlinear Schr\"odinger equation, low regularity, {explicit symmetric schemes}, unconditionally stable
	\end{keywords}
	
	\begin{MSCcodes}
		65C30, 65M70, 35Q41, 35Q55
	\end{MSCcodes}
	\section{Introduction} 
	We consider the numerical approximation of low-regularity solutions to the nonlinear Schr\"odinger equation (NLSE) with cubic nonlinearity on the torus $\mathbb{T}^d \ (d=1, 2, 3)$ in the following form:
	\begin{equation}
		\label{NLSE}
		\begin{cases}i \partial_t u = -\Delta u +\mu |u|^{2} u, &\quad  t > 0, {\quad \bx} \in\mathbb{T}^d,\\
			u(0, \bx) = u_0(\bx), & \quad \bx\in\mathbb{T}^d,
		\end{cases}
	\end{equation} 
	where $u=u(t,\bx)$ and $\mu$ is a positive/negative constant corresponding to the defocusing/focusing NLSE. The NLSE can be used to model a wide range of physical phenomena including wave propagation in fiber optics and deep water waves \cite{sulem2007nonlinear,tai1986observation,thomas2012nonlinear}. With minor modifications the numerical methods introduced in this work can also be extended to the case of the Gross-Pitaevskii equation (GPE) (cf. \cite{bronsard2023error}) which is derived from the mean-field approximation of many-body problems in quantum physics and chemistry, especially in the modeling and simulation of Bose-Einstein condensation \cite{bao2012mathematical}.
    The NLSE \eqref{NLSE} conserves the {\sl mass}
     \begin{equation}\label{eqn:mass}
        M(u(t, \cdot)) = \int_{\mathbb{T}^d}|u(t, \bx)|^2 d \bx \equiv M(u_0),\quad t \geq 0,
    \end{equation}
    and the {\sl energy}
    \begin{equation}\label{eqn:energy}
        E(u(t, \cdot)) = \int_{\mathbb{T}^d} \left[|\nabla u(t, \bx)|^2 + \frac{\mu}{2} |u(t, \bx)|^4\right] d \bx \equiv E(u_0),\quad t \geq 0.
    \end{equation}
    
    In this paper, we are interested in the design of numerical schemes for the NLSE \eqref{NLSE} with low-regularity initial data (which also results in the low regularity of solutions), meaning that $u_0\in H^{\alpha}$, where $H^{\alpha}=H^\alpha(\mathbb{T}^d)$ is the periodic Sobolev space of order $\alpha$, for some $\alpha>0$ of small magnitude. The specific regularity assumptions (and thus the minimum value of $\alpha$ required for the convergence of our schemes) will become apparent throughout Sections~\ref{sec:explicit-low_regularity_integrators} \& \ref{sec:error_analysis}. For the NLSE with sufficiently smooth data, various accurate and efficient numerical methods have been proposed and analyzed in the past few years, including {finite difference methods \cite{akrivis1993finite,antoine2013computational,bao2013optimal}, exponential integrators \cite{bao2014uniform,celledoni2008symmetric,hochbruck2010exponential}, and time splitting methods \cite{bao2023improved,bao2003numerical,besse2002order,lubich2008splitting}.} However, these classical numerical methods generally require much higher regularity of the solution to converge at desired optimal order. {In the past decade, numerical approximation to nonsmooth solutions of the NLSE and other dispersive equations with low-regularity solutions has received significant attention in the numerical analysis community.} The first low-regularity integrator (LRI) treating this approximation for the NLSE was introduced by Ostermann \& Schratz \cite{ostermann_schratz2018low}, and this was followed by higher order methods constructed by Ostermann et al. \cite{ostermann2022second} and Bruned \& Schratz \cite{bruned_schratz_2022}, and integrators with lower regularity assumptions by Cao et al. \cite{cao2024new}. These initial constructions on a torus have also been extended to non-periodic boundary conditions by Alama Bronsard \cite{bronsard2023error} and Bai et al. \cite{bai2022constructive}, and a fully discrete error analysis was provided by Li \& Wu \cite{li2021fully} and Ostermann \& Yao \cite{ostermann2022fully}. \revisions{Later, a sequence of works by Alama Bronsard \cite{alama2023symmetric}, Alama Bronsard et al. \cite{alamabronsard_et_al23}, Banica et al. \cite{banicamaierhoferschratz22}, Feng et al. \cite{feng_maierhofer_schratz_2023}, and Maierhofer \& Schratz \cite{maierhofer_schratz_24} focussed on constructing structure preserving low-regularity integrators for the NLSE and other dispersive systems.} These works have demonstrated that it is possible to develop integrators that maintain favourable convergence behaviour for low-regularity data, while introducing structure preservation properties, in particular constructing symmetric low-regularity integrators. {The favourable long-time behaviour of such symmetric methods is rigorously understood when applied to finite-dimensional integrable reversible systems \cite{hairer2004symmetric,hairer2013geometric}. Even though the rigorous understanding of the long-time behaviour of such methods in the case of PDEs is much less straightforward \revisions{(cf.  Cohen et al. \cite{cohen2008conservation}, Faou \cite{faou2012geometric} and  Gauckler \& Lubich \cite{gauckler2010splitting})}, numerical observations of favourable structure preservation properties have motived the study of symmetric methods for PDEs for several decades, for example in the classification of symmetric splitting methods \cite{mclachlan2002splitting} and symmetric exponential integrators \cite{celledoni2008symmetric}.} However, in all the aforementioned works {on symmetric low-regularity integrators}, the introduction of symmetry into the integrators required the methods to become implicit, thus incurring the additional computational cost of the solution of a nonlinear equation at each time step. Recent work in the construction of exponential integrators has successfully overcome this issue by designing (and analysing) {multi-step} methods that are fully explicit, which maintain similar favourable features as their single step analogues but are also symmetric. This construction was performed in particular for the NLSE by Bao \& Wang \cite{bao2024explicit}, for the nonlinear Dirac equation by Jahnke \& Kirn \cite{jahnke2023numerical}, and for the dispersion-managed nonlinear Schrödinger equation by Jahnke \& Mikl \cite{jahnke2018adiabatic,jahnke2019adiabatic}.
	
	In this work, we seek to take on board the approach taken in these recent constructions and we design what are, to the best of our knowledge, the first fully explicit symmetric low-regularity integrators for the nonlinear Schr\"odinger equation. We study the convergence properties of these new schemes and demonstrate that efficient computations are possible while maintaining both guaranteed low-regularity convergence and symmetry of the method, the latter of which leads to favourable long-time behaviour of the methods observed in numerical experiments.
	
	The remainder of this paper is structured as follows. In Section~\ref{sec:explicit-low_regularity_integrators}, we introduce the construction of our novel explicit two-step low-regularity schemes based on the availability of single-step non-symmetric low-regularity integrators of specific form. {In particular}, we provide several examples of existing non-symmetric low-regularity integrators available in the literature which fit this framework in Section~\ref{sec:examples_construction_slri}. {The convergence analysis of these methods is presented in Section \ref{sec:error_analysis}, where a general framework is introduced in Section \ref{sec:general_convergence} and the detailed error estimates for specific examples are carried out in Sections~\ref{sec:application_to_sLRI1} \& \ref{sec:application_to_sLRI2}.} We provide detailed numerical experiments highlighting the advantages of our new methodology in Section~\ref{sec:numerical_experiments}. Finally, some concluding remarks and directions for future research are provided in Section~\ref{sec:conclusions}. {Throughout the paper, we denote by $ C $ a generic positive constant independent of the mesh size $ h $ and time step size $ \tau $, and by $ C(M) $ a generic positive constant depending on the parameter $ M $. The notation $ A \lesssim B $ is used to represent that there exists a generic constant $ C>0 $, such that $ |A| \leq CB $.}
	
	\section{Explicit symmetric low-regularity integrators}\label{sec:explicit-low_regularity_integrators}
	To begin with, we introduce some relevant preliminaries that will help our exposition of the construction of explicit symmetric low-regularity integrators in Sections~\ref{sec:idea_on_simple_integrator} \& \ref{sec:general_construction_of_schemes}. 
  We shall use the following convention for the Fourier transform and these spaces: For a function $u:\mathbb{T}^{d}\mapsto\mathbb{C}$, {we denote its Fourier expansion by $u(\vx)=\sum_{\textbf{k}\in\mathbb{Z}^d}\hat{u}_{\textbf{k}}e^{i\textbf{k}\cdot \textbf{x}}$ whenever it is defined. The periodic Sobolev norm of order $\alpha\geq 0$ is then denoted by
		$\|u\|_{H^{\alpha}}^2:=\sum_{\textbf{k}\in\mathbb{Z}^d}(1+|\textbf{k}|^2)^{\alpha}|\hat{u}_{\textbf{k}}|^2,$
		and the periodic Sobolev space of order $\alpha$ consists of all $L^2(\mathbb{T}^d)$ functions for which $\|\cdot\|_{H^{\alpha}}$ is bounded.} 
	Finally, we recall the definition of symmetric numerical schemes for two-step methods:
	{\begin{definition}\label{def:symmetric_scheme} 
			A two-step numerical scheme $u^{n+1} = \Phi^\text{\rm 2-step}_\tau(u^n, u^{n-1}) \ (n \geq 1)$ with $\Phi_{\tau}^\text{\rm 2-step}$ being the numerical flow is called \textbf{symmetric} if {it satisfies
            \begin{equation}
                u^{n-1} = \Phi_{-\tau}^{\text{\rm 2-step}}(u^n,  \Phi_{\tau}^{\text{\rm 2-step}}(u^n, u^{n-1})), \quad {\text{for all}} \quad u^{n-1}, u^n, \tau.
            \end{equation}
            }
	\end{definition}}
 In the above, $\Phi_{-\tau}^\text{\rm 2-step}$ is formally defined by simply replacing $\tau$ with $-\tau$ in the definition of the scheme, and we assume that the corresponding map is well-defined. This is the case for all the schemes we consider because the NLSE is time-reversible.
	\subsection{Construction of a specific explicit symmetric LRI}\label{sec:idea_on_simple_integrator}
	Before describing the construction more generally, let us begin with a specific example to introduce the main ideas {for the explicit symmetric LRI.}
	
	Choose a time step size $ \tau > 0 $ and denote time steps as $ t_n = n\tau $ for $ n = 0, 1, \dots $. By Duhamel's formula, whenever the exact solution $u$ of the NLSE \cref{NLSE} exists, it satisfies (by using the short notation $u(t)=u(t, \vx)$)
	\begin{align}\label{eq:duhamel_exact}
		u(t_n + \zeta) = e^{i\zeta\Delta} u(t_n)  - i\mu \int_0^\zeta e^{i(\zeta -s )\Delta} \left(|u(t_n + s)|^2 u(t_n+s)\right) \rmd s, \quad \zeta \in \R. 
	\end{align}
	According to the construction in \cite{bao2024explicit}, multiplying $e^{-i\zeta\Delta}$ on both sides of \cref{eq:duhamel_exact}, we obtain
	\begin{equation}\label{eq:duhamel_modified}
		e^{-i\zeta\Delta} u(t_n + \zeta) = u(t_n) - i\mu \int_0^\zeta e^{-is\Delta} \left(|u(t_n + s)|^2 u(t_n+s)\right) \rmd s, \quad \zeta \in \R.
	\end{equation}
	Taking $\zeta = \tau$ and $\zeta=-\tau$ in \cref{eq:duhamel_modified}, we get, respectively, for $n \geq 1$, 
	\begin{align}\label{eq:exact_plus}
		e^{-i\tau\Delta} u(t_{n+1}) &= u(t_n) - i\mu \int_0^\tau e^{-is\Delta} \left(|u(t_n + s)|^2 u(t_n+s)\right) \rmd s,\\
		e^{i\tau\Delta} u(t_{n-1}) 
		&= u(t_n) - i\mu \int_0^{-\tau} e^{-is \Delta} \left(|u(t_n + s)|^2 u(t_n+s)\right) \rmd s \notag \\
		&= u(t_n) + i\mu \int_{-\tau}^0 e^{-is \Delta} \left(|u(t_n + s)|^2 u(t_n+s)\right) \rmd s. \label{eq:exact_minus}
	\end{align}
	Subtracting \cref{eq:exact_minus} from \cref{eq:exact_plus}, we obtain
	\begin{equation}\label{eq:exact_symm}
		e^{-i \tau \Delta} u(t_{n+1}) - e^{i\tau\Delta} u(t_{n-1}) = -i\mu \int_{-\tau}^{\tau} e^{-is\Delta} \left(|u(t_n + s)|^2 u(t_n+s)\right) \rmd s.   
	\end{equation}
    {Note \eqref{eq:exact_symm} is simply Duhamel's formula applied on the interval $[t_{n-1},t_{n+1}]$ but the above derivation is instructive in light of the general construction presented in Section~\ref{sec:general_construction_of_schemes}.}
	Multiplying $e^{i \tau \Delta}$ on both sides, we have
	\begin{equation}\label{eq:exact}
		u(t_{n+1}) = e^{2i\tau\Delta} u(t_{n-1}) -i\mu e^{i\tau\Delta} \int_{-\tau}^{\tau} e^{-is\Delta} \left(|u(t_n + s)|^2 u(t_n+s)\right) \rmd s. 
	\end{equation}
	
	Then we use the technique proposed in \cite{ostermann_schratz2018low} to obtain an approximation to the integral above at low regularity.  
	Define
	\begin{align*}
		I^\tau := \int_{-\tau}^{\tau} e^{-is\Delta} \left(|u(t_n + s)|^2 u(t_n+s)\right) \rmd s = \int_{-\tau}^{\tau} e^{-is\Delta} \left(u(t_n + s)^2 \overline{u(t_n+s)} \right) \rmd s. 
	\end{align*}
	By iterating the Duhamel's formula, i.e. by substituting \eqref{eq:duhamel_exact} with $\zeta = s$ into the integral above, omitting terms involving two or more integrals, and rewriting the integral in the Fourier space, we obtain
    \begin{align}\label{eq:int1}
			I^\tau 
			&\approx \int_{-\tau}^{\tau} e^{-is\Delta} \left((e^{is\Delta}u(t_n))^2 (e^{-is\Delta}\overline{u(t_n)}) \right) \rmd s \notag \\
			&=\sum_{\textbf{k}=\textbf{k}_1+\textbf{k}_2-\textbf{k}_3}\int_{-\tau}^{\tau} e^{is{|\textbf{k}|}^2}\left(e^{-is(|\textbf{k}_1|^2 + |\textbf{k}_2|^2 - |\textbf{k}_3|^2)} \widehat{u}_{\textbf{k}_1}(t_n) \widehat{u}_{\textbf{k}_2}(t_n) \overline{\widehat{u}}_{\textbf{k}_3}(t_n)\right) \rmd s \  e^{i\textbf{k}\cdot \textbf{x}} \notag \\
			&= \sum_{\textbf{k}=\textbf{k}_1+\textbf{k}_2-\textbf{k}_3} \int_{-\tau}^{\tau} e^{is(\mathcal{L}_\text{dom} + \mathcal{L}_\text{low})} \rmd s \  \widehat{u}_{\textbf{k}_1}(t_n) \widehat{u}_{\textbf{k}_2}(t_n) \overline{\widehat{u}}_{\textbf{k}_3}(t_n) e^{i\textbf{k}\cdot\textbf{x}} =: \widetilde I^\tau, 	
		\end{align}
		where $\mathcal{L}_\text{dom} = 2 {|\textbf{k}_3|}^2$, and $ \mathcal{L}_\text{low} = 2(\textbf{k}_1\cdot \textbf{k}_2 - \textbf{k}_1\cdot \textbf{k}_3 - \textbf{k}_2\cdot \textbf{k}_3)$. Then we apply the approximation $\exp(is(\mathcal{L}_\text{dom} + \mathcal{L}_\text{low})) \approx \exp(is\mathcal{L}_\text{dom})$ 
		to \cref{eq:int1} and get a further approximation 
		\begin{align}\label{eq:approxmation}
			\widetilde I^\tau
			&\approx \sum_{\textbf{k}=\textbf{k}_1+\textbf{k}_2-\textbf{k}_3} \int_{-\tau}^{\tau} e^{is\mathcal{L}_\text{dom}} \rmd s \  \widehat{u}_{\textbf{k}_1}(t_n) \widehat{u}_{\textbf{k}_2}(t_n) \overline{\widehat{u}}_{\textbf{k}_3}(t_n) {e^{i\textbf{k}\cdot\textbf{x}}} \notag \\
			&= \sum_{\textbf{k}=\textbf{k}_1+\textbf{k}_2-\textbf{k}_3} 2 \tau \vphi_s(2\tau |\textbf{k}_3|^2)  \widehat{u}_{\textbf{k}_1}(t_n) \widehat{u}_{\textbf{k}_2}(t_n) \overline{\widehat{u}}_{\textbf{k}_3}(t_n) {e^{i\textbf{k}\cdot\textbf{x}}} \notag \\
			&= 2 \tau (u(t_n)^2) \vphi_s(2 \tau \Delta) \overline{u(t_n)}, 
		\end{align}
		where $\vphi_s(x) = \sin(x)/x$ for $x \in \R$.
	Plugging \cref{eq:approxmation} into \cref{eq:exact}, we have
	\begin{equation}
		u(t_{n+1}) \approx e^{2i\tau\Delta} u(t_{n-1}) - 2i\mu\tau e^{i\tau\Delta} ( u(t_n)^2 \vphis(2\tau\Delta) \overline{u(t_n)} ), \quad n \geq 1.  
	\end{equation}
	For the first step, we can use the same low regularity approximation mentioned above, which corresponds to the first-order LRI (LRI1) in \cite{ostermann_schratz2018low} as
		\begin{equation}\label{eqn:ostermann_schratz_lri}
			u^1 = \Phi^{\text{LRI1}}_\tau(u_0) := e^{i \tau \Delta} u_0 - i \mu \tau e^{i \tau \Delta} \left( u_0^2 \vphi_1(- 2 i \tau \Delta) \overline{u_0} \right), 
		\end{equation}
		where $\vphi_1(z) = (e^z-1)/z$ for $z \in \mathbb{C}$. Let $u^n$ be the numerical approximation to $u(t_n)$ for $n \geq 0$. Then we obtain an explicit symmetric LRI (sLRI1) as
		\begin{equation}\label{eq:sLRI1}
			\begin{aligned}
				&u^{n+1} = e^{2 i \tau \Delta} u^{n-1} - 2 i \mu \tau e^{i \tau \Delta} \left( (u^n)^2 \vphis(2 \tau \Delta) \overline{u^n} \right), \quad {n \geq 1}, \\
				&u^1 = \Phi^{\text{LRI1}}_\tau(u_0), \quad u^0 = u_0. 
			\end{aligned} 
		\end{equation}
		The sLRI1 \cref{eq:sLRI1} is an explicit two-step method, and it can be checked easily that it is symmetric in the sense of Definition~\ref{def:symmetric_scheme}. Moreover, similar to the sEWI in \cite{bao2024explicit}, the sLRI1 is unconditionally stable as we shall show later.
	
	\subsection{General construction of symmetrised two-step schemes}\label{sec:general_construction_of_schemes}
	It turns out that the above construction can be extended to obtain explicit symmetric low-regularity integrators from explicit non-symmetric methods more generally. 
	
	Fix a final time $0<T<T_\text{max}$ with $T_\text{max}$ being the maximum existing time of the solution to the NLSE \cref{NLSE}. Given a one-step numerical scheme for the NLSE \cref{NLSE}: $u^{n+1}=\Phi_{\tau}(u^n) \ (n \geq 0)$ with $\Phi_\tau$ being the numerical flow, we assume that the following properties hold for all $0<\tau\leq\tau_0$ with some $\tau_0>0$ possibly depending on $u_0$ and $T$ (see Section~\ref{sec:examples_construction_slri} for several examples). 
		\begin{assumption}\label{assumption:general_properties_higher_order_schemes}
			Let $ \Phi_t(v)=e^{it\Delta}v+\widetilde{\Phi}_{t}(v)$ for $v \in H^\alpha$ and $t=\pm \tau$.  
			\begin{enumerate}[(i)]
				\item \underline{The method is unconditionally stable,} i.e. there is a continuous function $M:\mathbb{R}\times \mathbb{R}\rightarrow \mathbb{R}_{\geq 0}$ such that for any $v,w\in H^{\alpha}$, we have
				\begin{align}\label{eqn:stability_general_tree_based-scheme}
					\|\widetilde{\Phi}_{\pm \tau}(v)-\widetilde{\Phi}_{\pm\tau}(w)\|_{H^\alpha}\leq {\tau M(\|v\|_{H^\alpha},\|w\|_{H^\alpha})}\|v-w\|_{H^\alpha}.
				\end{align}
				\item \underline{Local error of order $p+1$}, {i.e.} there is a continuous, increasing function $C:\mathbb{R}_{\geq0}\rightarrow\mathbb{R}_{\geq0}$ such that the local truncation error $\mathcal{R}^n_\tau := u(t_n+\tau)-\Phi_\tau(u(t_n))$ and $\mathcal{R}^{n+1}_{-\tau} := u(t_{n+1}-\tau)-\Phi_{-\tau}(u(t_{n+1}))$ for $0 \leq n \leq T/\tau-1$ satisfies
				\begin{equation}\label{eqn:local_error_general_tree_based-scheme}
					\|\mathcal{R}^n_{\tau}\|_{H^\alpha} + \|\mathcal{R}^{n+1}_{-\tau}\|_{H^\alpha} \leq \tau^{p+1} C\left(\sup_{t \in [t_n, t_{n+1}]} \|u(t)\|_{H^{\alpha+\gamma_1}}\right). 
				\end{equation}
				\item \underline{Improved local error due to symmetry {when $p$ is odd},} i.e. there is a continuous, increasing function $\widetilde{C}: \mathbb{R}_{\geq0}\rightarrow\mathbb{R}_{\geq0}$ such that,  for $1 \leq n \leq T/\tau-1$, 
				\begin{align}\label{eqn:improved_local_error_general_tree_based-scheme}
					\|\mathcal{R}^n_{\tau}-e^{2i\tau\Delta}\mathcal{R}^n_{-\tau}\|_{H^\alpha}\leq \tau^{p+2} \widetilde{C}\left(\sup_{t \in [t_{n-1}, t_{n+1}]} \|u(t)\|_{H^{\alpha+\gamma_1+\gamma_2}}\right). 
				\end{align}
			\end{enumerate}
	\end{assumption}
{\begin{remark}
    As we shall see later, Assumption~\ref{assumption:general_properties_higher_order_schemes} (i) can actually be relaxed to the estimate
    \begin{align}\label{eqn:weaker_stability_estimate}
        \|\widetilde{\Phi}_{\pm \tau}(v)-\widetilde{\Phi}_{\pm\tau}(w)\|_{H^\alpha}\leq \tau M(v,w)\|v-w\|_{H^\alpha},
    \end{align}
    where $M$ is uniformly bounded for both the numerical flow and exact solution, i.e. there is a constant $M_T>0$ such that $M(u(t_n),u^n)\leq M_T$ for all $t_n\leq T$.
\end{remark}}
 Using such a {basic} scheme, we can follow the recipe from Section~\ref{sec:idea_on_simple_integrator} to construct an explicit symmetric two-step integrator of order at least $p$. Indeed, we have	
{\begin{align*}
		 &u(t_{n+1}) = e^{i\tau\Delta}u(t_n)+\widetilde{\Phi}_{\tau}(u(t_n))+\mathcal{R}^n_{\tau}, \\
		 &u(t_{n-1}) =e^{-i\tau\Delta}u(t_n)+\widetilde{\Phi}_{-\tau}(u(t_n))+\mathcal{R}^n_{-\tau}.    
	\end{align*}}
		Thus we have equivalently 
		\begin{align}\label{eqn:general_construction_forward_basis}
			e^{-i\tau\Delta}u(t_{n+1})&=u(t_n)+e^{-i\tau\Delta}\widetilde{\Phi}_{\tau}(u(t_n))+e^{-i\tau\Delta}\mathcal{R}_{\tau}^n,\\ \label{eqn:general_construction_backward_basis}
			e^{i\tau\Delta}u(t_{n-1})&=u(t_n)+e^{i\tau\Delta}\widetilde{\Phi}_{-\tau}(u(t_n))+e^{i\tau\Delta}\mathcal{R}_{-\tau}^n, 
		\end{align}
  and therefore, subtracting \eqref{eqn:general_construction_backward_basis} from \eqref{eqn:general_construction_forward_basis} {and applying $e^{i\tau\Delta}$ to the resulting equation}, we find
		\begin{align}\label{eqn:general_scheme_with_remainder_term}
			u(t_{n+1})
			&=e^{2i\tau\Delta}u(t_{n-1}) + \widetilde{\Phi}_{\tau}(u(t_n))\!-e^{2i\tau\Delta}\widetilde{\Phi}_{-\tau}(u(t_n))+\mathcal{R}_{\tau}^n - e^{2i\tau\Delta}\mathcal{R}_{-\tau}^n.
		\end{align}
		This leads to the following natural definition of a corresponding {multi-step} method:
		\begin{equation}\label{eqn:general_scheme}
			\begin{aligned}
				&u^{n+1}=e^{2i\tau\Delta}u^{n-1} + \widetilde{\Phi}_{\tau}(u^n)-e^{2i\tau\Delta}\widetilde{\Phi}_{-\tau}(u^n), \quad n \geq 1, \\
				&u^{1} = e^{i\tau\Delta}u_0 + \widetilde \Phi_\tau(u_0), \quad u^0 = u_0. 
			\end{aligned}
		\end{equation}
		In fact, the explicit sLRI1 \cref{eq:sLRI1} constructed in the previous section is exactly \cref{eqn:general_scheme} with $\Phi_\tau = \Phi^\text{LRI1}_\tau $ given by \cref{eqn:ostermann_schratz_lri}. This will be further explained in the next subsection.
	\begin{remark}
		Although here we focus on the NLSE, the above construction is equally possible for other systems, and in particular can be exploited using the formalism in \cite{bruned_schratz_2022} to construct explicit symmetric low-regularity integrators for a large class of dispersive nonlinear PDEs.
	\end{remark}
 
	By the construction of the scheme \eqref{eqn:general_scheme}, we have the following symmetric property.
	\begin{proposition}
		The scheme \eqref{eqn:general_scheme} is symmetric in the sense of Definition~\ref{def:symmetric_scheme}.
	\end{proposition}
	
{ \begin{proof}
		{Multiplying $e^{-2i\tau\Delta}$ on both sides of the scheme \cref{eqn:general_scheme}, we have 
			\begin{align*}
				u^{n-1}=e^{-2i\tau\Delta} u^{n+1}+\widetilde{\Phi}_{-\tau}(u^n)-e^{-2i\tau\Delta}\widetilde{\Phi}_{\tau}(u^n).
			\end{align*}
			According to \cref{def:symmetric_scheme}, it is symmetric.}
	\end{proof}}
	
	\subsection{Applications of the general construction}\label{sec:examples_construction_slri}
	In this subsection, we apply the general construction \cref{eqn:general_scheme} to obtain a range of explicit symmetric LRIs for the NLSE \eqref{NLSE} through existing non-symmetric LRIs which satisfy Assumption~\ref{assumption:general_properties_higher_order_schemes} ({at least (i) and (ii)}) and can therefore be symmetrised as described in Section~\ref{sec:general_construction_of_schemes}.
	\paragraph{First-order LRI} 
		The first example is the first-order LRI1 \cref{eqn:ostermann_schratz_lri}, with the numerical flows for positive and negative time steps given by
		\begin{equation}\label{LRI1pm}
			\begin{aligned}
				&\Phi^\text{LRI1}_\tau(v) = e^{i\tau\Delta}v + \widetilde \Phi_\tau(v) :=  e^{i\tau\Delta}v - i \mu \tau e^{i\tau\Delta} \left(v^2 \vphi_1(-2i\tau\Delta) \overline{v}\right), \\
				&\Phi^\text{LRI1}_{-\tau}(v) = e^{-i\tau\Delta}v + \widetilde \Phi_{-\tau}(v) :=  e^{-i\tau\Delta}v + i \mu \tau e^{-i\tau\Delta} \left(v^2 \vphi_1(2i\tau\Delta) \overline{v}\right). 
			\end{aligned}
		\end{equation}
		Plugging \cref{LRI1pm} into \cref{eqn:general_scheme} yields exactly the sLRI1 \cref{eq:sLRI1}. By the analysis in \cite{ostermann_schratz2018low}, the numerical flow $\Phi_{\pm \tau} = \Phi^\text{LRI1}_{\pm \tau}$ satisfies, for $\alpha>d/2$,  \cref{assumption:general_properties_higher_order_schemes} (i) and (ii) with $p=1$ and $\gamma_1 = 1$. Since the LRI1 is of odd order, we shall show, for the first time, {that} the local truncation error also satisfies the refined estimate \cref{assumption:general_properties_higher_order_schemes} (iii) with $\gamma_2 = 2$ in Section \ref{sec:application_to_sLRI1}.
	
	\paragraph{Second-order LRI} 
	An example of an integrator satisfying the above assumptions with $p=2$ is given by the following second-order method, which was introduced by Bruned \& Schratz \cite{bruned_schratz_2022} and Ostermann et al. \cite{ostermann2022second}. The numerical flow for positive and negative time steps are given by
	\begin{equation}\label{eqn:bruned_schratz_lri}
		\begin{aligned}
			\Phi_{\tau}^\text{LRI2}(v)
			&=\mathrm{e}^{i \tau \Delta} v-i\mu \tau \mathrm{e}^{i \tau \Delta}\left(\left(v\right)^2[\varphi_1(-2 i \tau \Delta) {-} \varphi_2(-2 i \tau \Delta)] \overline{v}\right) \\
			&\quad -i \mu\tau\left(\mathrm{e}^{i \tau \Delta} v\right)^2\left[\mathrm{e}^{i \tau \Delta} \varphi_2(-2 i \tau \Delta) \overline{v}\right]-\mu^2\frac{\tau^2}{2} \mathrm{e}^{i \tau \Delta}\left[\left|v\right|^4 v\right], \\
			\Phi^\text{LRI2}_{-\tau}(v)
			&=\mathrm{e}^{-i \tau \Delta} v+i\mu \tau \mathrm{e}^{-i \tau \Delta}\left(\left(v\right)^2[\varphi_1(2 i \tau \Delta) {-} \varphi_2(2 i \tau \Delta)] \overline{v}\right) \\
			&\quad i \mu\tau\left(\mathrm{e}^{-i \tau \Delta} v\right)^2\left[\mathrm{e}^{-i \tau \Delta} \varphi_2(2 i \tau \Delta) \overline{v}\right]-\mu^2\frac{\tau^2}{2} \mathrm{e}^{-i \tau \Delta}\left[\left|v\right|^4 v\right],
		\end{aligned}
	\end{equation}
	where $\varphi_2(z)=(e^{z}-\varphi_1(z))/z$ for $z\in\mathbb{C}$. The local error of the method was studied in \cite{bruned_schratz_2022} and the stability of the scheme (together with a global convergence analysis) was proved in \cite{ostermann2022second} - further details are provided in Section~\ref{sec:application_to_sLRI2}. The recipe in Section~\ref{sec:general_construction_of_schemes} thus allows us to construct the following symmetric low-regularity integrator (sLRI2):
{	\begin{equation}\label{eq:sLRI2}
		\begin{aligned}
			u^{n+1}
			&=e^{2i\tau\Delta}u^{n-1}-i\mu 2\tau e^{i\tau\Delta}\left(\left(u^n\right)^2\left[\Re\,\varphi_1(-2i\tau\Delta) {-}\Re\,\varphi_2(-2i\tau\Delta)\right]\overline{u^n}\right) \\
			&\quad-i\mu\tau\left(e^{i\tau\Delta}u^n\right)^2\left[e^{i\tau\Delta}\varphi_2(-2i\tau\Delta)\overline{u^n}\right]\\
			&\quad-i\mu\tau e^{{2}i\tau\Delta}\left(\left(e^{-i\tau\Delta}u^n\right)^2\left[e^{-i\tau\Delta}\varphi_2(2i\tau\Delta)\overline{u^n}\right]\right), \quad n \geq 1, \\
			u^1 &= \Phi_\tau^\text{LRI2}(u_0), \quad u^0 = u_0. 
		\end{aligned}
	\end{equation}}
	
	\paragraph{Integrators with ultra-low regularity requirements} An example of an integrator with ultra-low regularity requirements for the cubic NLSE in one dimension was given by Cao et al. \cite{cao2024new}, who proved (cf. Lemma 4.3 in \cite{cao2024new}) that their method satisfies Assumption~\ref{assumption:general_properties_higher_order_schemes} (ii) for {$\alpha=0$, $\gamma_1=3/2$, and $p=1$}. {They also provide an indication of how a stability estimate of the form Assumption~\ref{assumption:general_properties_higher_order_schemes} (i) can be obtained with $\alpha>1/2$ which can then be extended to the estimate \eqref{eqn:weaker_stability_estimate} in $L^2$ analogously to the arguments presented in the proof of Theorem~\ref{cor:globalerrorinL2}.} In the interest of brevity, the reader is referred to (2.4) in \cite{cao2024new} for the full {stability argument and} definition of the integrator.
	
	\paragraph{Higher order resonance-based schemes obtained from decorated trees} In recent work, Bruned \& Schratz \cite{bruned_schratz_2022} introduced a decorated tree formalism for the construction of low-regularity integrators of arbitrary desired order for a general class of dispersive equations. Their analysis includes a structured study of the local error which allows the establishment of local error estimates precisely of the form of Assumption~\ref{assumption:general_properties_higher_order_schemes} (ii), where the value of $\gamma_1$ can be computed using a recursive expression in terms of operators of decorated trees. We note that the stability of these types of schemes currently has to be proven on a case-by-case basis, although the interpolatory framework described in \cite[Section~3.1]{alamabronsard_et_al23} may provide a road to a general stability analysis.
	
	\section{Convergence analysis of explicit symmetric two-step methods}\label{sec:error_analysis}
	Let us now consider the error analysis of the scheme given in \eqref{eqn:general_scheme}. {For this, we take the standard approach in the analysis of multistep methods (cf. \cite{bao2024explicit} for a similar analysis in the case of exponential integrators) and write the scheme \eqref{eqn:general_scheme} in matrix form}
	\begin{align*}
		\mathbf{u}^{n+1}=\mathbf{\Phi}_{\tau}(\mathbf{u}^{n}),
	\end{align*}
	where $\mathbf{u}^{n}:=(u^{n},u^{n-1})^{\rm T}, n\geq 1$, and we introduced the new numerical flow $\mathbf{\Phi}_{\tau}$ as
	\begin{align}\label{eqn:vectoral_form_time_stepper}
		\mathbf{\Phi}_{\tau}(\mathbf{u}^{n})= \mathbf{\Phi}_{\tau}\!\begin{pmatrix}u^{n}\\u^{n-1}\end{pmatrix}:=\begin{pmatrix}\widetilde{\Phi}_{\tau}(u^n)-e^{2i\tau\Delta}\widetilde{\Phi}_{-\tau}(u^n)+e^{2i\tau\Delta}u^{n-1}\\
			u^{n}
		\end{pmatrix}.
	\end{align}
	In the following, we will make use of the product norm on $H^{\alpha}\times H^{\alpha}$ (which we shall simply denote by $\|\,\cdot\,\|_{H^\alpha}$ as it is clear form context which norm we refer to), i.e. $\| (v, w)^{\rm T}\|_{H^\alpha}:=\sqrt{\|v\|_{H^\alpha}^2+\|w\|_{H^\alpha}^2}.$
	
	\subsection{General framework of convergence analysis}
	\label{sec:general_convergence}
	We begin with proving stability of our two-step integrator. 
	\begin{proposition}[Stability]\label{prop:stability_multistep} Suppose $\Phi_{\tau}$ satisfies \cref{assumption:general_properties_higher_order_schemes} (i) for some $\alpha\geq 0$. Then there is a continuous function $C$ such that, for any $v_0,v_1,w_0,w_1\in H^{\alpha}$ the method \eqref{eqn:vectoral_form_time_stepper} satisfies, with $\mathbf v = (v_1, v_0)^{\rm T}$ and $\mathbf w = (w_1, w_0)^{\rm T}$, 
		\begin{align}\label{eqn:stability_estimate_vectoral_form}
			\left\|\mathbf{\Phi}_{\tau}(\mathbf v)-\mathbf{\Phi}_{\tau}(\mathbf w)\right\|_{H^\alpha}\leq (1+\tau C(\|v_1\|_{H^\alpha},\|w_1\|_{H^\alpha}))\left\|\mathbf{v}-\mathbf{w}\right\|_{H^\alpha}.
		\end{align}
	\end{proposition}
	\begin{proof}
		For the analysis it is helpful to write\vspace{-0.1cm}
		\begin{align*}
			\mathbf{\Phi}_{\tau}(\mathbf{v})=\begin{pmatrix}
				0&e^{2i\tau\Delta}\\1&0
			\end{pmatrix}\begin{pmatrix}
				v_1\\v_0
			\end{pmatrix}+\begin{pmatrix}
				\widetilde{\Phi}_{\tau}(v_1)-e^{2i\tau\Delta}\widetilde{\Phi}_{-\tau}(v_1)\\0
			\end{pmatrix}.
		\end{align*}
		Then we have
		\begin{align*}
			\left\|\mathbf{\Phi}_{\tau}(\mathbf{v})-\mathbf{\Phi}_{\tau}(\mathbf{w}))\right\|_{H^\alpha}&\leq\left\|\begin{pmatrix}
				0&e^{2i\tau\Delta}\\1&0
			\end{pmatrix}\begin{pmatrix}
				v_1-w_1\\v_0-w_0
			\end{pmatrix}\right\|_{H^\alpha}\\
			&\quad+\left\|\widetilde{\Phi}_{\tau}(v_1)-\widetilde{\Phi}_{\tau}(w_1)-(e^{2i\tau\Delta}\widetilde{\Phi}_{-\tau}(v_1)-e^{2i\tau\Delta}\widetilde{\Phi}_{-\tau}(w_1))\right\|_{H^\alpha}\\
			&\leq \left\|\mathbf{v} - \mathbf{w}\right\|_{H^\alpha}+2\tau M(\|v_1\|_{H^\alpha},\|w_1\|_{H^\alpha})\|v_1-w_1\|_{H^\alpha},
		\end{align*}
		where in the final line we used the fact that $e^{2i\tau\Delta}$ is an isometry on $H^{\alpha}$ and \eqref{eqn:stability_general_tree_based-scheme}. This clearly implies the desired bound.
	\end{proof}
	
	Then we estimate the local error of the multi-step method \cref{eqn:general_scheme}. We remain in the formulation \eqref{eqn:vectoral_form_time_stepper} and denote by $\mathbf{u}(t_n):=(u(t_n), u(t_{n-1}))^{\rm T}$ for $n \geq 1$. Then we have the following local error estimate.
	\begin{proposition}[Local error]\label{prop:local_error_multistep} 
	{Supposing $\Phi_{\tau}$ satisfies Assumption \ref{assumption:general_properties_higher_order_schemes} (ii), we have}
		\begin{equation}\label{eqn:local_error_multistep}
			\left\|\mathbf{u}(t_{n+1}) - {\mathbf{\Phi}_\tau}(\mathbf{u}(t_n))\right\|_{H^\alpha}\leq \tau^{p+1}C(\sup_{ |t|\leq\tau}\|u(t_n+t)\|_{H^{\alpha+\gamma_1}}), \  1 \leq n \leq \frac{T}{\tau}-1. 
		\end{equation}
		Furthermore, if $\Phi_{\tau}$ satisfies, in addition, Assumption \ref{assumption:general_properties_higher_order_schemes} (iii), then 
		\begin{equation}\label{eqn:improved_local_error_multistep}
			\left\|\mathbf{u}(t_{n+1}) \!-\! {\mathbf{\Phi}_\tau}(\mathbf{u}(t_n))\right\|_{H^\alpha} \!\leq\! \tau^{p+2}\widetilde C(\sup_{ |t|\leq\tau}\|u(t_n\!+\!t)\|_{H^{\alpha+\gamma_1+\gamma_2}}), \  1 \leq n \leq \frac{T}{\tau}-1. 
		\end{equation}
	\end{proposition}
	\begin{proof} 
		Recalling \cref{eqn:vectoral_form_time_stepper}, using \eqref {eqn:general_scheme_with_remainder_term}, we have
		\begin{align*}
			&\mathbf{u}(t_{n+1}) - {\mathbf{\Phi}_\tau}(\mathbf{u}(t_n))\\
			&=\left(u(t_{n+1})-\left(\widetilde{\Phi}_{\tau}(u(t_n))-e^{2i\tau\Delta}\widetilde{\Phi}_{-\tau}(u(t_n))+e^{2i\tau\Delta} u(t_{n-1})\right), 0 \right)^{\rm T}\\
			&=\left(\mathcal{R}^n_{\tau}-e^{2i\tau\Delta}\mathcal{R}^n_{-\tau}, 0\right)^{\rm T}, \quad 1 \leq n \leq T/\tau-1. 
		\end{align*}
		The results immediately follow from \cref{eqn:local_error_general_tree_based-scheme,eqn:improved_local_error_general_tree_based-scheme}. 
	\end{proof}
	
	Combining the above stability and local error estimates, we can obtain the following global convergence result of our explicit symmetric {multi-step} methods. Define
	\begin{equation}\label{eq:Mdef}
		M_\delta = \sup_{0 \leq t \leq T} \| u(t) \|_{H^\delta}, \qquad \delta \geq 0. 
	\end{equation}
	
	\begin{theorem}[Global error]\label{thm:global_error_multistep_schemes}
		Suppose $u^{n}$ is computed using \eqref{eqn:general_scheme}, where $\Phi_{\tau}$ satisfies Assumptions~\ref{assumption:general_properties_higher_order_schemes}. (i) and (ii). For all $0<\tau\leq \tau_0$ with $\tau_0>0$ sufficiently small depending on $M_{\alpha+\gamma_1}$ and $T$, we have
		\begin{align}\label{eqn:global_error_bound1}
			\|u(t_n) - u^n\|_{H^\alpha}\leq C\left(M_{\alpha+\gamma_1}\right)\tau^{p}, \quad 0 \leq n \leq T/\tau. 
		\end{align}
		If, in addition, Assumption~\ref{assumption:general_properties_higher_order_schemes} (iii) is satisfied, then
		\begin{align}\label{eqn:global_error_bound2}
			\|u(t_n) - u^n\|_{H^\alpha}\leq C\left(M_{\alpha+\gamma_1+\gamma_2}\right)\tau^{p+1}, \quad 0 \leq n \leq T/\tau. 
		\end{align}
	\end{theorem}
	
	\begin{proof}
		Recall $\mathbf{u}(t_n) = (u(t_n), u(t_{n-1}))^{\rm T} $ and $ \mathbf{u}^n = (u^n, u^{n-1})^{\rm T}$, and define the error functions $\mathbf{e}^n = (e^n, e^{n-1})^{\rm T} := \mathbf{u}(t_n) - \mathbf{u}^n$ for $1 \leq n \leq T/\tau$. Then we have
		\begin{equation*}
			\|u(t_n) - u^n\|_{H^\alpha}
			\leq \| \mathbf u(t_n) - \mathbf u^n\|_{H^\alpha} = \| \mathbf e^n \|_{H^\alpha}. 
		\end{equation*}
		Thus, it is sufficient to show the bounds \eqref{eqn:global_error_bound1} and \eqref{eqn:global_error_bound2} for $\| \mathbf e^n \|_{H^\alpha}$. We proceed with a Lady Windermere's fan argument to prove \eqref{eqn:global_error_bound1}. 
		First, by \cref{eqn:local_error_general_tree_based-scheme}, 
		\begin{equation}\label{eq:firststeperror}
			\| \mathbf{e}^1 \|_{H^\alpha} = \| e^1 \|_{H^\alpha} = \| u(t_1) - \Phi_\tau(u_0) \|_{H^\alpha} = \| \mathcal{R}^0_\tau \|_{H^\alpha} \leq C(M_{\alpha+\gamma_1})\tau^{p+1}. 
		\end{equation}
		By the triangle inequality, we have, for $1 \leq n \leq T/\tau-1$, 
		\begin{align}\label{eq:error_eq_1}
			\| \mathbf e^{n+1} \|_{H^\alpha} 
			&\leq \| \mathbf{u}(t_{n+1}) - \bm{\Phi}_\tau(\mathbf{u}(t_n)) \|_{H^\alpha} + \| \bm{\Phi}_\tau(\mathbf{u}(t_n)) -\mathbf{\Phi}_\tau(\mathbf{u}^{n}) \|_{H^\alpha}. 
		\end{align}
		From \cref{eq:error_eq_1}, by \cref{prop:stability_multistep,eqn:local_error_multistep}, we get, for $1 \leq n \leq T/\tau-1$, 
		\begin{equation}\label{eq:error_eq_est}
			\| \mathbf e^{n+1} \|_{H^\alpha} \leq (1+C(M_\alpha, \| u^n \|_{H^\alpha}) \tau) \| \mathbf{e}^n \|_{H^\alpha} + C(M_{\alpha+\gamma_1}) \tau^{p+1}.  
		\end{equation}
		To conclude the proof by applying the discrete Gronwall's inequality to \cref{eq:error_eq_est} with \cref{eq:firststeperror}, we need to control $\| u^n \|_{H^\alpha} \  (0 \leq n \leq T/\tau-1)$ independently of $\tau$. This can be done by the standard induction argument with an requirement of $0 <\tau \leq \tau_0$ for some $\tau_0>0$ sufficiently small depending on $M_{\alpha+\gamma_1}$. Then we prove \cref{eqn:global_error_bound1} and  the proof of \eqref{eqn:global_error_bound2} follows analogously using \cref{eqn:improved_local_error_multistep} instead of \cref{eqn:local_error_multistep}. 
	\end{proof}
 {\begin{remark}
     Clearly, if we had used a slightly weaker stability estimate of the form \eqref{eqn:weaker_stability_estimate} in \eqref{eq:error_eq_est} then the global convergence result would still similarly follow. This type of argument is typically used for convergence analysis in $L^2$ and is presented for example in the proof of Theorem~\ref{cor:globalerrorinL2}.
 \end{remark}}

	In the following part, we will illustrate the workings of this general framework on two examples, sLRI1 \eqref{eq:sLRI1} and sLRI2 \eqref{eq:sLRI2}.
	
	\subsection{{Convergence analysis for sLRI1}}\label{sec:application_to_sLRI1}
	In this subsection, we establish the global error estimates for the sLRI1 \cref{eq:sLRI1}. We start with the error estimate in $H^\alpha$-norm with $\alpha > d/2$ by following the general framework in the previous subsection. Then we further push down the error estimates to $L^2$-norm and $H^1$-norm (when $d=2, 3$), which are natural norms associated with the NLSE in terms of mass and energy. Particular attention is paid to the analysis of the improved convergence order (i.e. the order greater than one) due to the symmetry of the scheme. 
	
	By {Lemmas} 3.1 \& 3.2 in \cite{ostermann_schratz2018low}, the numerical flow $\Phi^\text{LRI1}_{\tau}$ of LRI1 \cref{eqn:ostermann_schratz_lri} satisfies Assumptions~\ref{assumption:general_properties_higher_order_schemes} (i) and (ii) for any $\alpha>d/2$ and $p=1$ with $\gamma_1 = p$. 
	{Note, in principle the analysis in \cite{ostermann_schratz2018low} studies the local truncation error $\mathcal{R}^n_{\tau}$ and the stability of the scheme only for non-negative values of $\tau$, while Assumptions~\ref{assumption:general_properties_higher_order_schemes} (i) \& (ii) require a similar bound when $\tau\leq 0$. This can be established analogously to the original estimates in \cite{ostermann_schratz2018low} and we therefore omit the details here.} A direct application of \cref{thm:global_error_multistep_schemes} implies the following first-order convergence of the sLRI1 \cref{eq:sLRI1} in $H^\alpha$-norm. 
	\begin{corollary}
		Let $\alpha>d/2$. Suppose $u^n$ is computed using \cref{eq:sLRI1}. For all $0<\tau\leq \tau_0$ with $\tau_0>0$ sufficiently small depending on $M_{\alpha+1}$ and $T$,
		\begin{equation*}
			\|u(t_n) - u^n\|_{H^\alpha} \leq C(M_{\alpha+1})\tau, \quad 0 \leq n \leq T/\tau.
		\end{equation*}
	\end{corollary}
	Due to the symmetry of the scheme, the sLRI1 is of second order and satisfies the improved local error estimate in \cref{assumption:general_properties_higher_order_schemes} (iii) albeit with a larger regularity requirement. In fact, we have the following fractional-order improvement of the local truncation error of the sLRI1 \cref{eq:sLRI1} beyond \cref{eqn:improved_local_error_general_tree_based-scheme}.  
	\begin{proposition}\label{thm:sLRI_decomp}
		Let $\alpha>d/2$ and $0 \leq \gamma \leq 1$. For the sLRI1 \cref{eq:sLRI1}, we have
		\begin{equation*}
			\| \mathcal{R}^n_{\tau}-e^{2i\tau\Delta}\mathcal{R}^n_{-\tau} \|_{H^\alpha} \leq C(M_{\alpha+1+2\gamma}) \tau^{2+\gamma}, \quad 1 \leq n \leq T/\tau-1. 
		\end{equation*}
	\end{proposition}
	As an immediate consequence of \cref{thm:sLRI_decomp,thm:global_error_multistep_schemes}, we have the following improved global convergence of the sLRI1. 
	\begin{corollary}\label{thm:improvedhigh}
		Let $\alpha>d/2$ and $0 \leq \gamma \leq 1$. Suppose $u^{n}$ is computed using \eqref{eq:sLRI1}. For all $0<\tau\leq \tau_0$ with $\tau_0>0$ depending on $M_{\alpha+1+2\gamma}$ and $T$, 
		\begin{align*}
			\|u(t_n) - u^n\|_{H^\alpha}\leq C(M_{\alpha+1+2\gamma})\tau^{1+\gamma}, \quad 0 \leq n \leq T/\tau.
		\end{align*}
	\end{corollary}
	The results in \cref{thm:sLRI_decomp,thm:improvedhigh} indicate that the {additional $\gamma$-th order convergence from symmetry requires $2\gamma$ \revisions{additional bounded derivatives of the exact solution.}} 
	\begin{proof}[Proof of \cref{thm:sLRI_decomp}]
		Recalling \cref{eq:duhamel_exact,eqn:ostermann_schratz_lri}, we have
		\begin{align}
			\mathcal{R}^n_\tau 
			&= u(t_{n+1}) - \Phi^\text{LRI1}_\tau(u(t_n)) \notag \\
			&= -i\mu e^{i \tau \Delta} \int_0^\tau \left[e^{-is\Delta} (u(t_n+s)^2 \overline{u(t_n+s)}) - u(t_n)^2 e^{-2is\Delta} \overline{u(t_n)} \right] \rmd s. 
		\end{align}
		Then the local truncation error of the sLRI1 scheme can be represented as
		\begin{align}\label{eq:LTEsLRI1}
			\mathcal{L}^n 
			&:= \mathcal{R}^n_{\tau}-e^{2i\tau\Delta}\mathcal{R}^n_{-\tau} \notag \\
			&= -i\mu e^{i \tau \Delta} \int_{-\tau}^\tau \left[e^{-is\Delta} (u(t_n+s)^2 \overline{u(t_n+s)}) - u(t_n)^2 e^{-2is\Delta} \overline{u(t_n)} \right] \rmd s. 
		\end{align}
		From \cref{eq:LTEsLRI1}, $\mathcal{L}^n$ can be decomposed as
		$\mathcal{L}^n = - i \mu e^{i \tau \Delta }(r_1^n + r_2^n)$, 
		where\vspace{-0.2cm}
		\begin{align}
			&r_1^n = \int_{-\tau}^\tau \left[e^{- i s \Delta} \left((e^{is\Delta}u(t_n))^2 e^{-is\Delta}\overline{u(t_n)} \right) - u(t_n)^2 e^{-2is\Delta} \overline{u(t_n)} \right] \rmd s, \label{eq:r1}\\
			&r_2^n = \int_{-\tau}^\tau e^{- i s \Delta} \left[ u(t_n+s)^2 \overline{u(t_n+s)} -  (e^{is\Delta}u(t_n))^2 e^{-is\Delta}\overline{u(t_n)} \right] \rmd s. \label{eq:r2}
		\end{align}
		
		We start with the estimate of $r^n_1$. Define the filtered function
		\begin{equation}\label{eq:N_def}
			N(\zeta, s) = e^{- i \zeta \Delta} \left((e^{i\zeta\Delta}u(t_n))^2 e^{i(\zeta - 2s)\Delta}\overline{u(t_n)} \right) ,\quad 0 \leq |\zeta| \leq |s| \leq \tau.  
		\end{equation}
		For $r^n_1$ defined in \cref{eq:r1}, we have
		\begin{align*}
			r_1^n 
			&= \int_{-\tau}^\tau \left[ N(s, s) - N(0, s) \right] \rmd s = \int_0^\tau \int_0^s \partial_\zeta N(\zeta, s) \rmd \zeta \rmd s - \int_{-\tau}^0 \int_s^0 \partial_{\zeta} N(\zeta, s) \rmd \zeta \rmd s \\
			&= \int_0^\tau \int_0^s \partial_\zeta N(\zeta, s) \rmd \zeta \rmd s - \int_0^\tau \int_0^s \partial_{\zeta} N(-\zeta, -s) \rmd \zeta \rmd s \\ 
			&= \int_0^\tau \int_0^s \left[\partial_\zeta N(\zeta, s) - \partial_{\zeta} N(-\zeta, -s)\right] \rmd \zeta \rmd s. 
		\end{align*}
		It follows immediately that
		\begin{equation}
			\| r_1^n \|_{H^\alpha} \leq \frac{\tau^2}{2} \sup_{0 \leq \zeta \leq s \leq \tau} \| \partial_\zeta N(\zeta, s) - \partial_{\zeta} N(-\zeta, -s) \|_{H^\alpha}, \quad \alpha \geq 0. 
		\end{equation}
		Due to the symmetry of the scheme, we obtain a cancellation term
		\begin{equation}\label{eq:E1_def}
			\mathcal{E}_1:= \partial_\zeta N(\zeta, s) - \partial_{\zeta} N(-\zeta, -s), 
		\end{equation}
		which can yield an increment of the order if additional regularity of the solution $u$ is satisfied. 
		In the following, we estimate $r_2^n$ in \cref{eq:r2}. We have, by letting $F(u) = |u|^2u$, \vspace{-0.4cm}
		\begin{equation}\label{eq:Gamma_def}
			\Gamma(\zeta, s) = F(e^{i(s-\zeta)\Delta}u(t_n + \zeta)), \quad 0 \leq |\zeta| \leq |s| \leq \tau.   
		\end{equation}
		Then we have, similar to $r^n_1$ above, 
		\begin{align*}
			r_2^n 
			&= \!\int_{-\tau}^\tau\!\!\!e^{- i s \Delta} (\Gamma(s, s) - \Gamma(0, s))\rmd s =\!\!\int_{0}^\tau\!\!\!\!e^{- i s \Delta}\!\! \int_0^s\!\!\!\! \partial_\zeta \Gamma(\zeta, s) \rmd \zeta \rmd s-\!\!\int_{-\tau}^0\!\!\!\!\! e^{- i s \Delta}\!\! \int_s^0\!\!\! \partial_\zeta \Gamma(\zeta, s) \rmd \zeta \rmd s \\
			&= \int_{0}^\tau \int_0^s \left[e^{- i s \Delta} \partial_\zeta \Gamma(\zeta, s)- e^{i s \Delta} \partial_\zeta \Gamma(-\zeta, -s) \right] \rmd \zeta \rmd s. 
		\end{align*}
		It follows that 
		\begin{equation}
			\| r_2^n \|_{H^\alpha} \leq \frac{\tau^2}{2} \sup_{0 \leq \zeta \leq s \leq \tau} \| e^{- i s \Delta} \partial_\zeta \Gamma(\zeta, s)- e^{i s \Delta} \partial_\zeta \Gamma(-\zeta, -s) \|_{H^\alpha}, \quad \alpha \geq 0. 
		\end{equation}
		Here, we obtain another cancellation term
		\begin{equation}\label{eq:E2_def}
			\mathcal{E}_2: = e^{- i s \Delta} \partial_\zeta \Gamma(\zeta, s)- e^{i s \Delta} \partial_\zeta \Gamma(-\zeta, -s),
		\end{equation}
		which can also increase the order under additional regularity of $u$. The conclusion follows from the following \cref{lem:E1andE2} immediately. 
	\end{proof}
	
	In the following, we shall frequently use the bilinear estimate
	\begin{equation}\label{eq:bi}
		\| v w \|_{H^\alpha} \lesssim \| v \|_{H^\alpha} \| w \|_{H^\alpha}, \qquad \alpha > d/2, 
	\end{equation}
	the isometry property of $e^{it\Delta}$ on $H^{\alpha}, \alpha\geq0,$ and the standard fractional estimate {(cf. Lemma 4.1 in \cite{cao2024new})}
	\begin{equation}\label{eq:exp_est}
		\| (e^{it\Delta} - I)v \|_{H^{\alpha_1}} \lesssim t^\frac{\alpha_2}{2} \| v \|_{H^{\alpha_1+\alpha_2}}, \qquad \alpha_1 \geq 0, \quad 0 \leq \alpha_2 \leq 2, \quad{t\geq 0}. 
	\end{equation}
	
	\begin{lemma}\label{lem:E1andE2}
		We have, for $\alpha>d/2$ and $0 \leq \gamma \leq 1$, 
		\begin{equation*}
			\| \mathcal{E}_1 \|_{H^\alpha} \leq C(M_{\alpha + 1+ 2\gamma}) \tau^{\gamma}, \quad \| \mathcal{E}_2 \|_{H^\alpha} \leq C(M_{\alpha + 2\gamma}) \tau^{\gamma}. 
		\end{equation*}
	\end{lemma}
	
	\begin{proof}
		We start with $\mathcal{E}_1$. Recalling \cref{eq:N_def}, {and according to the direct calculation \cite{alama2023symmetric}, we have} 
		\begin{equation}\label{eq:parN}
			\partial_\zeta N(\zeta, s) = 2e^{- i \zeta \Delta} \left((\nabla w(\zeta))^2 \overline{w(2s - \zeta)} + 2 w(\zeta) \nabla w(\zeta) \cdot \nabla \overline{w(2s-\zeta)}\right), 
		\end{equation}
		where $w(\sigma) = e^{i\sigma\Delta} u(t_n)$. 
		From \cref{eq:E1_def}, by triangle inequality, the isometry property of $e^{it\Delta}$, we have
		\begin{align}\label{eq:E1_decomp}
			\| \mathcal{E}_1 \|_{H^\alpha} 
			&\leq 2\| (1-e^{2i\zeta \Delta}) ((\nabla w(\zeta))^2 \overline{w(2s - \zeta)} + 2 w(\zeta) \nabla w(\zeta) \cdot \nabla \overline{w(2s-\zeta)}) \|_{H^\alpha} \notag \\
			&\quad + 2\left \| (\nabla w(\zeta))^2 \overline{w(2s - \zeta)} - (\nabla w(-\zeta))^2 \overline{w(\zeta-2s)} \right \|_{H^\alpha} \notag \\
			&\quad + 4 \left \| w(\zeta) \nabla w(\zeta) \cdot \nabla \overline{w(2s-\zeta)} - w(-\zeta) \nabla w(-\zeta) \cdot \nabla \overline{w(\zeta - 2s)} \right \|_{H^\alpha} \notag \\
			&=:2\| (1-e^{2i\zeta\Delta})e_1 \|_{H^\alpha} + 2 \| e_2 \|_{H^\alpha} + 4 \| e_3 \|_{H^\alpha},
		\end{align}
        {with $e_1$, $e_2$ and $e_3$ given by
        \begin{align*}
             e_1 & = (\nabla w(\zeta))^2 \overline{w(2s - \zeta)} + 2 w(\zeta) \nabla w(\zeta) \cdot \nabla \overline{w(2s-\zeta)}, \\
             e_2 & = (\nabla w(\zeta))^2 \overline{w(2s - \zeta)} - (\nabla w(-\zeta))^2 \overline{w(\zeta-2s)}, \\ 
             e_3 & =  w(\zeta) \nabla w(\zeta) \cdot \nabla \overline{w(2s-\zeta)} - w(-\zeta) \nabla w(-\zeta) \cdot \nabla \overline{w(\zeta - 2s)}.
        \end{align*}
        By \cref{eq:exp_est} with $\alpha_1 =\alpha$, $\alpha_2 = 2\gamma$, we have 
        \begin{equation}
        \| \mathcal{E}_1 \|_{H^\alpha} \lesssim \tau^{\gamma} \| e_1 \|_{H^{\alpha+2\gamma}} + \| e_2 \|_{H^\alpha} + \| e_3 \|_{H^\alpha}.
        \end{equation}
        }       
		By the bilinear estimate \cref{eq:bi} and the isometry property of $e^{it\Delta}$, we get
		\begin{align}\label{eq:e1}
			\| e_1 \|_{H^{\alpha+2\gamma}} \lesssim \| u(t_n) \|_{H^{\alpha+1+2\gamma}}^2\| u(t_n) \|_{H^{\alpha + 2\gamma}} \lesssim M_{\alpha + 1+ 2\gamma}^3. 
		\end{align}
		To estimate $e_2$ and $e_3$, we need to estimate terms of the following two forms:
		\begin{equation}\label{eq:T1T2}
			T_1 = (w(t) - w(-t)) \nabla w_1 \cdot \nabla w_2, \quad T_2 = w_1 \nabla (w(t) - w(-t)) \cdot \nabla w_2, \quad t \geq 0,
		\end{equation}
		where $w(t) = e^{it\Delta} u(t_n)$ as defined before, and $w_1, w_2$ satisfy $\| w_j \|_{H^\delta} \leq M_\delta \ (j=1,2)$ for any $\delta \geq 0$. By \cref{eq:exp_est,eq:bi} again, we have
		\begin{align}
			&\| T_1 \|_{H^\alpha} \lesssim \| w(t) - w(-t) \|_{H^\alpha} \| \nabla w_1 \|_{H^\alpha} \| \nabla w_2 \|_{H^\alpha} \lesssim \tau^\gamma M_{\alpha+2\gamma} M_{1+\alpha}^2, \label{eq:T1}\\
			&\| T_2 \|_{H^\alpha} \lesssim \| w_1 \|_{H^\alpha} \| \nabla (w(t) - w(-t)) \|_{H^\alpha} \|  \nabla w_2 \|_{H^\alpha} \lesssim \tau^\gamma M_{\alpha+1 +2\gamma} M_\alpha M_{1+\alpha}. \label{eq:T2}
		\end{align}
		For $e_2$ and $e_3$ in \cref{eq:E1_decomp}, by \cref{eq:T1,eq:T2}, we have $\| e_2 \|_{H^\alpha} + \| e_3 \|_{H^\alpha} \lesssim \tau^\gamma M_{\alpha + 1 +2\gamma}^3$, which, together with \cref{eq:e1}, yields from \cref{eq:E1_decomp} that
		\begin{equation}\label{eq:E1_est}
			\| \mathcal{E}_1 \|_{H^\alpha} \leq C(M_{\alpha+1+2\gamma}) \tau^{\gamma}. 
		\end{equation} 
		Then we deal with $\mathcal{E}_2$. To simplify the notations, we define, for $-\tau \leq t \leq \tau$, 
		\begin{equation}\label{eq:notation_simp}
			u_n(t) = u(t_n+t), \quad  v_n(t) = e^{-it\Delta}u_n(t), \quad  w_n(t) = e^{-it\Delta} F(e^{it\Delta}v_n(t)).
		\end{equation}
		Taking \cref{eq:notation_simp} into \cref{eq:duhamel_modified}, we have\vspace{-0.3cm}
		\begin{equation}\label{eq:vn_int}
			v_n(t) = u(t_n) - i \int_0^t w_n(s) \rmd s, \quad -\tau \leq t \leq \tau, 
		\end{equation}\vspace{-0.6cm}\ \\
		which implies, by \cref{eq:bi}, 
		\begin{equation}\label{eq:vn_est}
			\| v_n(t) - v_n(0) \|_{H^\alpha} \leq \tau \sup_{ |s| \leq \tau} \| w_n(s) \|_{H^\alpha} \lesssim \tau M_\alpha^3, \quad - \tau \leq t \leq \tau.  	
		\end{equation}
		From \cref{eq:Gamma_def}, recalling \cref{eq:notation_simp}, we get
		\begin{equation}\label{eq:parGamma}
			\partial_\zeta \Gamma(\zeta, s) = \partial_\zeta F(e^{is\Delta} v_n(\zeta)) = dF(e^{is\Delta} v_n(\zeta))[-ie^{is\Delta} w_n(\zeta) ],  
		\end{equation}
		where $dF(\cdot)[\cdot]$ is the Gateaux derivative of $F$ given by
		\begin{equation}\label{eq:dF_def}
			dF(v)[w] := \lim_{\delta \rightarrow 0} \frac{F(v+\delta w) - F(v)}{\delta} = 2|v|^2w + v^2 \overline{w}. 
		\end{equation}
		Compared to \cref{eq:parN}, there is no gradient in \cref{eq:parGamma} and thus the additional regularity needed by $\mathcal{E}_2$ is lower than for $\mathcal{E}_1$. Recalling \cref{eq:dF_def}, using the bilinear estimate \cref{eq:bi}, we have, for $v_1, v_2, w_1, w_2 \in H^\alpha(\mathbb{T}^d)$, 
		\begin{align}\label{eq:diff_dF}
			\| dF(v_1)[w_1] - dF(v_2)[w_2] \|_{H^\alpha} 
			&\lesssim \| v_1 - v_2 \|_{H^\alpha}\| v_1 \|_{H^\alpha} \| w_1 \|_{H^\alpha} \notag \\
			&\ + \| v_2 \|_{H^\alpha} \| v_1 - v_2 \|_{H^\alpha} \| w_1 \|_{H^\alpha} + \| v_2 \|_{H^\alpha}^2 \| w_1 - w_2 \|_{H^\alpha}. 
		\end{align} 
		From {the definition of $\mathcal{E}_2$ \cref{eq:E2_def}, the bilinear estimate \cref{eq:bi}, the estimate of $(e^{it\Delta}-I)$ \cref{eq:exp_est} and the definition of the Gateaux derivative of $F$ \cref{eq:diff_dF}, we have,} recalling \cref{eq:parGamma,eq:notation_simp},
		\begin{align}\label{eq:E2_decomp}
			\| \mathcal{E}_2 \|_{H^\alpha} 
			&\leq \| (e^{-is\Delta} - e^{is\Delta}) \partial_\zeta \Gamma(\zeta, s) \|_{H^\alpha} +  \| \partial_\zeta \Gamma(\zeta, s) - \partial_\zeta \Gamma(-\zeta, -s) \|_{H^\alpha} \notag \\
			&\lesssim \tau^{\gamma} \| \partial_\zeta \Gamma(\zeta, s) \|_{H^{{\alpha+2\gamma}}} \notag \\
			&\quad + \| dF(e^{is\Delta} v_n(\zeta))[-ie^{is\Delta} w_n(\zeta) ] - dF(e^{-is\Delta} v_n(-\zeta))[-ie^{-is\Delta} w_n(-\zeta) ] \|_{H^\alpha} \notag \\
			&\lesssim \tau^{\gamma} M_{\alpha+2\gamma}^5  + M_\alpha^4 \| e^{is\Delta} v_n(\zeta) - e^{-is\Delta}v_n(-\zeta) \|_{H^\alpha} \notag \\
			&\quad + M_\alpha^2 \| e^{is\Delta} w_n(\zeta) - e^{-is\Delta}w_n(-\zeta) \|_{H^\alpha} \notag \\
			&=: \tau^{\gamma} M_{\alpha+2\gamma}^5 + M_\alpha^4 \| e_4 \|_{H^\alpha} + M_\alpha^2 \| e_5 \|_{H^\alpha},
		\end{align}
        {with $e_4$ and $e_5$ given by
        \begin{align*}
            e_4 & = e^{is\Delta} v_n(\zeta) - e^{-is\Delta}v_n(-\zeta),\\
            e_5 & = e^{is\Delta} w_n(\zeta) - e^{-is\Delta}w_n(-\zeta).
        \end{align*}
        }
		For $e_4$, by \cref{eq:vn_est,eq:exp_est}, we have
		\begin{align}\label{eq:e4_est}
			\| e_4 \|_{H^\alpha} 
			&\lesssim  \| (e^{is\Delta} - e^{is\Delta}) v_n(0) \|_{H^\alpha} + \| v_n(\zeta) - v_n(0) \|_{H^\alpha} + \| v_n(-\zeta) - v_n(0) \|_{H^\alpha} \notag \\
			&\lesssim \tau^{\gamma} M_{\alpha + 2\gamma} + \tau M_\alpha^3 \leq \tau^{\gamma} C(M_{\alpha + 2\gamma}). 
		\end{align}
		For $e_5$, by \cref{eq:vn_est,eq:exp_est}, and \cref{eq:e4_est} for $e_4$ with $s = \zeta$, we have
		\begin{align}\label{eq:e5_est}
			\| e_5 \|_{H^\alpha} 
			&\leq \| (e^{is\Delta} - e^{-is\Delta}) w_n(\zeta) \|_{H^\alpha} + \| w_n(\zeta) - w_n(-\zeta) \|_{H^\alpha} \notag \\
			&\lesssim \tau^{\gamma} \| w_n(\zeta) \|_{H^{\alpha+2\gamma}} + \| e^{-i\zeta \Delta} F(e^{i\zeta\Delta}v_n(\zeta)) - e^{i\zeta \Delta} F(e^{-i\zeta\Delta}v_n(-\zeta)) \|_{H^\alpha} \notag \\
			&\lesssim \tau^{\gamma} M_{\alpha+2\gamma}^3 + \tau^{\gamma} \| F(e^{i\zeta\Delta} v_n(\zeta)) \|_{H^{\alpha+2\gamma}} \notag \\
			&\quad + \| F(e^{i\zeta\Delta}v_n(\zeta)) - F(e^{-i\zeta\Delta}v_n(-\zeta)) \|_{H^\alpha} \notag \\
			&\lesssim \tau^{\gamma} M_{\alpha+2\gamma}^3 + M_\alpha^2\| e^{i\zeta\Delta}v_n(\zeta) - e^{-i\zeta\Delta} v_n(-\zeta) \|_{H^\alpha} \notag \\
			&\lesssim \tau^{\gamma} M_{\alpha+2\gamma}^3 + M_\alpha^2 \tau^\gamma C(M_{\alpha + 2\gamma}) \lesssim \tau^{\gamma} C(M_{\alpha + 2\gamma}). 
		\end{align}
		Plugging \cref{eq:e4_est,eq:e5_est} into \cref{eq:E2_decomp}, we obtain\vspace{-.1cm}
		\begin{equation}\label{eq:E2_est}
			\| \mathcal{E}_2 \|_{H^\alpha}  \lesssim C(M_{\alpha + 2\gamma}) \tau^{\gamma}. 
		\end{equation}
		{Combining} \cref{eq:E1_est,eq:E2_est}, we complete the proof. 
	\end{proof}

	
	
	In the following, we perform error estimates in low-order Sobolev spaces, i.e. $L^2$ and $H^1$ spaces, beyond the general framework of convergence analysis introduced in the previous subsection. As mentioned before, $L^2$- and $H^1$-norms are natural norms in terms of the mass and energy, and these also fit well with our low-regularity setting. To do so, let us define, 
	for $ 0\leq \gamma \leq 1$, the index $\sigma = \sigma(\gamma)>d/2$ and $\sigma_1 = \sigma_1(\gamma)>1+d/2$ as 
	\begin{equation}
		\sigma:= \left\{
		\begin{aligned}
			&2\gamma+1, && \gamma>d/4, \\
			&d/2+1+\vep, && \gamma = d/4, \\
			&1+d/4+\gamma, && \gamma<d/4,
		\end{aligned}
		\right. \quad 
		\sigma_1: =  \left\{
		\begin{aligned}
			&1+d/2+\vep, && \gamma \leq d/4-1/2, \\
			&2+2\gamma, && \gamma > d/4- 1/2,  
		\end{aligned}
		\right.
	\end{equation}
	where $\vep>0$ is not fixed and can be arbitrarily small. We shall adopt such convention for $\vep$ throughout the rest of this paper, which, in particular, applies to \cref{eq:bi3} below. 
	
	For the $L^2$-norm error bound, we have the following result. 
	\begin{theorem}\label{cor:globalerrorinL2}
		Let $0 \leq \gamma \leq 1$. Suppose $u^{n}$ is computed using \eqref{eq:sLRI1}. For all $0<\tau\leq \tau_0$ with $\tau_0>0$ depending on $M_{\sigma}$ and $T$, we have
		\begin{equation*}
			\|u(t_n) - u^n\|_{L^2} \leq C(M_\sigma)\tau^{1+\gamma},\quad 0 \leq n \leq T/\tau. 
		\end{equation*}
	\end{theorem}
	
	For the $H^1$-norm error bound in 2D and 3D (the 1D case is covered by \cref{thm:improvedhigh}), we have the following result. 
	\begin{theorem}\label{thm:sLRI1_H1}
		Let $0 \leq \gamma \leq 1$ and $d = 2, 3$. Suppose $u^{n}$ is computed using \eqref{eq:sLRI1}. For all $0<\tau\leq \tau_0$ with $\tau_0>0$ depending on $M_{\sigma_1}$ and $T$, we have
		\begin{equation*}
			\|u(t_n) - u^n\|_{H^1} \leq C(M_{\sigma_1})\tau^{1+\gamma}, \quad 0 \leq n \leq T/\tau. 
		\end{equation*}
	\end{theorem}
	

{
\begin{remark}
    Theorems~\ref{cor:globalerrorinL2} \& \ref{thm:sLRI1_H1} imply that sLRI1 converges at second order in $H^{s}(\mathbb{T}^d)$ so long as the exact NLSE solution is bounded in $H^{s+3}(\mathbb{T}^d)$ for $s=0,1$. This is in line with the regularity requirements obtained for the symmetric implicit low-regularity integrator introduced by \cite{alama2023symmetric} and is still a lower regularity requirement than what classical symmetric methods would require for the same task - for example the Strang splitting (which is symmetric and explicit) converges at second order only for solutions that are bounded in $H^{s+4}(\mathbb{T}^d)$ (cf. \cite{lubich2008splitting}).
\end{remark}
}
	
	The estimate of local truncation error in $L^2$- and $H^1$-norms proceeds similarly to \cite{alama2023symmetric}, where the following bilinear estimates are introduced:
	\begin{align}
		&\| vw \|_{L^2} \lesssim \| v \|_{H^{\frac{d}{4} + \gamma}} \| w \|_{H^{\frac{d}{4} - \gamma}}, \qquad 0 \leq \gamma < d/4, \label{eq:bi1} \\ 
		&\| vw \|_{H^{2\gamma}} \lesssim \| v \|_{H^{\frac{d}{4} + \gamma}} \| w \|_{H^{\frac{d}{4} + \gamma}}, \qquad 0 \leq \gamma < d/4,  \label{eq:bi2}\\ 
		&\| vw \|_{H^{2\gamma}} \lesssim \| v \|_{H^{\frac{d}{2}+\vep}} \| w \|_{H^{2\gamma}}, \qquad 0 \leq \gamma \leq d/4. \label{eq:bi3}
	\end{align}
	
	We start with the $L^2$-norm estimate of the local truncation error $\mathcal{L}^n$ in \cref{eq:LTEsLRI1}. 
	\begin{proposition}\label{thm:LTE_L2}
		Let $0 \leq \gamma \leq 1$. We have
		\begin{equation*}
			\| \mathcal{L}^n \|_{L^2} \leq \tau^{2+\gamma} C(M_\sigma), \quad 1 \leq n \leq T/\tau-1. 
		\end{equation*}
	\end{proposition}
	
	\begin{proof}
		Following the proof of \cref{thm:sLRI_decomp}, it suffices to obtain $L^2$-norm estimates of $\mathcal{E}_1$ in \cref{eq:E1_def} and $\mathcal{E}_2$ in \cref{eq:E2_def}. We start with $\mathcal{E}_1$. By noting that \cref{eq:E1_decomp} is still valid when $\alpha = 0$, it reduces to the estimates of $e_j \ (1 \leq j \leq 3)$ in $H^{2\gamma}$-norm or $L^2$-norm. For $e_1$, we have
		\begin{align}\label{eq:e1_L2}
			\| e_1 \|_{H^{2\gamma}} \lesssim M_{\sigma}^3, 
		\end{align}
		where the estimate for $\gamma \geq d/4$ follows directly from \cref{eq:bi} (and $\| e_1 \|_{H^{2\gamma}} \lesssim \|e_1 \|_{H^{d/2+\vep}}$ when $\gamma = d/4$), and the estimate for $\gamma < d/4$ follows from \cref{eq:bi2,eq:bi3} as
		\begin{align*}
			\| v_1 \nabla v_2 \cdot \nabla v_3 \|_{H^{2\gamma}} 
			&\lesssim \| v_1 \|_{H^{\frac{d}{2}+\vep}} \| \nabla v_2 \cdot \nabla v_3 \|_{H^{2\gamma}} 
			\lesssim \| v_1 \|_{H^{\frac{d}{2}+\vep}} \| \nabla v_2 \|_{H^{\frac{d}{4} + \gamma}} \| \nabla v_3 \|_{H^{\frac{d}{4} + \gamma}} \\
			&\lesssim \| v_1 \|_{H^{\frac{d}{2}+\vep}} \| v_2 \|_{H^{\frac{d}{4} + \gamma + 1}} \| v_3 \|_{H^{\frac{d}{4} + \gamma+1}}. 
		\end{align*} 
		To estimate $e_2$ and $e_3$ in \cref{eq:E1_decomp}, we need to estimate terms of the form \cref{eq:T1T2}, which can be estimated in the same way as (4.16) and (4.17) in \cite{alama2023symmetric} (see appendix): 
		\begin{equation}\label{eq:T1T2_est}
			\| T_1 \|_{L^2} + \| T_2 \|_{L^2} \leq \tau^\gamma C(M_\sigma). 
		\end{equation}
		By \cref{eq:T1T2_est}, we have $\| e_2 \|_{L^2} + \| e_3 \|_{L^2} \leq \tau^\gamma C(M_\sigma)$, which implies, together with \cref{eq:e1_L2}, \begin{equation}\label{eq:E1_est_L2}
	\| \mathcal{E}_1 \|_{L^2} \leq \tau^\gamma C(M_\sigma). 
	\end{equation}	 
	Then we estimate $\mathcal{E}_2$ \cref{eq:E2_def}. By \cref{eq:exp_est}, 
		\begin{align}\label{eq:E2_est_L2}
			\| \mathcal{E}_2 \|_{L^2} 
			&\leq \| (e^{-is\Delta} - e^{is\Delta}) \partial_\zeta \Gamma(\zeta, s) \|_{L^2} + \| \partial_\zeta \Gamma(\zeta, s) - \partial_\zeta \Gamma(-\zeta, -s) \|_{L^2} \notag \\
			&\lesssim \tau^\gamma \| \partial_\zeta \Gamma(\zeta, s) \|_{H^{2\gamma}} + \| \partial_\zeta \Gamma(\zeta, s) - \partial_\zeta \Gamma(-\zeta, -s) \|_{L^2}. 
		\end{align}
		From {the expression of $\partial_\zeta \Gamma(\zeta, s)$ in terms of the Gateaux derivative of $F$} \cref{eq:parGamma}, {recalling the notation \cref{eq:notation_simp} and using the bilinear estimate \cref{eq:bi}}, we have
		\begin{equation}\label{eq:E2_est_L2_1}
			\| \partial_\zeta \Gamma(\zeta, s) \|_{H^{2\gamma}} \lesssim \| \partial_\zeta \Gamma(\zeta, s) \|_{H^{\max\{2\gamma, \frac{d}{2}+\vep\}}} \lesssim M^3_{\max\{2\gamma, \frac{d}{2}+\vep\}}, 
		\end{equation}
		where, according to our definition of $\vep$, $\max\{2\gamma, \frac{d}{2}+\vep\} = 2\gamma$ only if $\gamma > d/4$. Recalling \cref{eq:parGamma}, using \cref{eq:bi3} with $\gamma=0$, we have
		\begin{align}\label{eq:E2_est_L2_2}
			\| \partial_\zeta \Gamma(\zeta, s) - \partial_\zeta \Gamma(-\zeta, -s) \|_{L^2} 
			&\lesssim C(M_{\frac{d}{2}+\vep})\| e^{is\Delta} v_n(\zeta)  - e^{-is\Delta} v_n(-\zeta) \|_{L^2} \notag \\
			&\quad + C(M_{\frac{d}{2}+\vep}) \| e^{is\Delta} w_n(\zeta) - e^{-is\Delta} w_n(-\zeta) \|_{L^2}. 
		\end{align}
		From \cref{eq:vn_int}, using $H^\frac{d}{3} \hookrightarrow L^6$, we have
		\begin{align}\label{eq:vn_est_L2}
			\| v_n(t) - v_n(0) \|_{L^2} 
			&\leq \tau \sup_{|s|\leq\tau} \| w_n(s) \|_{L^2} \lesssim \tau \sup_{|s|\leq\tau} \| e^{is\Delta} v_n(s) \|_{L^6}^3 \notag \\
			& \lesssim \tau \sup_{|s|\leq\tau} \| e^{is\Delta} v_n(s) \|_{H^\frac{d}{3}}^3 \leq \tau C(M_{\frac{d}{3}}). 
		\end{align}
		By \cref{eq:vn_est_L2,eq:exp_est}, we have
		\begin{align}\label{eq:T4}
			&\| e^{is\Delta} v_n(\zeta)  - e^{-is\Delta} v_n(-\zeta) \|_{L^2} \notag \\
			&\leq \| (e^{is\Delta} - e^{-is\Delta}) v_n(0) \|_{L^2} + \| v_n(\zeta)  - v_n(0) \|_{L^2} + \| v_n(-\zeta) - v_n(0) \|_{L^2} \notag \\
			&\lesssim \tau^\gamma \| v_n(0) \|_{H^{2\gamma}} + \tau C(M_\frac{d}{3}) \leq \tau^\gamma M_{2\gamma} + \tau C(M_\frac{d}{3}). 
		\end{align}
		Using {the bilinear estimate \cref{eq:bi}, the bound on $(e^{it\Delta}-I)$ \cref{eq:exp_est},  the refined bilinear estimate \cref{eq:bi3} with $\gamma=0$,} and \cref{eq:T4} with $s = \zeta$, we have
		\begin{align}\label{eq:T5}
			&\| e^{is\Delta} w_n(\zeta)  - e^{-is\Delta} w_n(-\zeta) \|_{L^2} \notag \\
			&\leq \| (e^{is\Delta} -  e^{-is\Delta}) w_n(\zeta) \|_{L^2} + \| w_n(\zeta) - w_n(-\zeta) \|_{L^2} \notag \\
			&\lesssim \tau^\gamma \| w_n(\zeta) \|_{H^{2\gamma}} + \| (e^{-i\zeta\Delta} -  e^{i\zeta\Delta}) F(e^{i\zeta\Delta} v_n(\zeta)) \|_{L^2} \notag \\
			&\quad + \| F(e^{i\zeta\Delta} v_n(\zeta)) - F(e^{-i\zeta\Delta} v_n(-\zeta)) \|_{L^2} \notag \\
			&\lesssim \tau^\gamma \| F(e^{i\zeta\Delta}v_n(\zeta)) \|_{H^{2\gamma}} \notag \\
			&\quad + \| e^{i\zeta\Delta} v_n(\zeta) - e^{-i\zeta\Delta} v_n(-\zeta) \|_{L^2} (\| v_n(\zeta) \|_{H^{\frac{d}{2}+\vep}}^2+ \| v_n(-\zeta) \|_{H^{\frac{d}{2}+\vep}}^2) \notag \\
			&\lesssim \tau^\gamma M_{\max\{2\gamma, \frac{d}{2}+\vep\}}^3 + ( \tau^\gamma M_{2\gamma} + \tau C(M_\frac{d}{3})) M_{\frac{d}{2}+\vep}^2 \lesssim \tau^\gamma C(M_{\max\{2\gamma, \frac{d}{2}+\vep\}}). 
		\end{align}
		Plugging \cref{eq:T4,eq:T5} into \cref{eq:E2_est_L2_2}, together with \cref{eq:E2_est_L2_1}, we get from \cref{eq:E2_est_L2} that
		$\| \mathcal{E}_2 \|_{L^2} \leq \tau^{\gamma} C(M_\sigma)$. 
		The proof is thus completed by combining \cref{eq:E1_est_L2,eq:E2_est_L2} and the proof of \cref{thm:sLRI_decomp}. 
	\end{proof}
	
	\begin{proof}[Proof of \cref{cor:globalerrorinL2}]
		Let $0 \leq \gamma \leq 1$. 
		Recalling \cref{eq:LTEsLRI1}, by \cref{eq:bi}, we have $\| \mathcal{L}^n \|_{H^\sigma} \lesssim C(M_{\sigma}) \tau.$ By interpolation, we have
		\begin{equation}
			\| \mathcal{L}^n \|_{H^{\widetilde{\alpha}}} \lesssim \| \mathcal{L}^n \|_{L^2}^{1-\frac{\widetilde{\alpha}}{\sigma}} \| \mathcal{L}^n \|_{H^\sigma}^{\frac{\widetilde{\alpha}}{\sigma}} \lesssim \tau^{1+(1+\gamma)(1-\widetilde{\alpha}/\sigma)}, \quad 0 \leq \widetilde{\alpha} \leq \sigma, 
		\end{equation}
		which yields\vspace{-0.2cm}
		\begin{equation}\label{eq:LTE_low}
			\| \mathcal{L}^n \|_{H^{\widetilde{\alpha}}} \leq C(M_\sigma)\tau^{1+(1+\gamma)(1-\widetilde\alpha/\sigma)}, \quad 0 \leq \widetilde{\alpha} \leq \sigma. 
		\end{equation}
		Choose a fixed $\widetilde \alpha$ such that $d/2<\widetilde\alpha<\sigma$. Following the proof of \cref{thm:global_error_multistep_schemes} with \cref{eq:LTE_low,prop:stability_multistep}, we get\vspace{-0.2cm}
		\begin{equation}\label{eq:uniformbound}
			\| u^n \|_{H^{\widetilde\alpha}} \leq C(M_{\sigma}), \quad 0 \leq n \leq T/\tau. 
		\end{equation}
		Noting \cref{eq:uniformbound}, using \cref{thm:LTE_L2} and the stability estimate for $\widetilde \Phi_{\pm \tau}$ in \cref{LRI1pm} as
		\begin{equation}
			\| \widetilde \Phi_{\pm \tau}(v) - \widetilde \Phi_{\pm \tau}(w) \|_{L^2} \leq C(\| v \|_{H^{\widetilde \alpha}}, \| w \|_{H^{\widetilde \alpha}})\tau \| v-w \|_{L^2}, 
		\end{equation}
		we complete the proof. 
	\end{proof}
	
	\begin{proof}[Proof of \cref{thm:sLRI1_H1}]
		Following the proof of \cref{cor:globalerrorinL2}, the global convergence in $H^1$-norm reduces to the estimate of $\mathcal{L}^n$ defined in \cref{eq:LTEsLRI1}, which further reduces to the estimates of $\mathcal{E}_1$ \cref{eq:E1_def} and $\mathcal{E}_2$ \cref{eq:E2_def} by \cref{thm:sLRI_decomp}. For $\mathcal{E}_1$, noting \cref{eq:E1_decomp} holds for $\alpha = 1$, it suffices to estimate $\| e_1 \|_{H^{1+2\gamma}}$, $\| e_2 \|_{H^1}$, and $\| e_3 \|_{H^1}$. By \cref{eq:bi}, noting $\sigma_1-1>d/2$ and $ \sigma_1 - 1 \geq1+2\gamma $, we have
		\begin{align}\label{eq:e1_H1}
			\| e_1 \|_{H^{1+2\gamma}} \lesssim \| e_1 \|_{H^{\sigma_1-1}} \leq C(M_{\sigma_1}). 
		\end{align}
		The estimates of $e_2$ and $e_3$ again reduce to the estimates of $T_2$ and $T_3$ in \cref{eq:T1T2}. By \cref{eq:bi3} with $\gamma = 1/2$,  
		\begin{align*}
			&\| T_1 \|_{H^1} \lesssim \| w(t) - w(-t) \|_{H^1} \| \nabla w_1 \|_{H^{\frac{d}{2}+\vep}} \| \nabla w_2 \|_{H^{\frac{d}{2}+\vep}} \lesssim \tau^\gamma M_{1+2\gamma} M_{\frac{d}{2}+\vep+1}^2, \quad \\
			&\| T_2 \|_{H^1} \lesssim \| w_1 \|_{H^{\frac{d}{2}+\vep}} \| \nabla (w(t) - w(-t)) \|_{H^1} \| \nabla w_2 \|_{H^{\frac{d}{2}+\vep}} \lesssim \tau^\gamma M_{\frac{d}{2}+\vep}M_{2+2\gamma}M_{\frac{d}{2}+\vep+1}, 
		\end{align*}
		which implies $\| e_2 \|_{H^1} + \| e_3 \|_{H^1} \lesssim \tau^\gamma C(M_{\sigma_1})$ and further implies $\| \mathcal{E}_1 \|_{H^1} \leq \tau^\gamma C(M_{\sigma_1})$. 
		To estimate $\mathcal{E}_2$ \cref{eq:E2_def}, similar to \cref{eq:E2_est_L2}, we have
		\begin{align*}
			\| \mathcal{E}_2 \|_{H^1} 
			&\lesssim \tau^\gamma \| \partial_\zeta \Gamma(\zeta, s) \|_{H^{1+2\gamma}} + \| \partial_\zeta \Gamma(\zeta, s) - \partial_\zeta \Gamma(-\zeta, -s) \|_{H^1} \notag \\
			&\lesssim \tau^\gamma \| \partial_\zeta \Gamma(\zeta, s) \|_{H^{\sigma_1- 1}} + \| \partial_\zeta \Gamma(\zeta, s) - \partial_\zeta \Gamma(-\zeta, -s) \|_{H^{\frac{d}{2}+\vep}}, 
		\end{align*}
		which implies, by \cref{eq:bi} and following a similar procedure in establishing \cref{eq:E2_est}, 
		\begin{equation*}
			\| \mathcal{E}_2 \|_{H^1} \lesssim \tau^\gamma C(M_{\sigma_1-1}) + \tau^\gamma C(M_{\frac{d}{2}+\vep+2\gamma}) \leq \tau^\gamma C(M_{\sigma_1}). 
		\end{equation*}
		The proof is thus completed. 
	\end{proof}
	
	\subsection{{Convergence analysis for sLRI2}}\label{sec:application_to_sLRI2}
	Again we focus on the non-symmetric integrator \eqref{eqn:bruned_schratz_lri} which served as the basis for the construction of sLRI2. Note that in this case the method is of even order ($p=2$ subject to sufficient regularity in the solution) so we cannot expect any improvement in the convergence order by symmetrisation (i.e. Assumption~\ref{assumption:general_properties_higher_order_schemes} {(iii)} is not satisfied). 
 The local error of the method \eqref{eqn:bruned_schratz_lri} was studied in \cite{bruned_schratz_2022} and the stability of the scheme (together with a global convergence analysis) was proved in \cite{ostermann2022second}, 
 {showing that the numerical flow $\Phi^\text{LRI2}_{t}$ satisfies Assumptions~\ref{assumption:general_properties_higher_order_schemes} (i) and (ii) for any $\alpha>d/2$ and $p=2$ with $\gamma_1=p$ (see Corollary~5.3 in \cite{bruned_schratz_2022} and Lemmas~4.1 \& 4.2 in \cite{ostermann2022second} for the proof). Hence, as an immediately corollary of \cref{thm:global_error_multistep_schemes}, we have the following second-order convergence result of the sLRI2. 
	\begin{corollary}\label{cor:sLRI2_general}
		Let $\alpha > d/2$. Suppose that $u^n$ is computed from \cref{eq:sLRI2}. For all $0<\tau\leq \tau_0$ with $\tau_0>0$ depending on $M_{\alpha+2}$ and $T$, we have 
		\begin{equation}
			\|u(t_n) - u^n\|_{H^\alpha} \leq C(M_{\alpha+2})\tau^2, \quad 0 \leq n \leq T/\tau.
		\end{equation}
	\end{corollary}}

{It is also possible to characterise the convergence properties of LRI2 \cref{eqn:bruned_schratz_lri} and hence of sLRI2 in $L^2$-norm and $H^1$-norm (when $d=2, 3$). The central step in this was taken by \cite{bronsard2023error} and we recall it here. By Proposition 4.9 in \cite{bronsard2023error}, the LRI2 \cref{eqn:bruned_schratz_lri} satisfies \cref{assumption:general_properties_higher_order_schemes} (ii) for $0 \leq \alpha \leq d/2$ and $p=2$ with  
	\begin{align*}
		\gamma_1=\begin{cases}
			2+d/4,&\text{\ if\ }\alpha=0,\\
			2+d/2+{\varepsilon},&\text{\ if\ }0< \alpha \leq d/2.\\     
		\end{cases}
	\end{align*}
	In addition, analogous to (32) \& (33) in \cite{bronsard2023error} (cf. also Section 4.4 in \cite{bronsard2023error}), we have the following stability estimate: Fixing $\widetilde \alpha=1+d/4$, when $0 \leq \alpha \leq d/2$, 
 for any $v,w\in H^{\widetilde \alpha}$, \begin{equation*}\label{eqn:stability_estimate_yvonne}
		\|\Phi_{\pm \tau}(v)-\Phi_{\pm \tau}(w)\|_{H^\alpha}\leq e^{\tau M(\|v\|_{\widetilde \alpha},\|w\|_{\widetilde \alpha})}\|v-w\|_{H^\alpha}.
	\end{equation*}
	These allow us to deduce the following error estimates in $L^2$- and $H^1$-norms.
	\begin{theorem}\label{thm:sLRI2inL2}
		Suppose $u^{n}$ is computed using \eqref{eq:sLRI2}. For all $0<\tau\leq \tau_0$ with $\tau_0>0$ depending on $M_{2+d/4}$ and $T$, we have
		\begin{equation*}
			\|u(t_n) - u^n\|_{L^2} \leq C(M_{2+d/4})\tau^2, \quad 0 \leq n \leq T/\tau. 
		\end{equation*}
	\end{theorem}
	
	\begin{theorem}\label{thm:sLRI2inH1}
		Let $d=2, 3$. Suppose $u^{n}$ is computed using \eqref{eq:sLRI2}. For all $0<\tau\leq \tau_0$ with $\tau_0>0$ depending on $M_{2+d/2+\vep}$ and $T$, we have
		\begin{equation*}
			\|u(t_n) - u^n\|_{H^1} \leq C(M_{2+d/2+\vep})\tau^2, \quad 0 \leq n \leq T/\tau. 
		\end{equation*}
	\end{theorem}}
The proofs of \cref{thm:sLRI2inL2,thm:sLRI2inH1} follow similarly to that of \cref{cor:globalerrorinL2} and are thus omitted. 

	\section{Numerical examples}\label{sec:numerical_experiments}
	
	In this section, we shall show some numerical results to validate our error estimates and to demonstrate the superiority of the newly proposed explicit sLRIs, with sLRI1 \cref{eq:sLRI1} and sLRI2 \cref{eq:sLRI2} as examples. We first test the convergence of the sLRI1 and sLRI2 with initial data of varying regularity. Then we perform some long-time simulations to show the near-conservation of mass and energy. We also present comparisons with existing LRIs to highlight the advantages of the explicit sLRIs in accuracy, efficiency, and long-time performance. 
	
	In the following, we only consider the one dimensional case and fix $d=1$. {Note that in this case the NLSE is completely integrable and thus falls into the class of equations for which (at least in the fully discrete case) rigorous long-time results are available (cf. \cite{hairer2013geometric}). As a word of caution, there are currently no rigorous guarantees on the long-time performance of symmetric methods applied to non-integrable PDEs.} To quantify the error, we introduce the following error functions:
	\begin{equation*}
		e_{L^2}(t_n) := \| u(\cdot, t_n) - u^n  \|_{L^2}, \quad  e_{H^1}(t_n) := \| u(\cdot, t_n) - u^{n} \|_{H^1}, \quad 0 \leq n \leq T/\tau. 
	\end{equation*}
	
	We generate an $H^\alpha$-initial datum $u_0$ through
	\begin{equation}\label{eq:ini}
		u_0(x) = \frac{\phi(x)}{\| \phi \|_{L^2}}, \quad \phi(x) = \sum_{l = -N_\text{ref}}^{N_\text{ref}-1} \frac{\xi_l}{\langle l \rangle^{\alpha+1/2}} e^{2 \pi i l x}, \qquad x \in \mathbb{T},  
	\end{equation}
	where $\xi_l = \text{rand}(-1, 1)+i \  \text{rand}(-1, 1)$ with $\text{rand}(-1, 1)$ returning a random number uniformly distributed in $[-1, 1]$, $N_\text{ref} = 2^{17}$ being the maximum frequency of the numerical initial datum (which is chosen much larger than the maximum frequency of the numerical solution), and 
	\begin{equation*}
		\langle s \rangle = \left\{
		\begin{aligned}
			&|s|, &s \neq 0, \\
			&1, &s=0, 
		\end{aligned}
		\right., \quad s \in \R. 
	\end{equation*}
    {We remark here that although we adopt the random initial data \cref{eq:ini} in our numerical experiments for convergence tests and long-time simulations, the same results can be observed for other choices of deterministic initial data.} {In all of the following experiments our spatial discretisation is a Fourier pseudospectral method on a uniform mesh in $x$ with spatial resolution $h$.}
    
	First, we show errors in $L^2$- and $H^1$-norms of the explicit sLRIs applied to the NLSE \cref{NLSE} with $H^\alpha$-initial data \cref{eq:ini} for different $\alpha$. We only present the results for $\mu = 1$ for simplicity. Similar results can be observed for $\mu = -1$. The {reference} solutions are obtained by the sLRI2 with $\tau = \tau_\text{e} = 10^{-5}$ and $h = h_\text{e} =\pi \times 2^{-13}$. In the computation, we fix $h = h_\text{e}$ and vary $\tau$ from $10^{-4}$ to $10^{-1}$. The results are presented in \cref{fig:conv_dt_sLRI1,fig:conv_dt_sLRI2} for the sLRI1 and sLRI2, respectively.
	
	\begin{figure}[htbp]
		\centering
		{\includegraphics[width=0.475\textwidth]{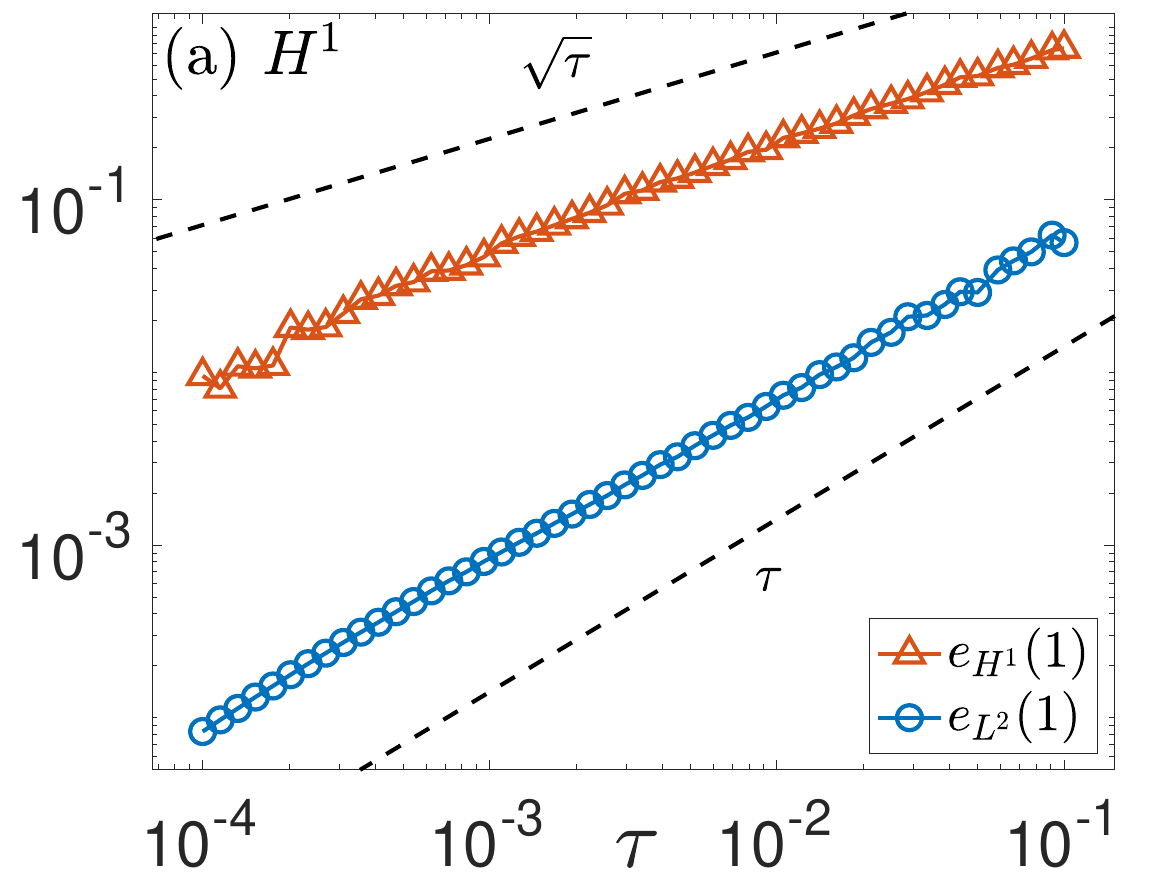}}\hspace{1em}
		{\includegraphics[width=0.475\textwidth]{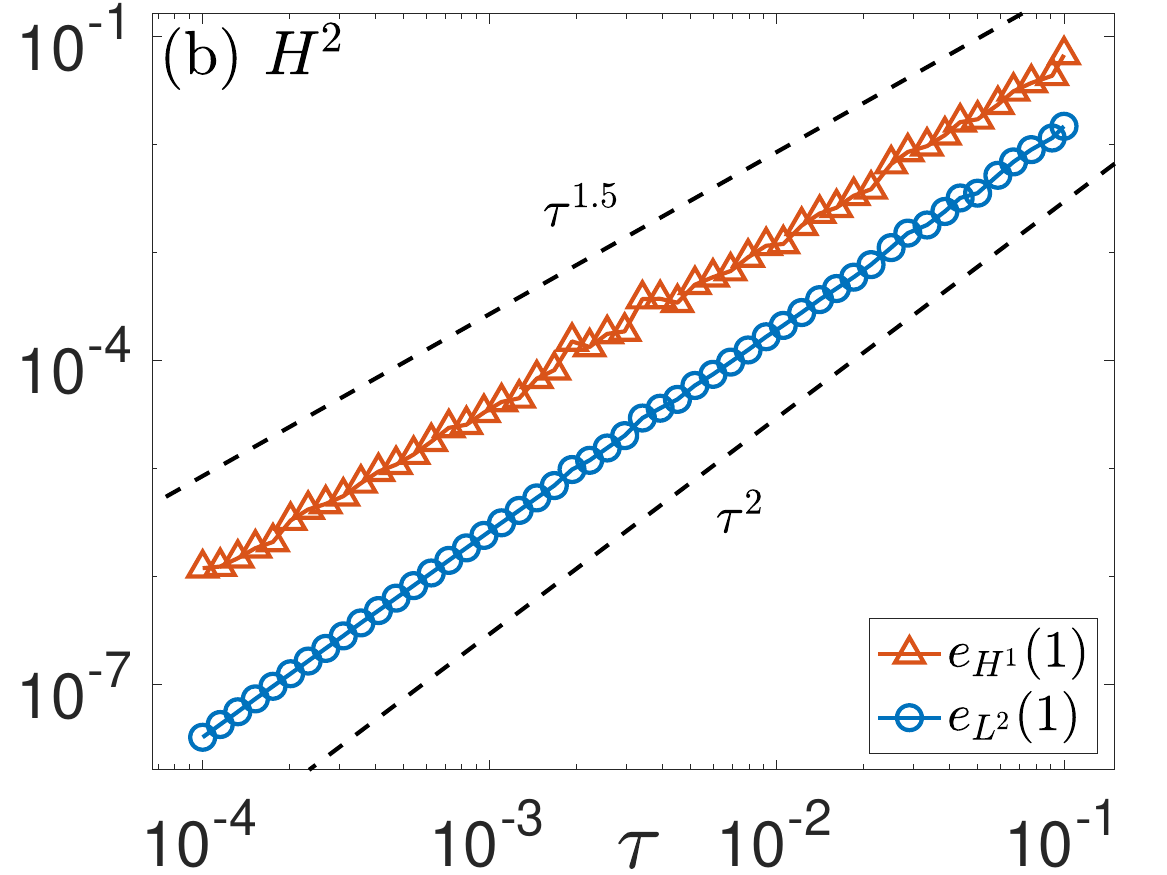}} \\
		{\includegraphics[width=0.475\textwidth]{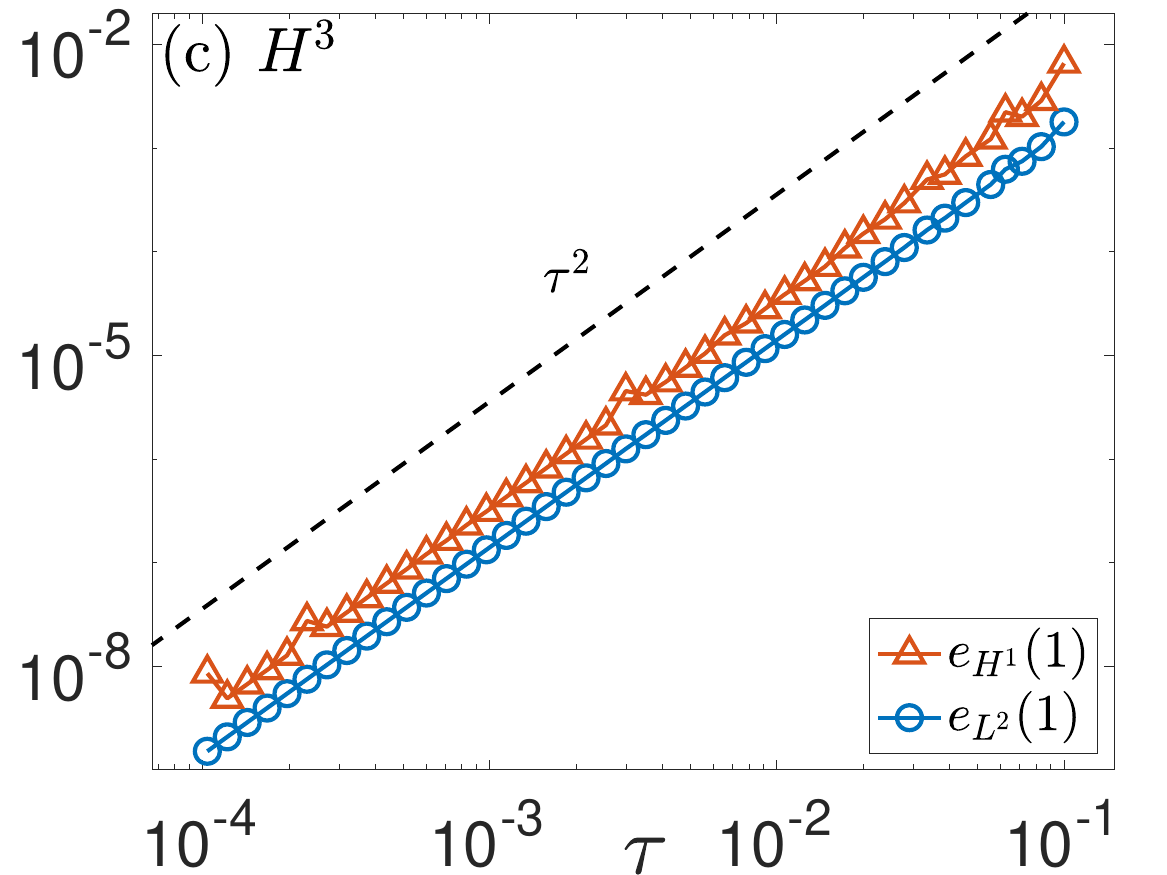}}\hspace{1em}
		{\includegraphics[width=0.475\textwidth]{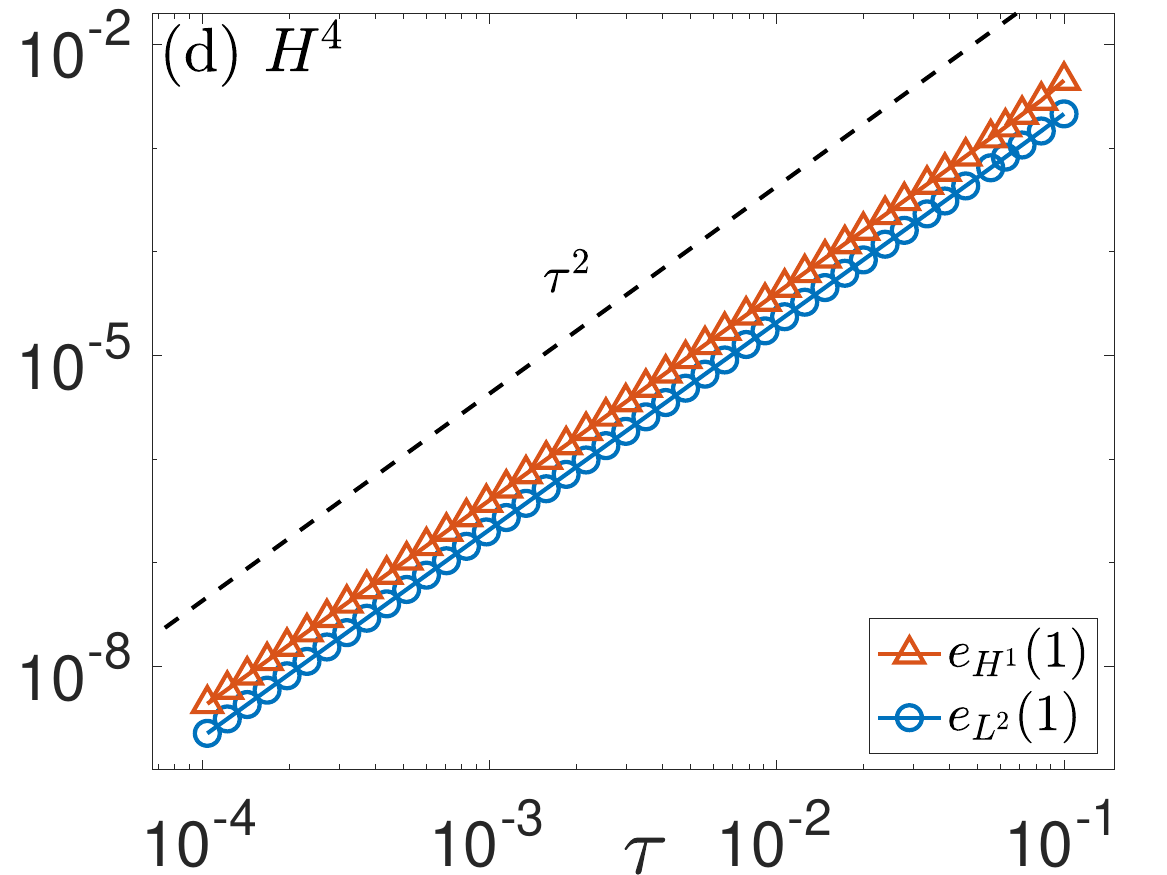}}
		\caption{Errors in $L^2$- and $H^1$-norms of the explicit sLRI1 for the NLSE \cref{NLSE} with initial data of regularity $H^\alpha \ (\alpha = 1, 2, 3, 4)$}
	\label{fig:conv_dt_sLRI1}
\end{figure}
\begin{figure}[htbp]
	\centering
	{\includegraphics[width=0.475\textwidth]{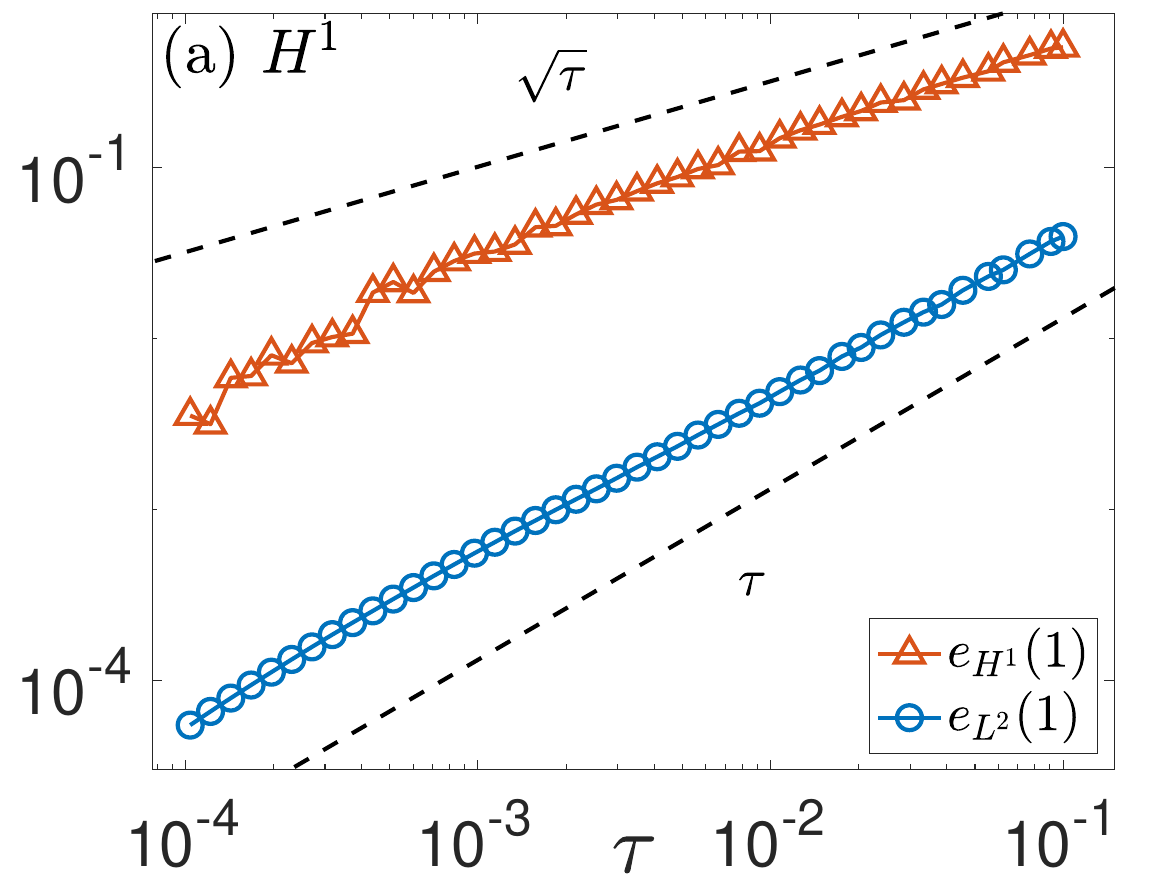}}\hspace{1em}
	{\includegraphics[width=0.475\textwidth]{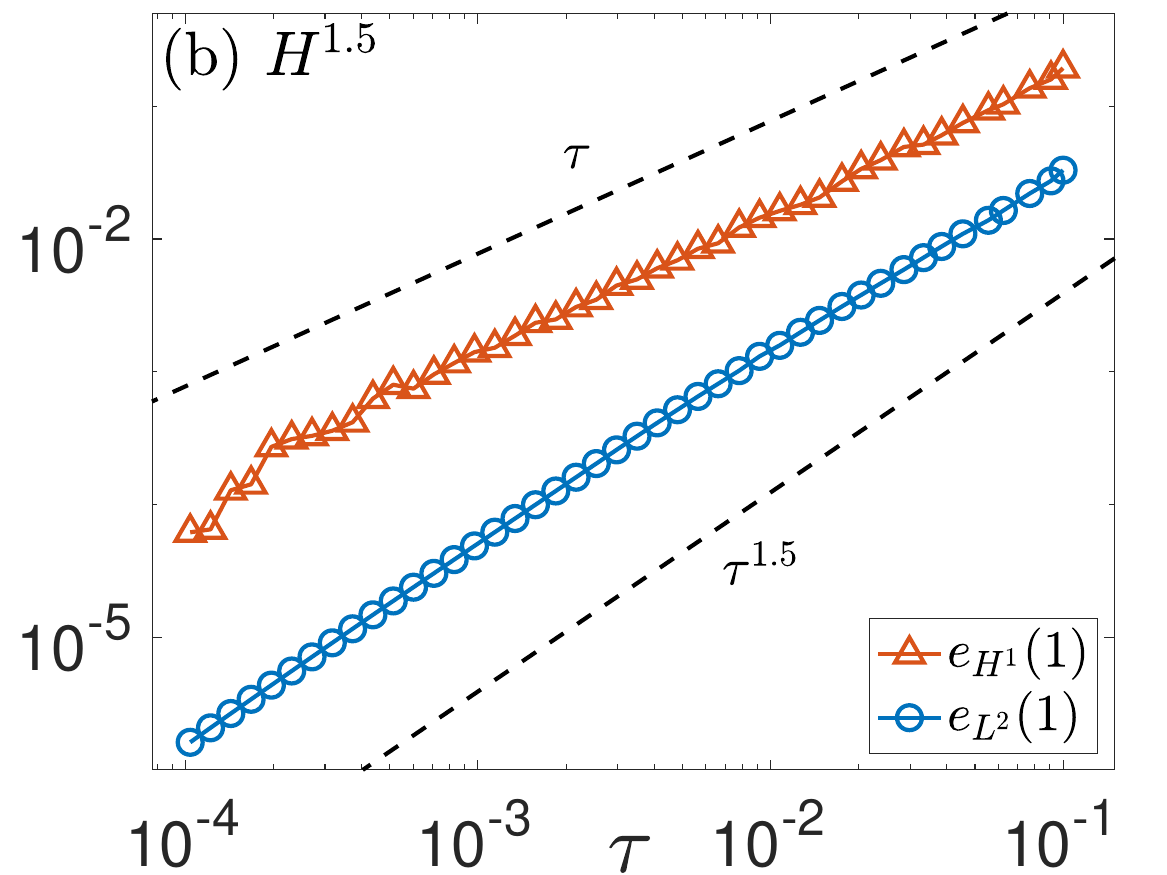}} \\
	{\includegraphics[width=0.475\textwidth]{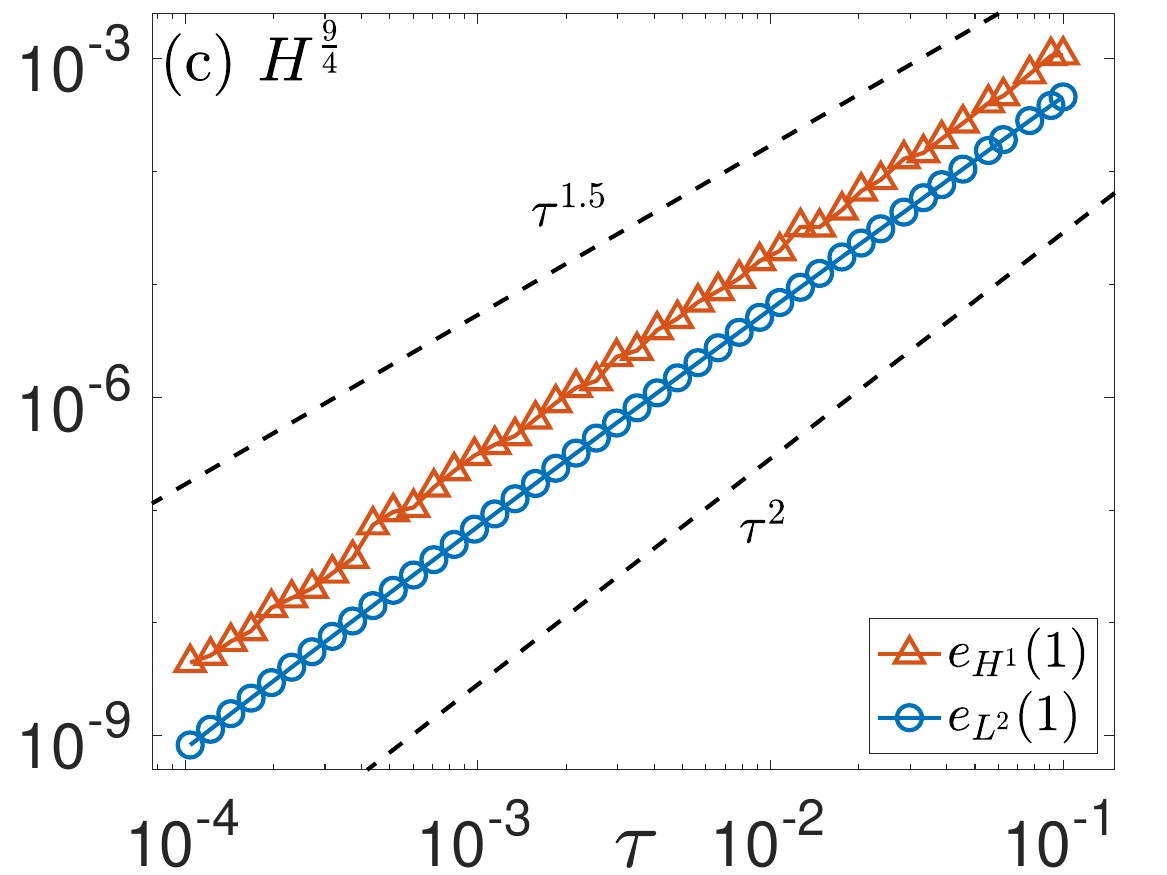}}\hspace{1em}
	{\includegraphics[width=0.475\textwidth]{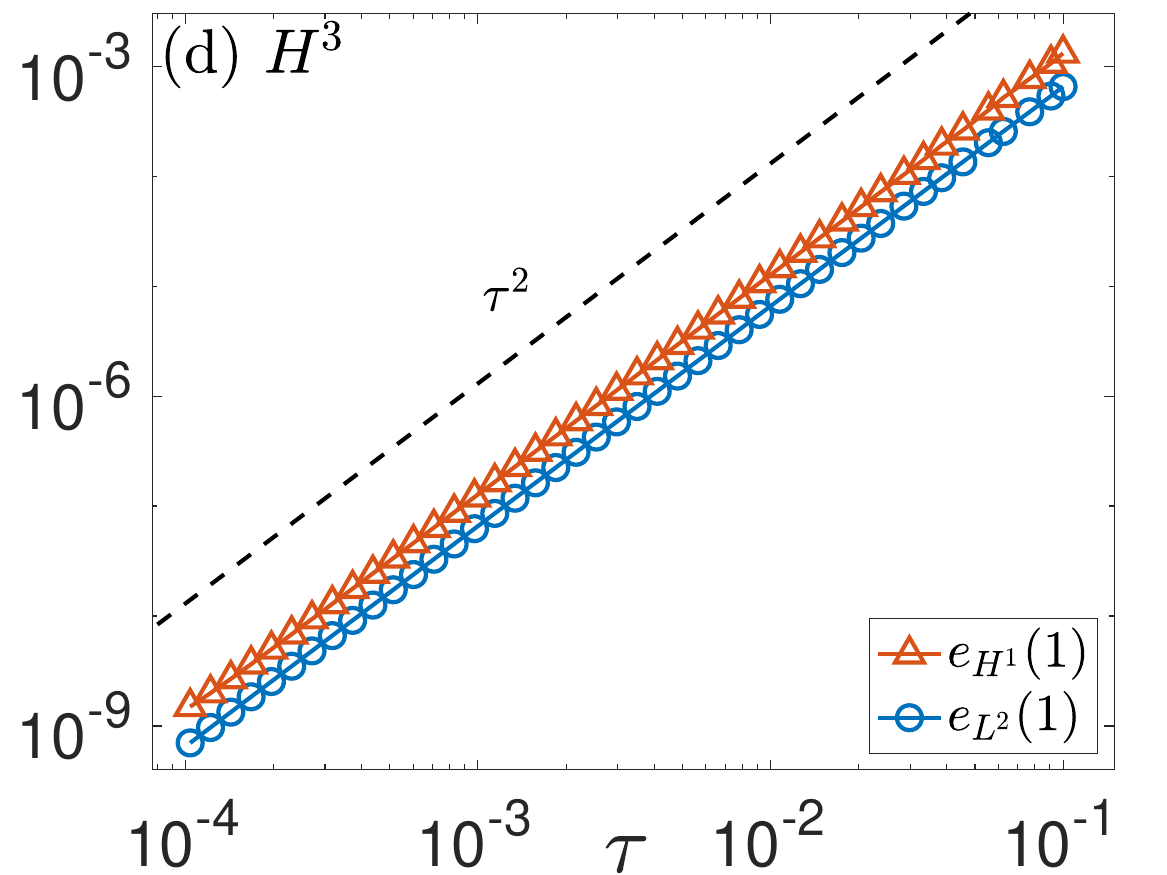}}
	\caption{Errors in $L^2$- and $H^1$-norms of the explicit sLRI2 for the NLSE \cref{NLSE} with initial data of regularity $H^\alpha \ (\alpha = 1, 1.5, \revisions{2.25}, 3)$}\vspace{-0.2cm}
\label{fig:conv_dt_sLRI2}
\end{figure}

From \cref{fig:conv_dt_sLRI1}, we can observe that for the sLRI1 method, second-order convergence in $L^2$- and $H^1$-norms can be observed for $H^3$- and $H^4$-initial data, respectively. In comparison, second-order convergence of the sLRI2 method can be observed in $L^2$- and $H^1$-norms for $H^\frac{9}{4}$- and $H^3$-initial data, respectively. {These observations are consistent with our error estimates in \cref{thm:improvedhigh,cor:globalerrorinL2,thm:sLRI2inL2,cor:sLRI2_general}. However, when the regularity requirements for optimal second-order convergence are slightly violated, the errors become oscillatory without clear order reduction. Also, for the fractional order convergence obtained for sLRI1 in \cref{thm:improvedhigh,cor:globalerrorinL2}, the numerical results exhibit slightly higher convergence orders than those predicted by the error estimates. It remains unclear whether the regularity assumptions in our error estimates are optimally weak.} 


Then we test the long-time performance of the explicit sLRIs under initial data \cref{eq:ini} of varying regularity. For comparison, we also show the results of the corresponding non-symmetric LRIs, namely LRI1 and LRI2. In this test, we choose $\mu = -1$. In computation, we choose $\tau = 0.005$ and $h = \pi \times 2^{-9}$ with the final time $T = 1000$. In \cref{fig:mass_energy_H1_rand_ini_comp,fig:mass_energy_H15_rand_ini_comp,fig:mass_energy_H2_rand_ini_comp,fig:mass_energy_H25_rand_ini_comp}, we plot the relative errors of {mass \eqref{eqn:mass} and energy \eqref{eqn:energy}} of different methods up to $T$ for initial data of varying regularity.

We can observe that the explicit symmetric LRIs demonstrate clear near conservation of both mass and energy up to very long time in all cases. In particular, sLRI2 performs slightly better than sLRI1 with smaller errors in mass and energy. In addition, the explicit ``symmetrized" integrators perform significantly superior to the non-symmetric integrators in terms of both the accuracy and the long-time behaviour, while their cost is nearly the same as the non-symmetric versions LRI1 \& LRI2. \revisions{Moreover, the additional numerical test for the non-integrable quintic NLSE in \cref{rem:quintic} suggests that such excellent long time performance of sLRI2 is not restricted to integrable systems.}
\begin{figure}[h!]
\centering
{\includegraphics[width=0.475\textwidth]{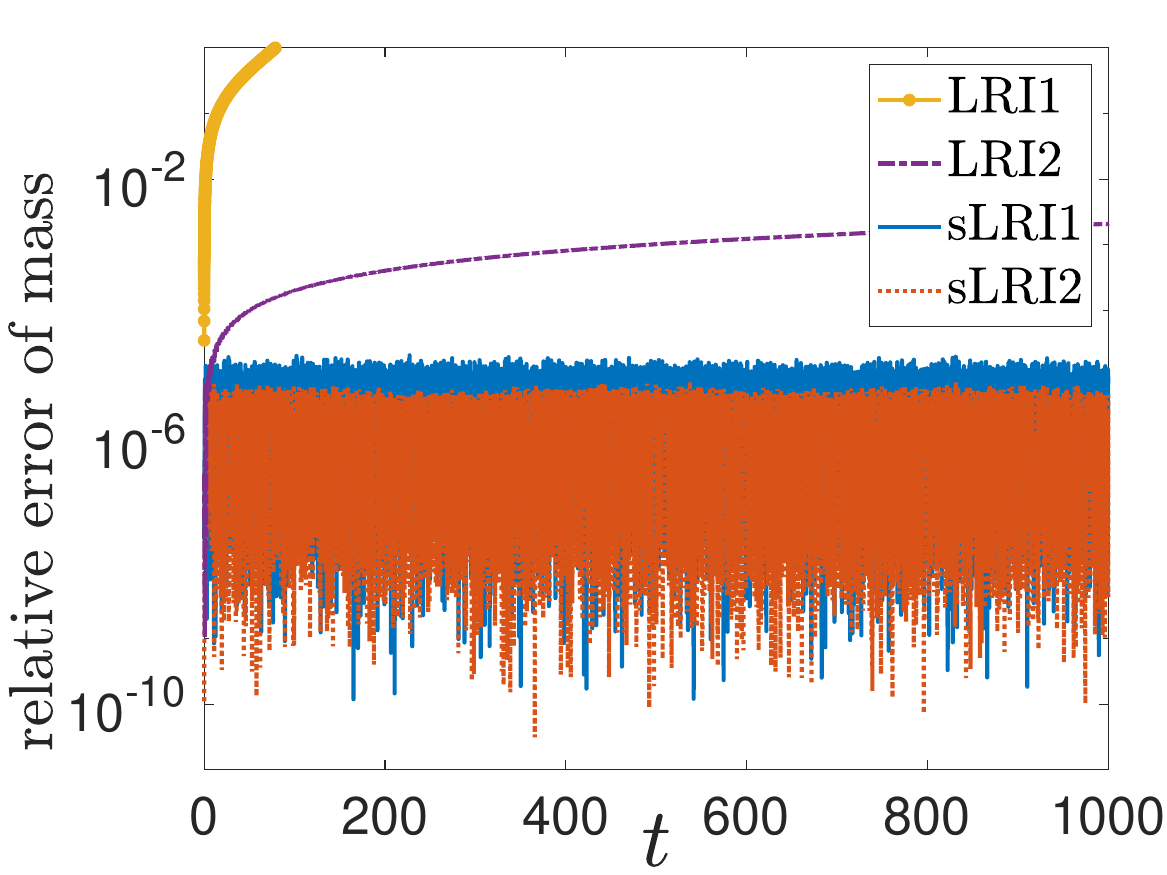}}\hspace{1em}
{\includegraphics[width=0.475\textwidth]{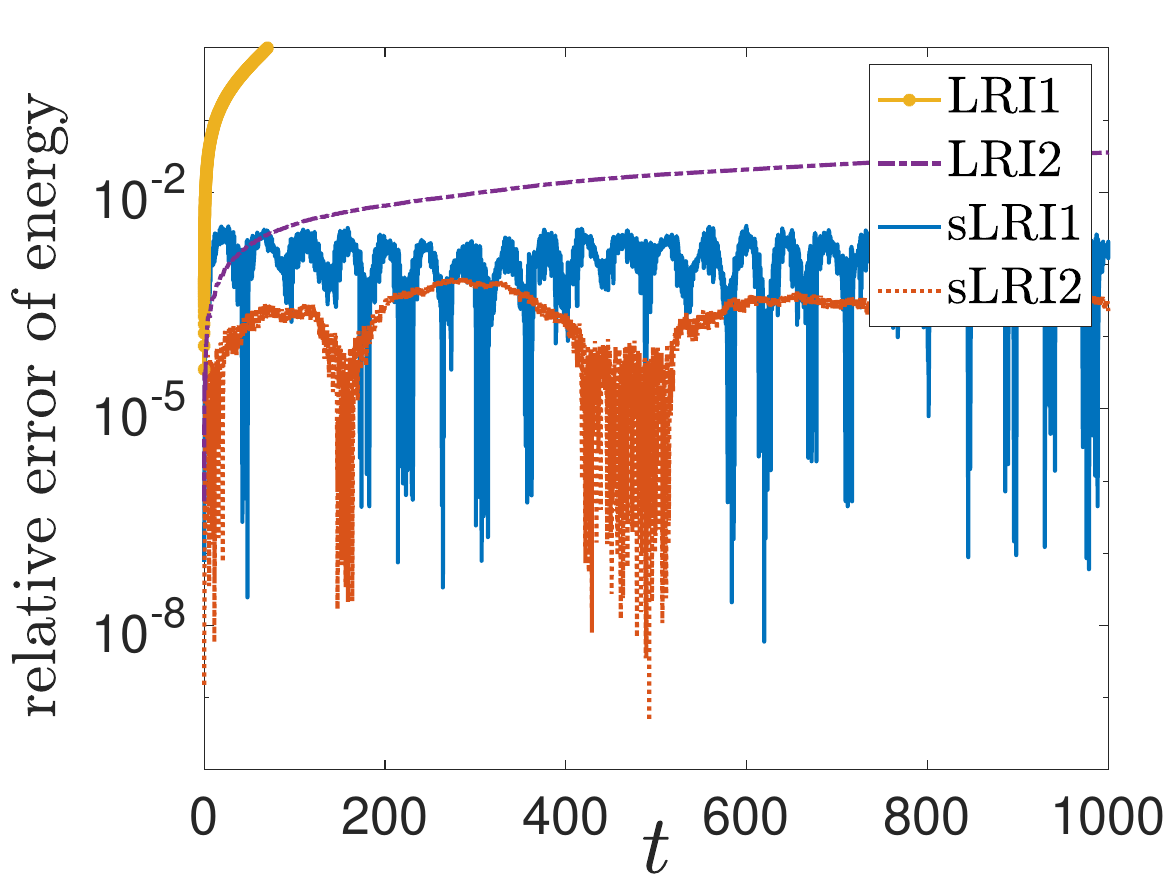}}
\caption{Relative errors of mass (left) and energy (right) of different methods for the NLSE \cref{NLSE} with $H^{1}$ initial datum}
\label{fig:mass_energy_H1_rand_ini_comp}
\end{figure}

\begin{figure}[htbp]
\centering
{\includegraphics[width=0.475\textwidth]{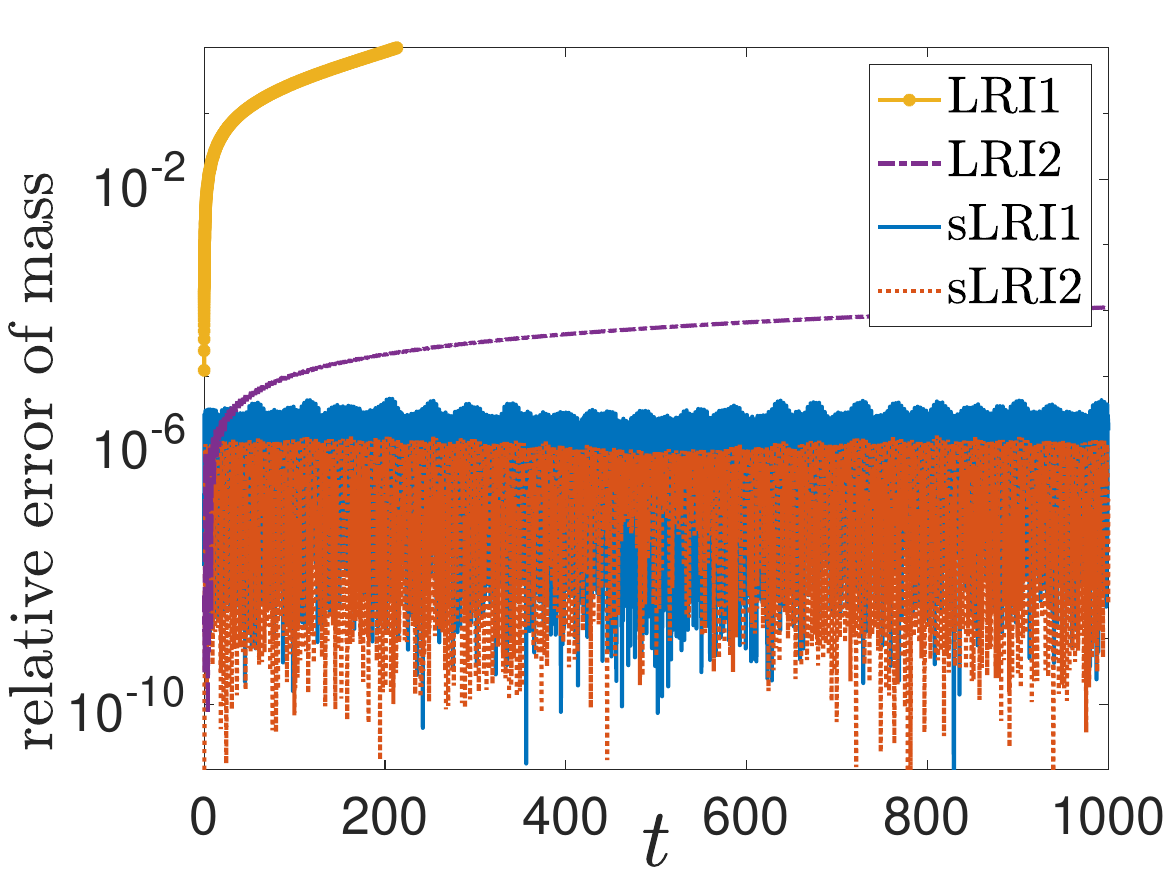}}\hspace{1em}
{\includegraphics[width=0.475\textwidth]{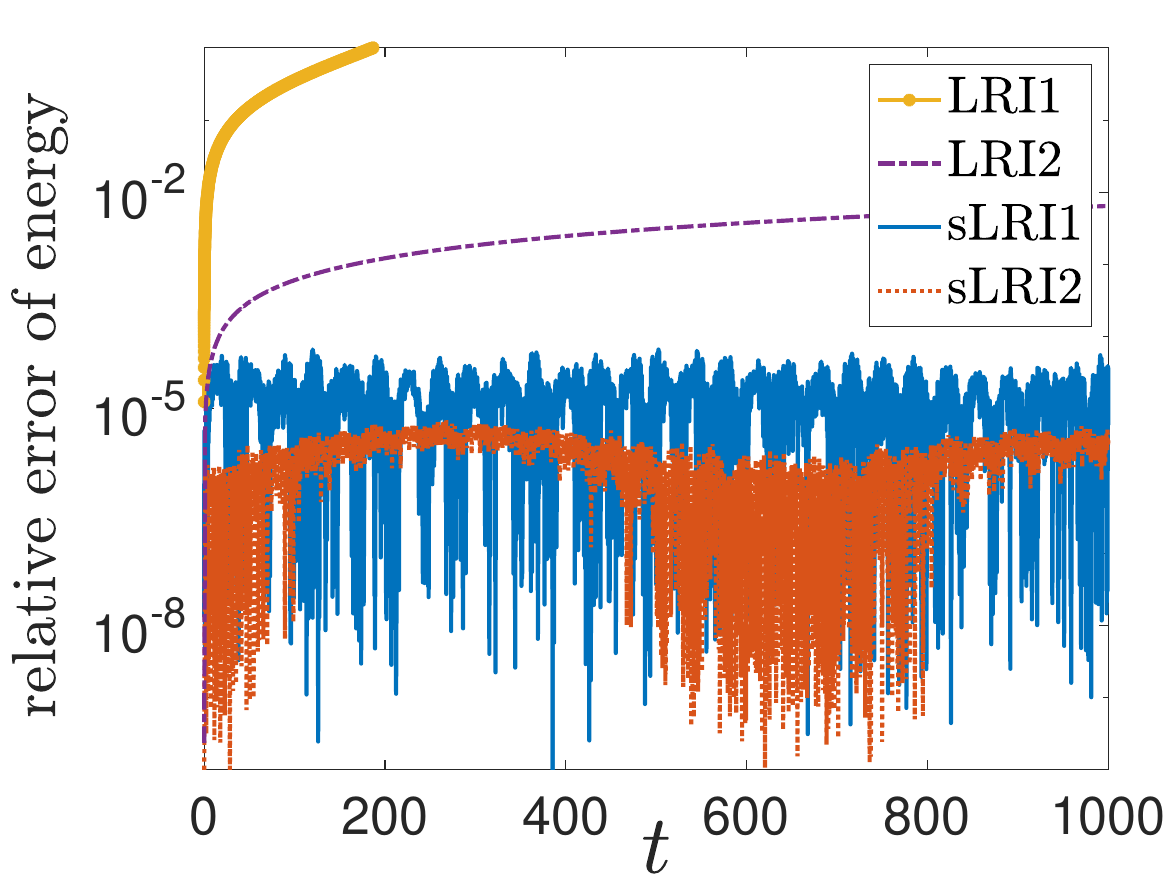}}
\caption{Relative errors of mass (left) and energy (right) of different methods for the NLSE \cref{NLSE} with $H^{1.5}$ initial datum}
\label{fig:mass_energy_H15_rand_ini_comp}
\end{figure}
\begin{figure}[h!]
\centering
{\includegraphics[width=0.475\textwidth]{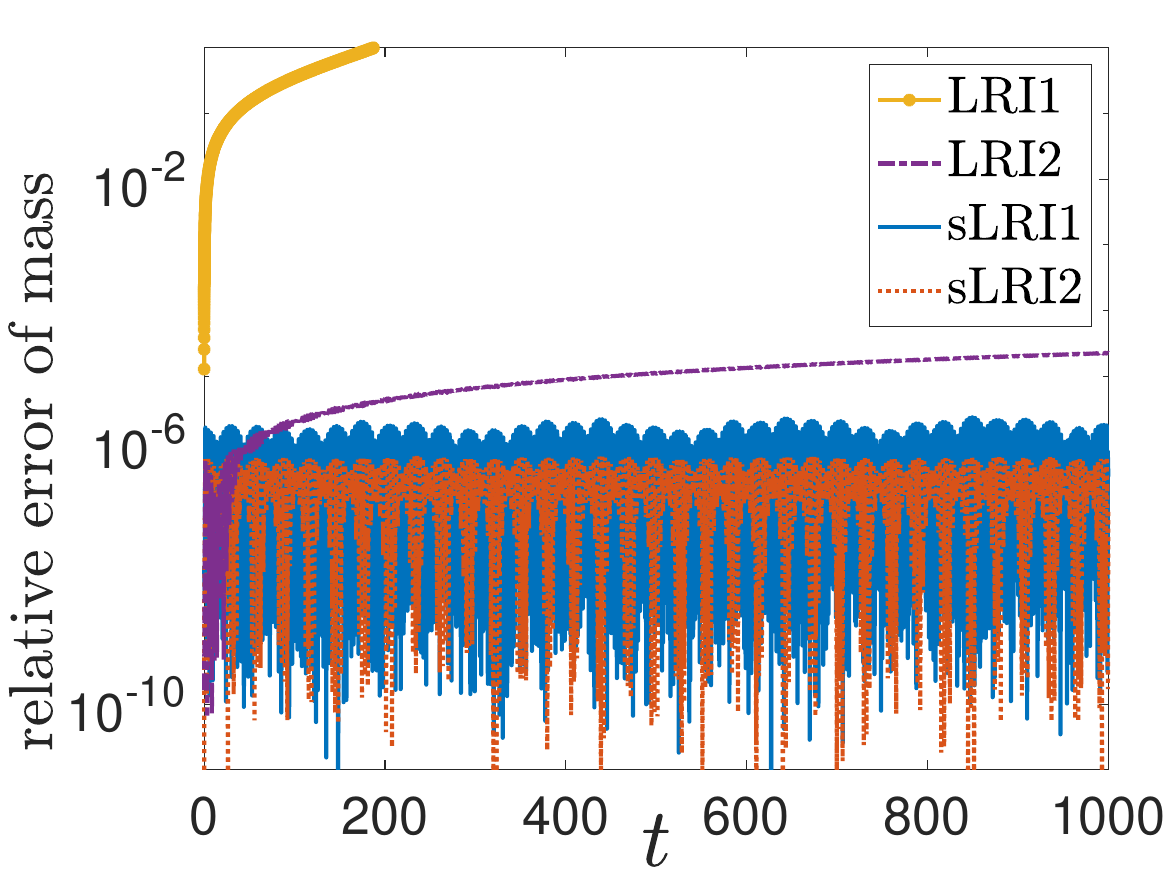}}\hspace{1em}
{\includegraphics[width=0.475\textwidth]{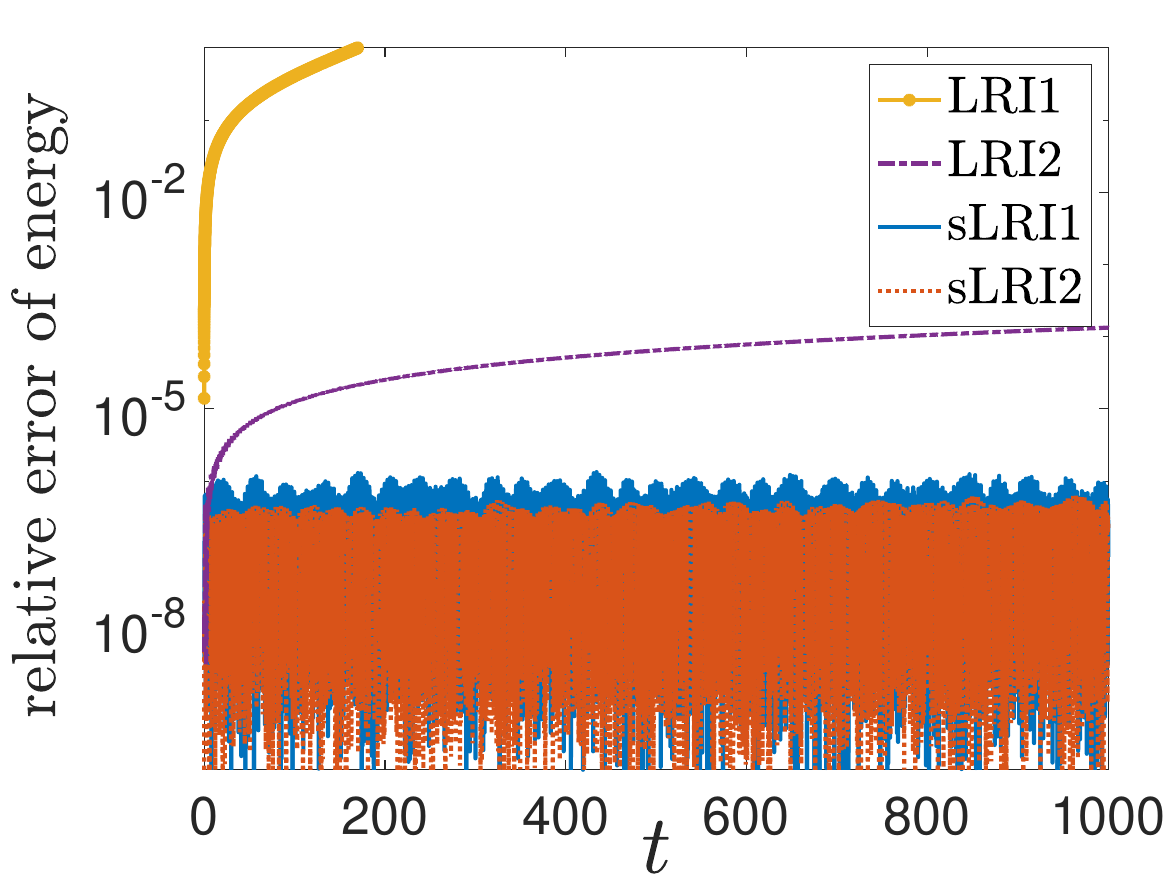}}
\caption{Relative errors of mass (left) and energy (right) of different methods for the NLSE \cref{NLSE} with $H^{2}$ initial datum}
\label{fig:mass_energy_H2_rand_ini_comp}
\end{figure}
\begin{figure}[h!]
\centering
{\includegraphics[width=0.475\textwidth]{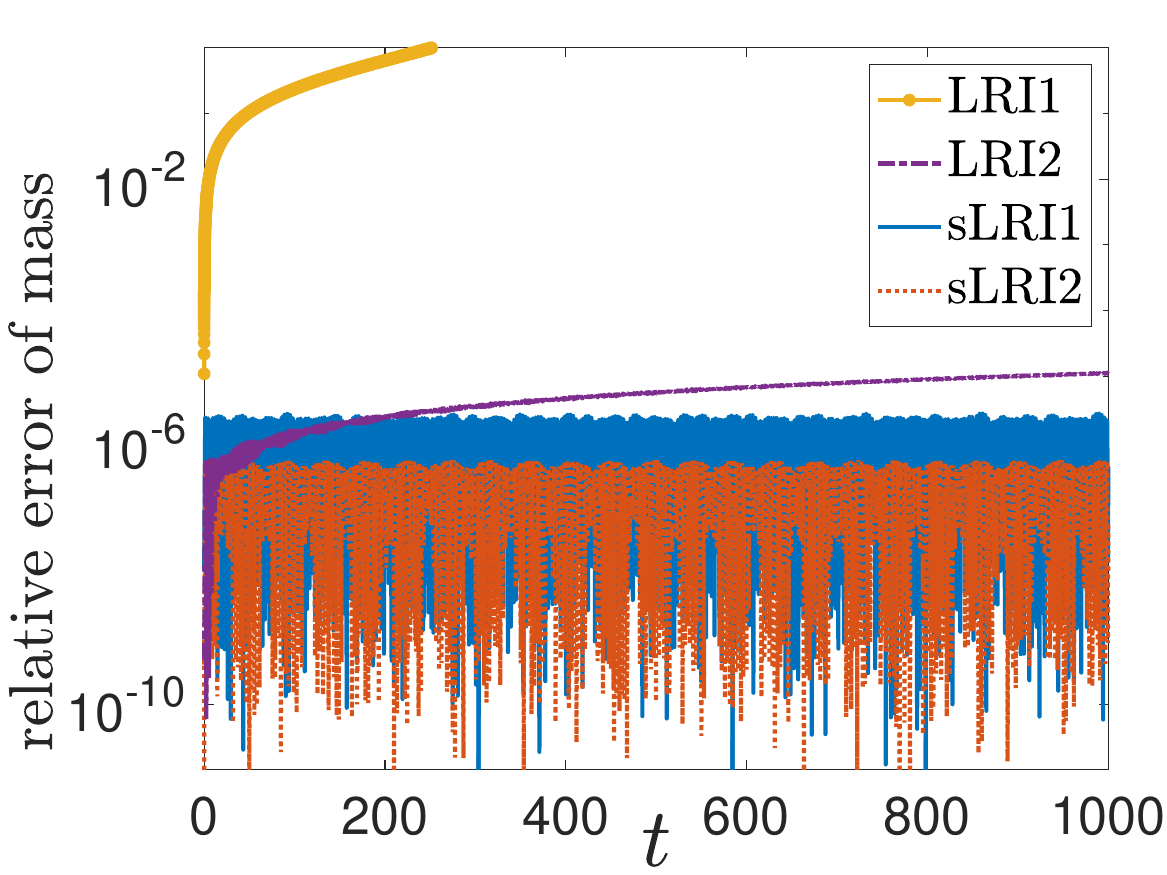}}\hspace{1em}
{\includegraphics[width=0.475\textwidth]{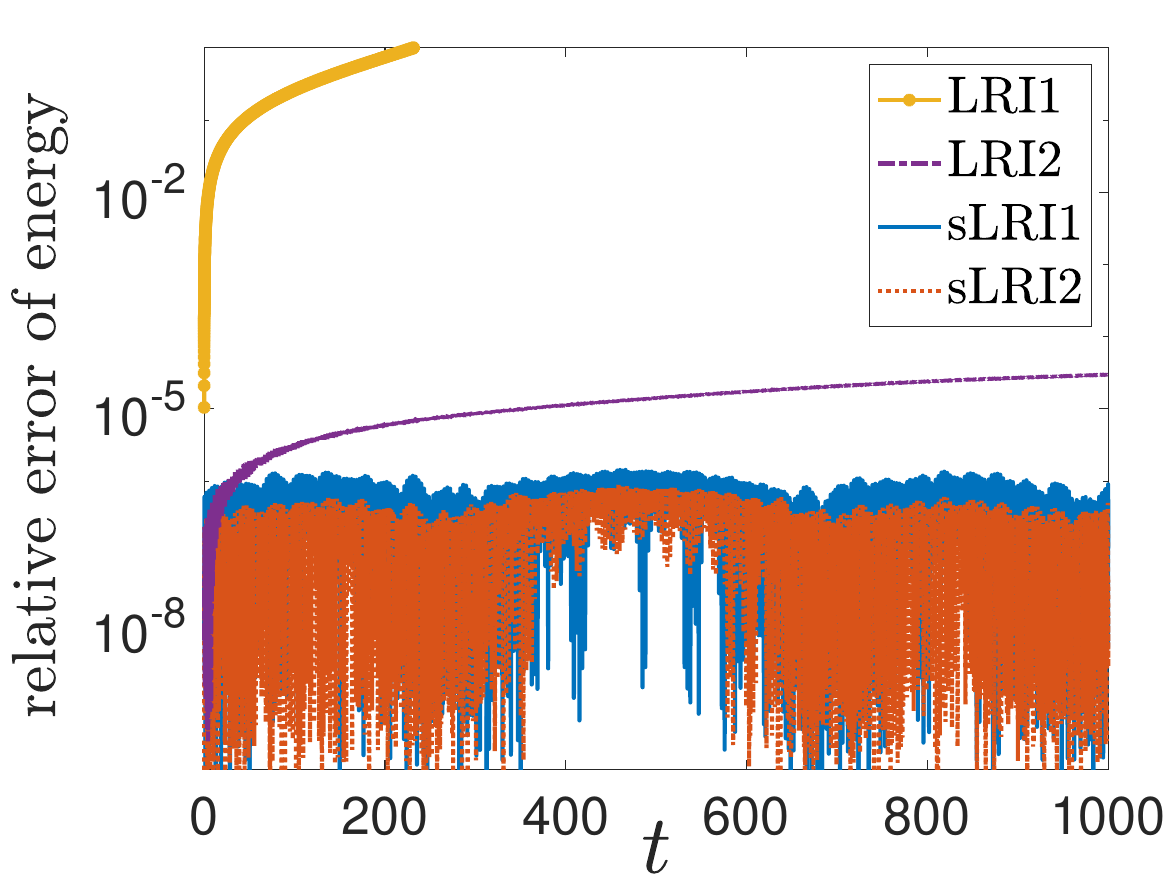}}
\caption{Relative errors of mass (left) and energy (right) of different methods for the NLSE \cref{NLSE} with $H^{2.5}$ initial datum}
\label{fig:mass_energy_H25_rand_ini_comp}
\end{figure}

\revisions{\begin{remark}\label{rem:quintic}
	As mentioned above, the cubic NLSE in one dimension is completely integrable. We present in this remark an additional numerical test for the non-integrable quintic NLSE 
	\begin{equation}\label{eq:qNLSE}
		i \partial_t u = -\Delta u + \mu |u|^4u, \qquad t>0, \quad x \in \mathbb{T}, 
	\end{equation} 
	under an $H^2$-initial data given by \cref{eq:ini}. Similarly, we compute the relative errors of mass and energy for the non-symmetric LRI1 and the symmetrized sLRI1. The LRI1 for the quintic NLSE \cref{eq:qNLSE} can be found in \cite{ostermann_schratz2018low}, and the corresponding sLRI1 can be obtained by \cref{eqn:general_scheme}. The numerical results are presented in \cref{fig:mass_energy_H2_quintic}, where we can still observe the near conservation of mass and energy of the sLRI1.
        \begin{figure}[htbp]
		\centering
		{\includegraphics[width=0.475\textwidth]{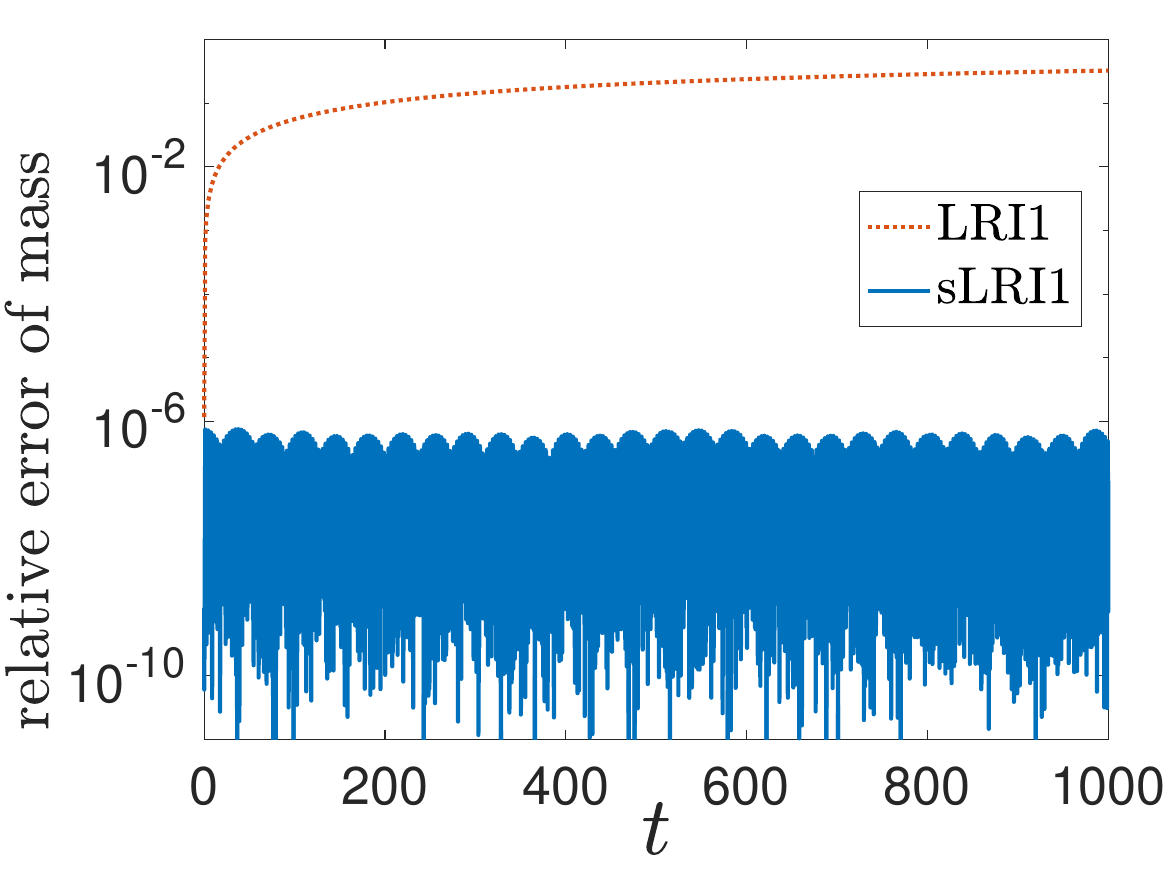}}\hspace{1em}
		{\includegraphics[width=0.475\textwidth]{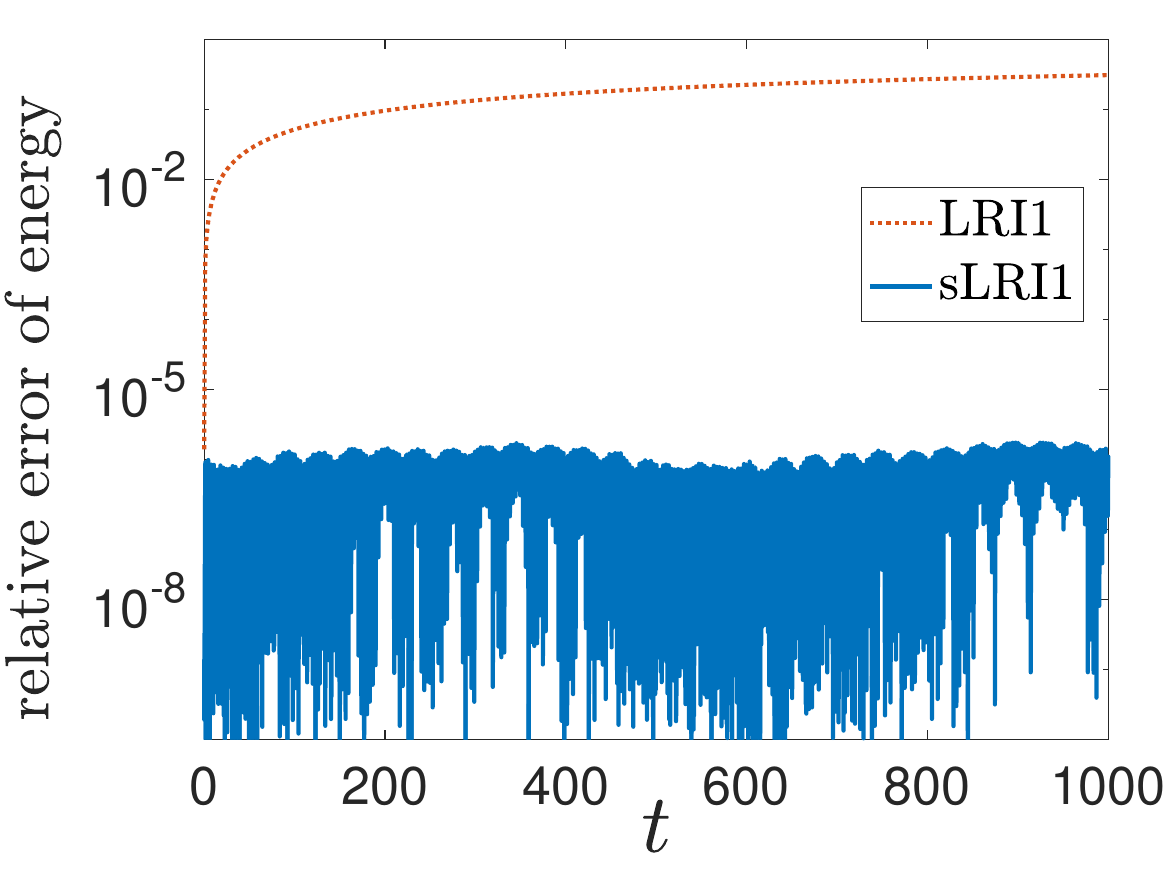}}
		\caption{\revisions{Relative errors of mass (left) and energy (right) of LRI1 and sLRI1 for the quintic NLSE \cref{eq:qNLSE} with $H^{2}$ initial datum}}
		\label{fig:mass_energy_H2_quintic}
	\end{figure}
\end{remark}}

Finally, to demonstrate the advantages of explicity, we present comparisons with existing implicit symmetric LRIs with two typical examples: IMsLRI1 given by (5) in \cite{alama2023symmetric} and IMsLRI2 given by (4.18) in \cite{alamabronsard_et_al23}. \revisions{In this example, we solve the NLSE \cref{NLSE} with low-regularity initial data that are physically more meaningful, i.e. the low-regularity initial data are chosen as stationary states of the NLSE with low-regularity potentials.} To be precise, we consider a scaled torus $\widetilde{\mathbb{T}} := (-16, 16)$ and consider the (action) ground states of the time-independent NLSE \cite{wang2022,liu2023}
\begin{equation}\label{sNLSE}
	-\Delta \phi(x) + V(x) \phi(x) - |\phi(x)|^2 \phi(x) + \omega \phi(x) = 0, \quad x \in \widetilde{\mathbb{T}},  
\end{equation}
where $V = V(x) \in \R$ is a (low regularity) potential and $\omega \in \R$ is a given constant. The following choices of $V$ and $\omega$ will be used:  
\begin{equation}
	\begin{aligned}
		&V_1 = -\delta(x),  \quad &&V_2(x) = -\frac{1}{\sqrt{x}}, \quad &&V_3(x) = -10 \times \mathbf{1}_{[-2, 2]}(x), \qquad x\in\widetilde{\mathbb{T}} \\
		&\omega_1 = 4,  \quad &&\omega_2 = 6, \quad &&\omega_3 = 14, 
	\end{aligned}
\end{equation}
where $\delta$ is the Dirac delta function and $\mathbf{1}_\Omega$ is the indicator function of a set $\Omega$. Let $\phi_j \ (j=1, 2, 3)$ be the action ground states of \cref{sNLSE} with $V = V_j$ and $\omega = \omega_j$ (plotted in \cref{fig:V_phi}), which can be computed by the standard discrete normalized gradient flow method \cite{wang2022}. According to the regularity of the potentials $V_j$, one has, roughly, $\phi_j \in H^{1+j/2}$ for $j=1, 2, 3$.
\begin{figure}[h!]
	\centering
	{\includegraphics[width=0.325\textwidth]{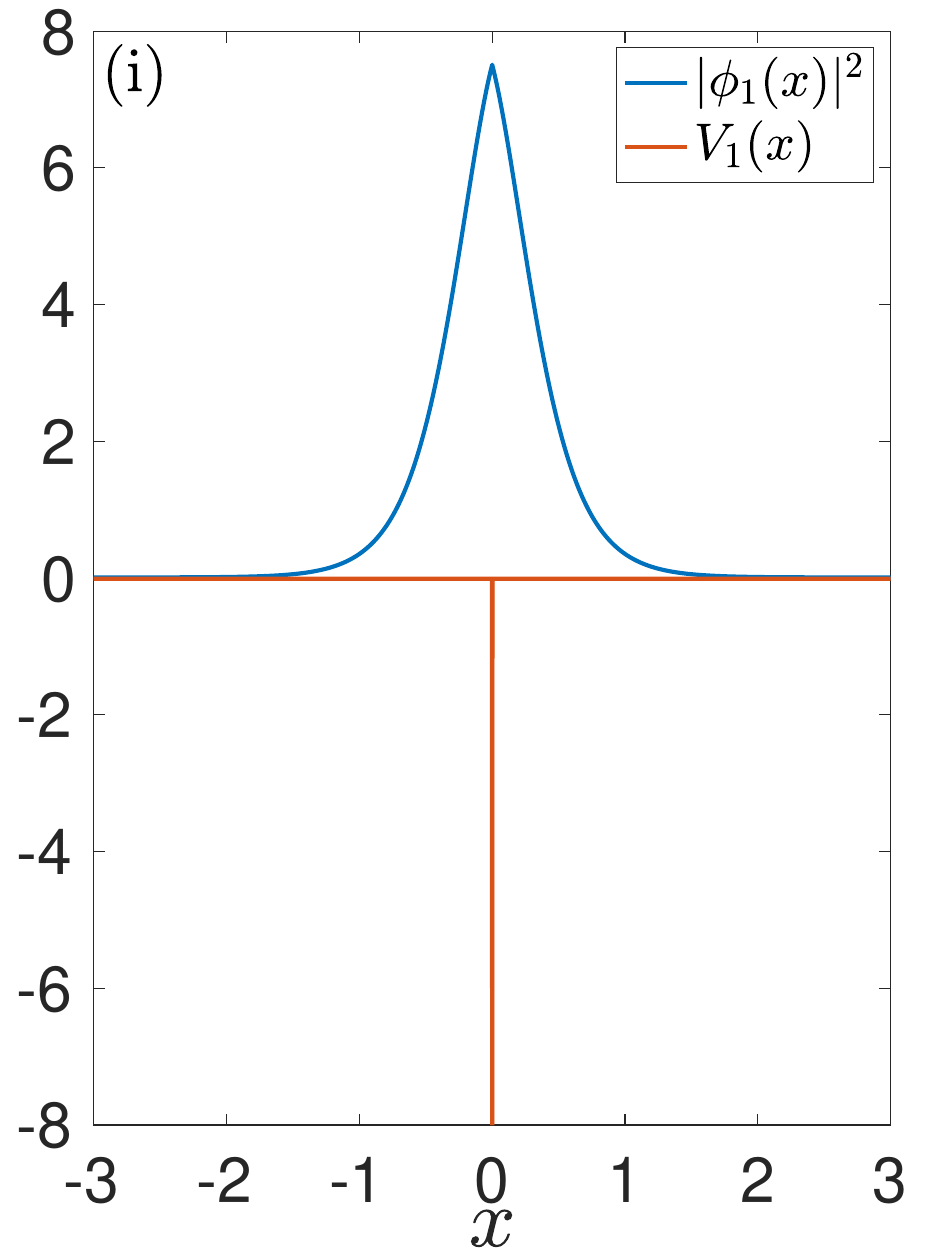}}\hspace{0.1em}
	{\includegraphics[width=0.325\textwidth]{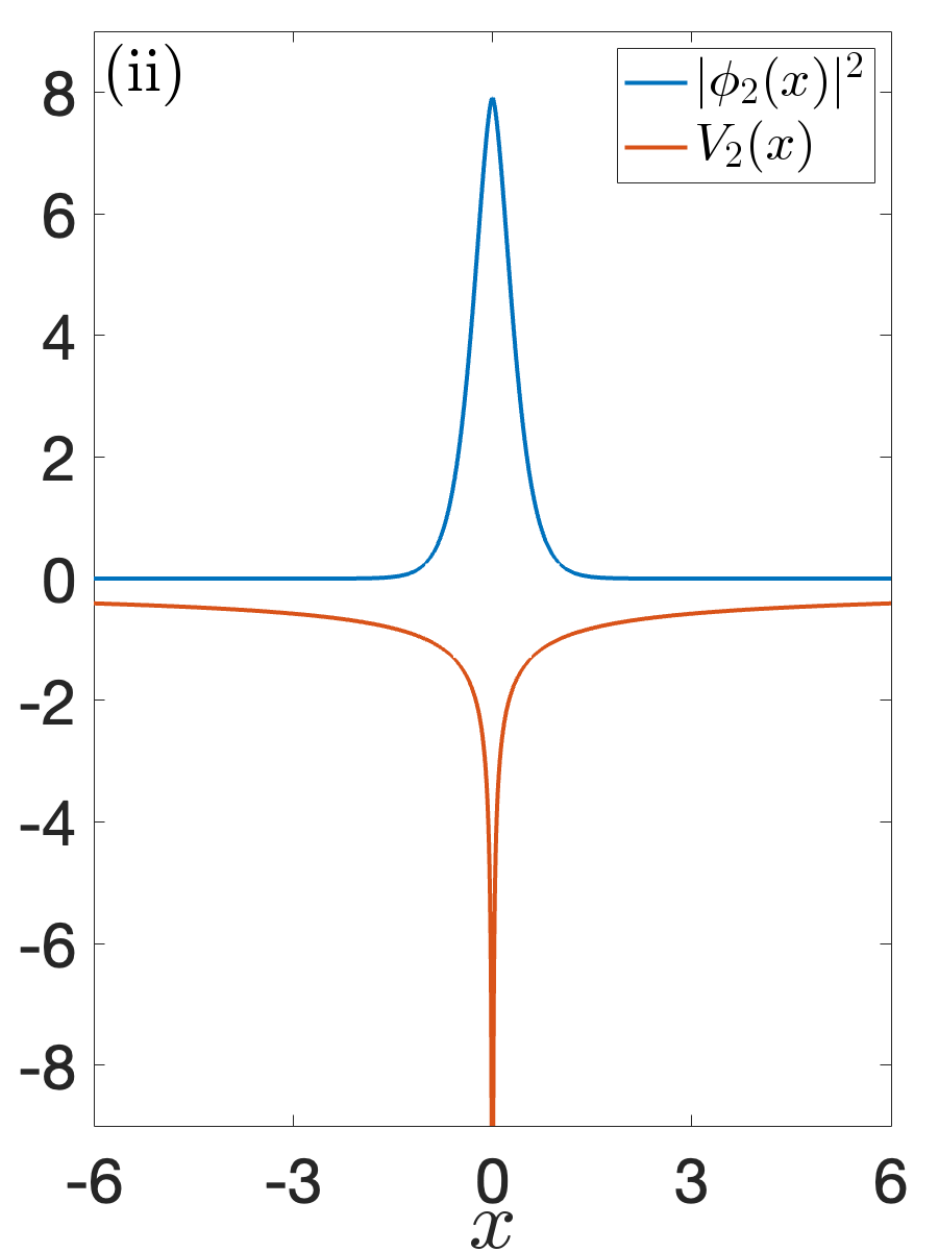}}\hspace{0.1em}
	{\includegraphics[width=0.325\textwidth]{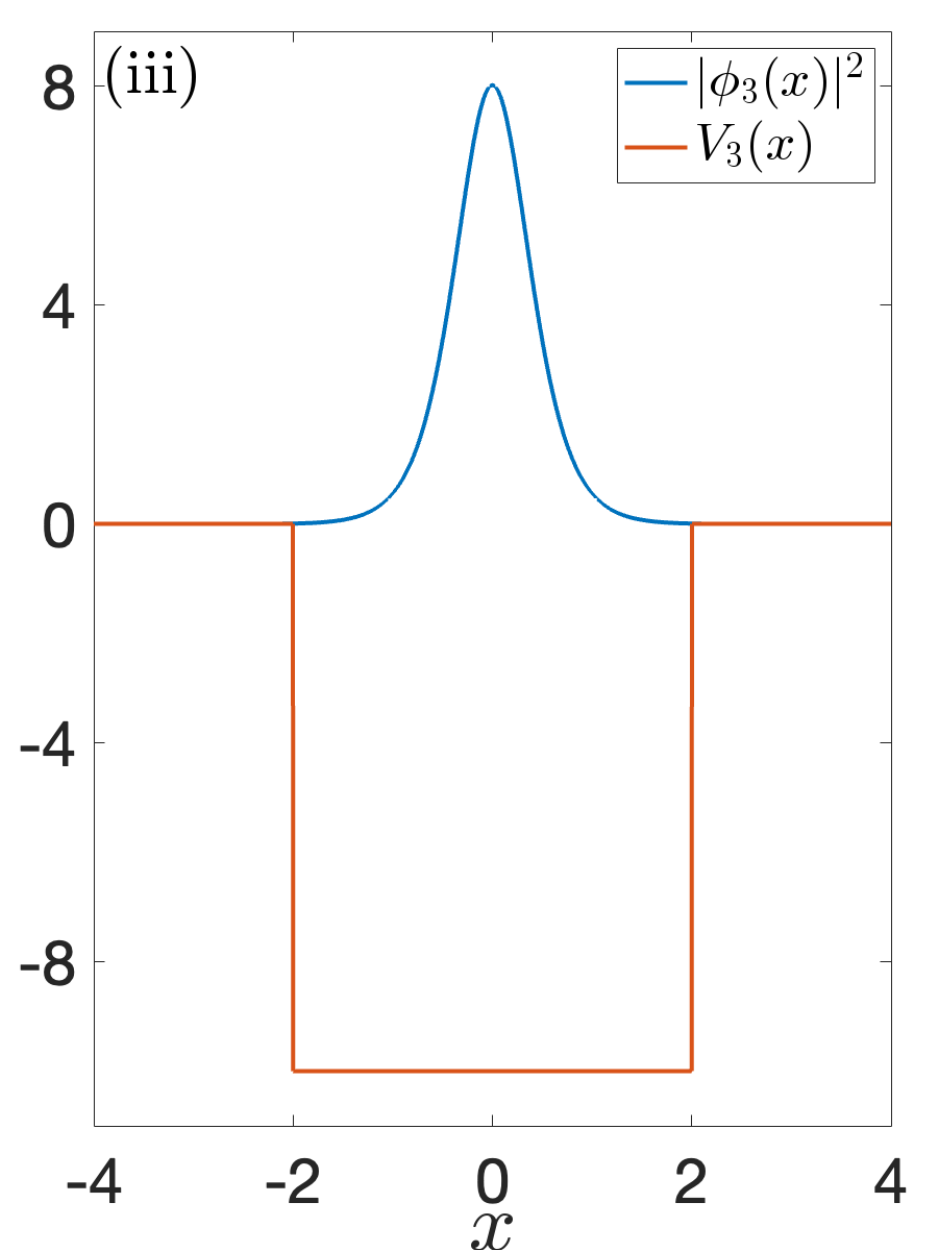}}
	\caption{Plots of the potential $V_j$ and the corresponding ground state $|\phi_j|^2$ for $j = 1, 2, 3$}
	\label{fig:V_phi}
\end{figure}

Then we apply the explicit sLRI1 and sLRI2 methods, and the implicit IMsLRI1 and IMsLRI2 methods to solve the NLSE \cref{NLSE} \revisions{(note that we do not solve the NLSE with potentials here)}. We show the results of three cases: (i) $u_0 = \phi_1, \mu = -1$; (ii) $u_0 =  \phi_2, \mu = -2 $; (iii) $u_0 = \phi_3, \mu = -1$. In the computation, the fixed-point iteration is adopted for the implicit methods to solve the nonlinear systems up to an accuracy of $10^{-9}$. The numerical results are shown in \cref{fig:conv_dt_run_time_H1,fig:conv_dt_run_time_H2,fig:conv_dt_run_time_H3} for the three cases, respectively, where, due to the significantly different computational costs of explicit and implicit methods, we plot the errors versus the computational time for a fair comparison. 

From the numerical results, we see that the two explicit sLRIs perform better than the implicit ones in all the cases in the sense that they can obtain smaller errors with the same computational cost. {The superiority of explicit methods becomes more obvious as the regularity of the initial data (though remains low) increases.} Moreover, in this metric, the explicit sLRI1 method is even better than the explicit sLRI2 method due to the simplicity of the scheme which requires fewer FFTs. These observations confirm the advantages of our explicit sLRIs over existing implicit ones. {It should be noted that implementing fully implicit methods, such as IMsLRI1 and IMsLRI2, is significantly more involved, requiring careful selection of the nonlinear solver, initial guess, and error tolerance. Moreover, their performance is sensitive to many factors including the time step sizes, the nonlinearity strength, the \revisions{underlying} exact solutions, etc.}

\begin{figure}[htbp]
	\centering
	{\includegraphics[width=0.475\textwidth]{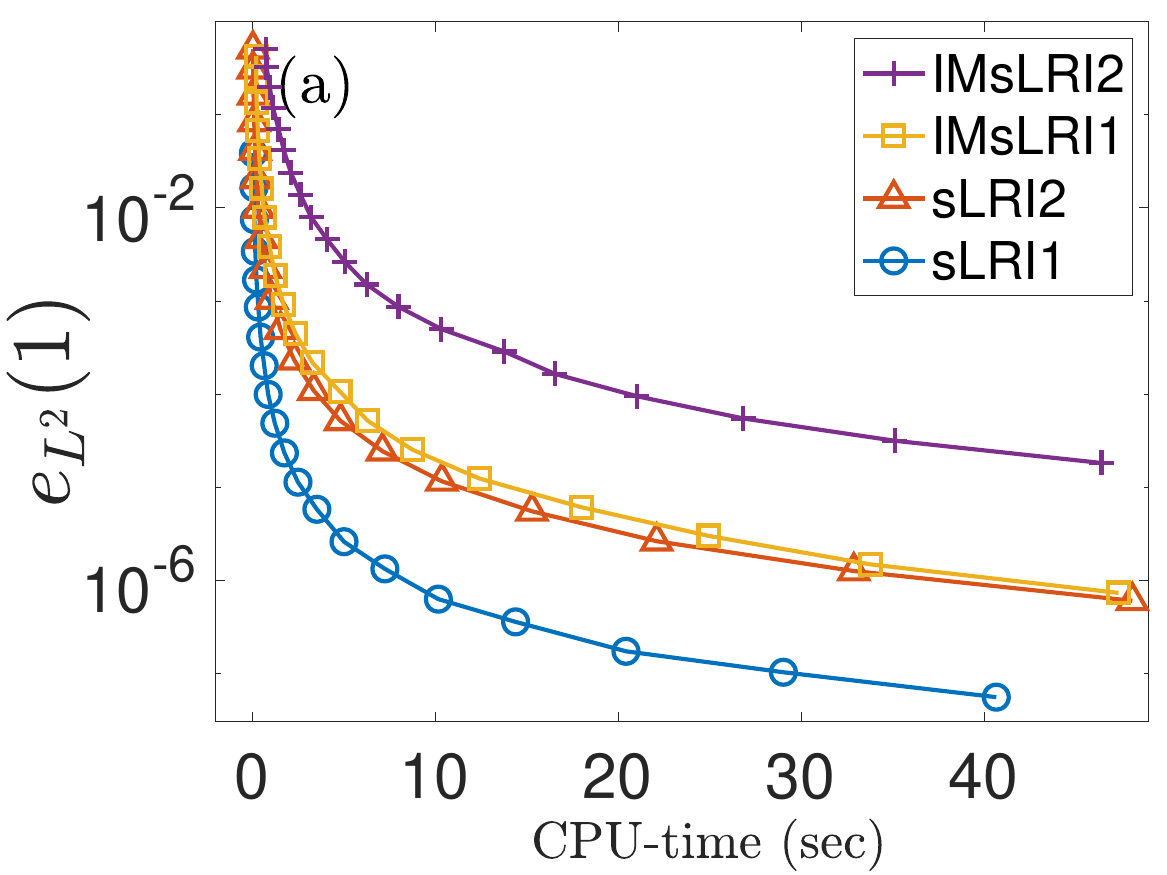}}\hspace{1em}
	{\includegraphics[width=0.475\textwidth]{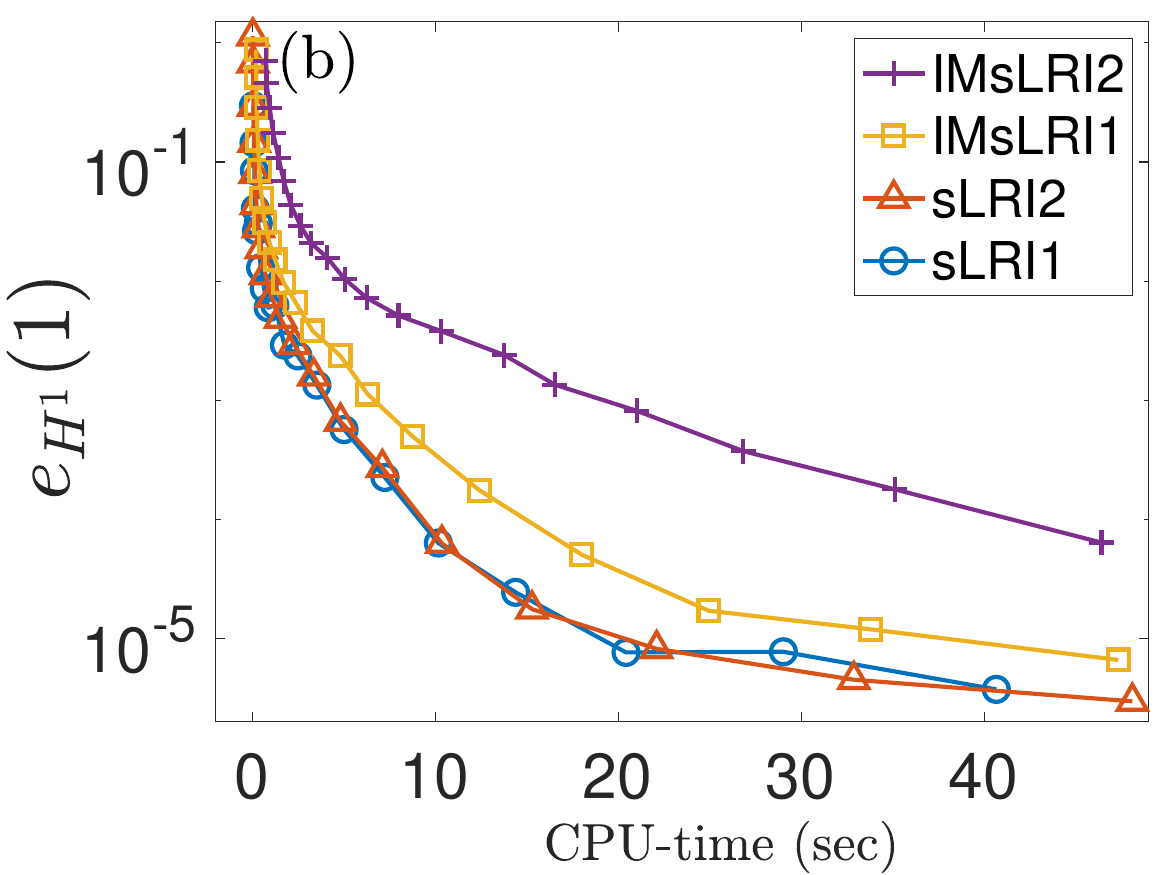}}
	\caption{Comparison of errors versus CPU time for the sLRI and the implicit sLRI under the NLSE \cref{NLSE} on $\widetilde{\mathbb T}$ with initial datum $u_0 = \phi_1$ and $\mu = -1$}
	\label{fig:conv_dt_run_time_H1}
\end{figure}
\begin{figure}[htbp]
	\centering
	{\includegraphics[width=0.475\textwidth]{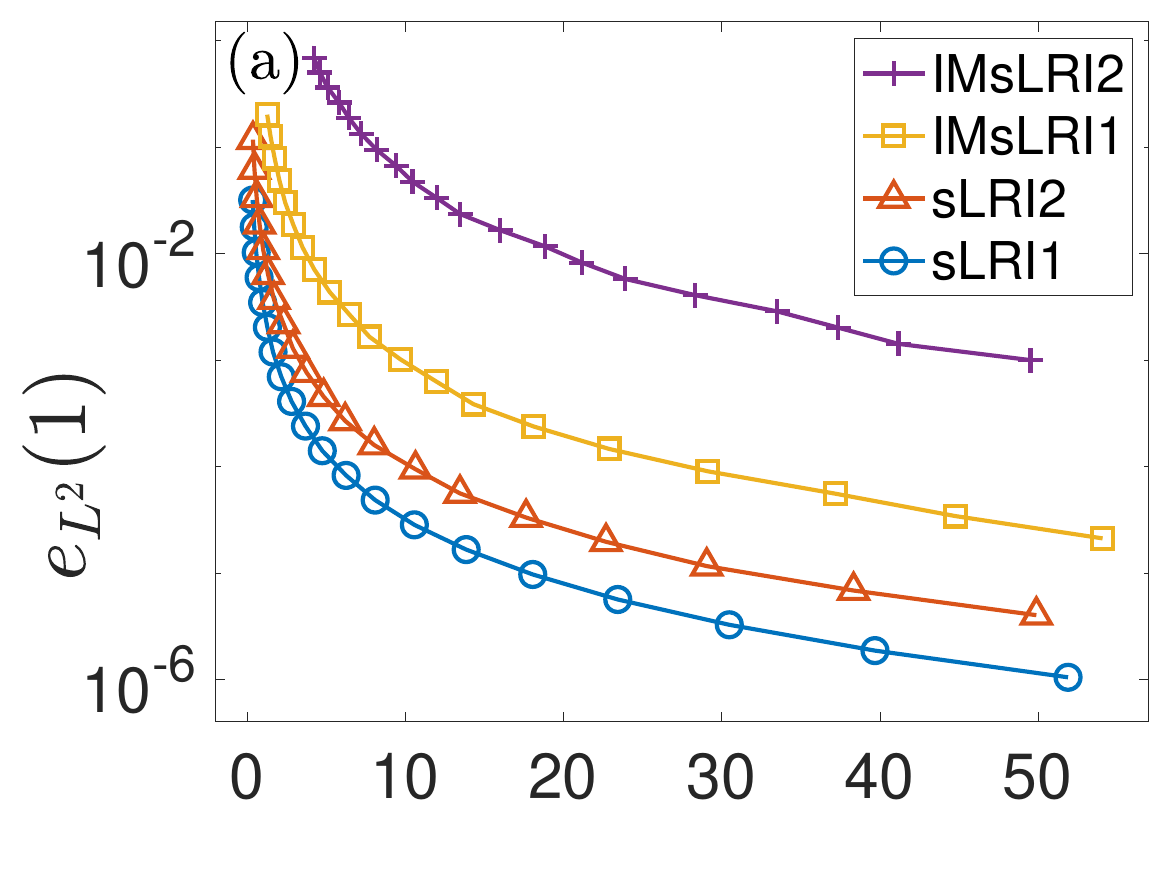}}\hspace{1em}
	{\includegraphics[width=0.475\textwidth]{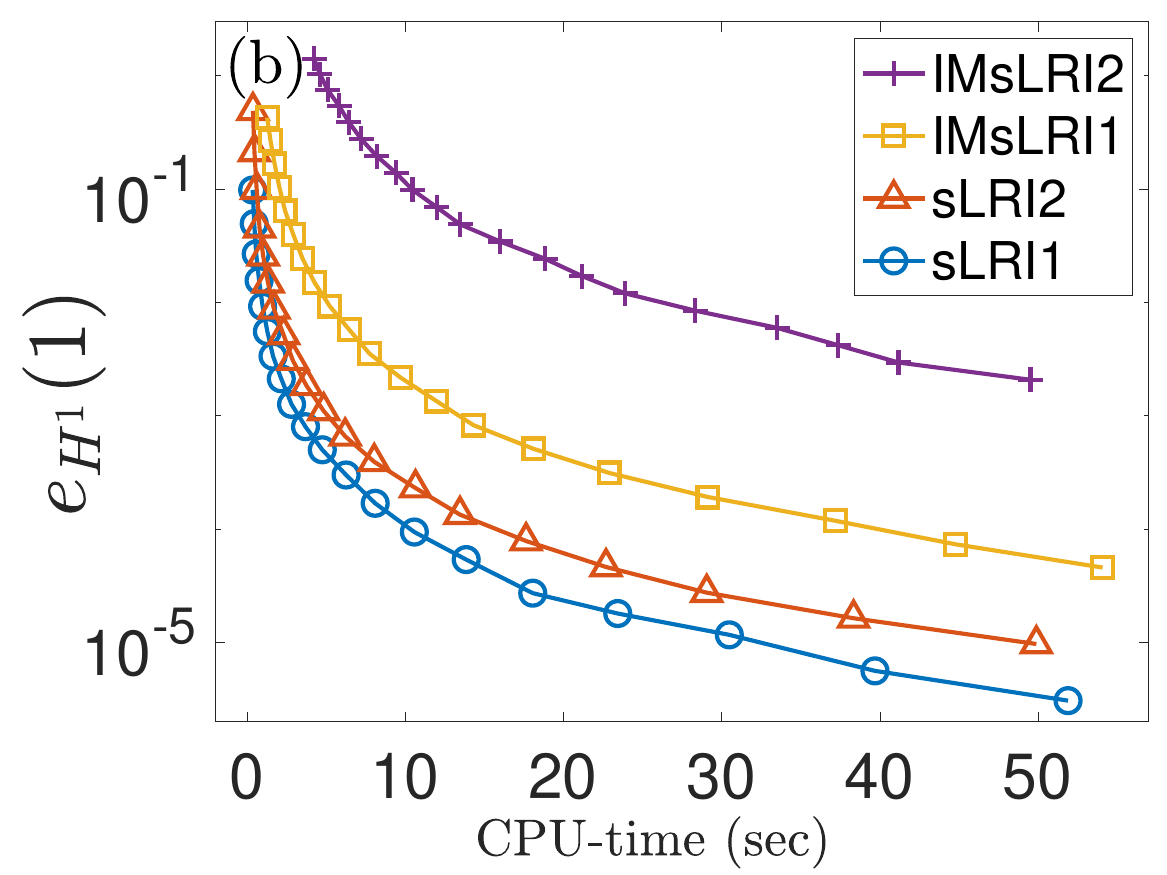}}
	\caption{Comparison of errors versus CPU time for the sLRI and the implicit sLRI under the NLSE \cref{NLSE} on $\widetilde{\mathbb T}$ with initial datum $u_0 = \phi_2$ and $\mu = -2$}
	\label{fig:conv_dt_run_time_H2}
\end{figure}
\begin{figure}[htbp]
	\centering
	{\includegraphics[width=0.475\textwidth]{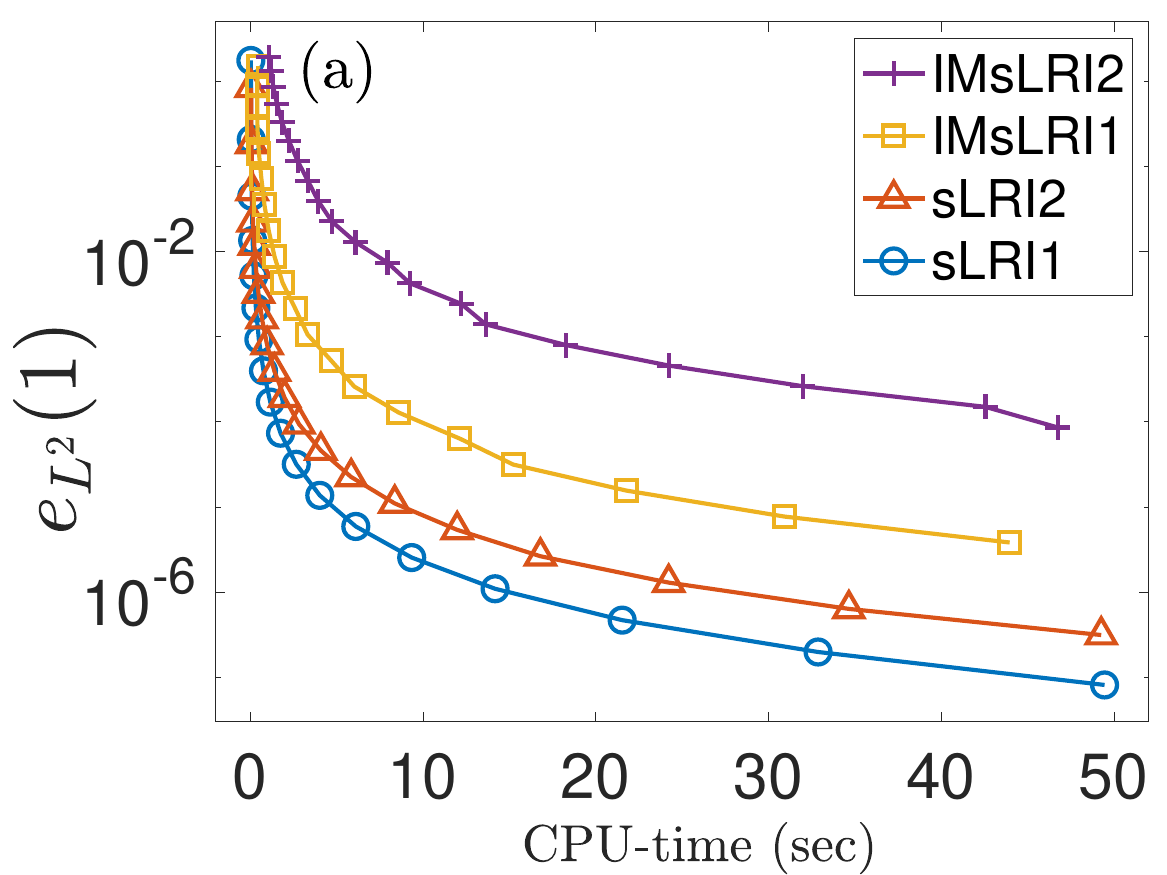}}\hspace{1em}
	{\includegraphics[width=0.475\textwidth]{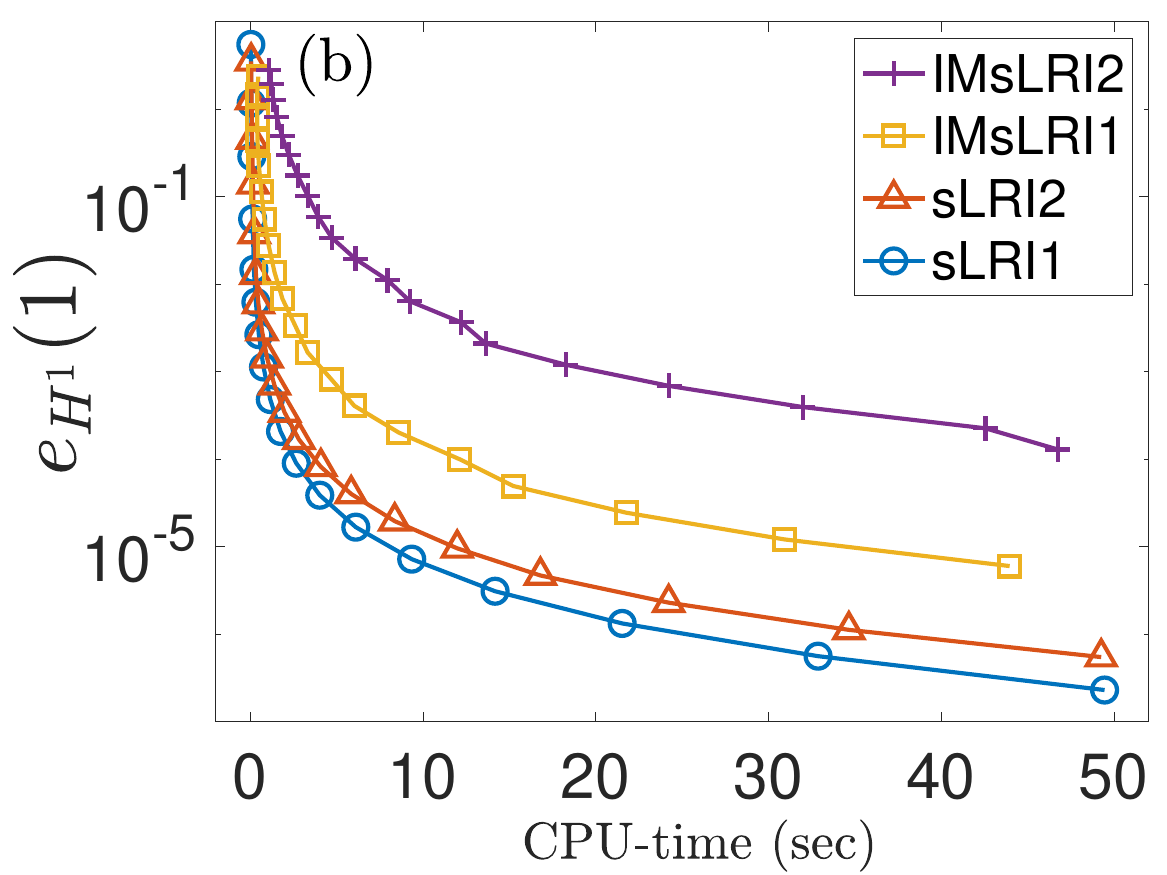}}
	\caption{Comparison of errors versus CPU time for the sLRI and the implicit sLRI under the NLSE \cref{NLSE} on $\widetilde{\mathbb T}$ with initial datum $u_0 = \phi_3$ and $\mu = -1$}
	\label{fig:conv_dt_run_time_H3}
\end{figure}

In conclusion, the newly proposed explicit sLRIs are both accurate and efficient under low regularity initial data. Moreover, they demonstrate excellent long-time performance. 

\section{Conclusion}\label{sec:conclusions} 
In this paper, we introduced the first fully explicit symmetric low-regularity integrators for the nonlinear Schr\"odinger equation. We demonstrated that, using a {multi-step} construction, we can overcome a significant amount of the computational cost incurred by implicit single step symmetric low-regularity schemes, while preserving the favourable structure preserving behaviour of symmetric methods and the low-regularity convergence properties of our schemes. The general construction extends to a wide class of numerical schemes, and is, in principle not limited to the NLSE. Thus we believe that in very similar vain one will be able to construct fully explicit symmetric low-regularity integrators for a large class of dispersive nonlinear systems and this will be investigated in future work.

\appendix
\section{{Proof of \cref{eq:T1T2_est}}}
\begin{proof}
When $\gamma > d/4$, by \cref{eq:exp_est} and \cref{eq:bi3} with $\gamma=0$, we have
\begin{align}
	\| T_1 \|_{L^2} 
	&\lesssim \| w(t) - w(-t)\|_{L^2} \| \nabla w_1 \|_{H^{\frac{d}{2}+\vep}} \| \nabla w_2 \|_{H^{\frac{d}{2}+\vep}} \notag \\
	&\lesssim \tau^{\gamma} \| u(t_n) \|_{H^{2\gamma}} \| w_1 \|_{H^{\frac{d}{2} + \vep + 1}} \| w_2 \|_{H^{\frac{d}{2} + \vep + 1}} \notag \\
	&\leq \tau^{\gamma}C(M_{2\gamma}, M_{\frac{d}{2} + \vep + 1})  \leq \tau^{\gamma}C(M_{2\gamma+1}),\end{align}
    \begin{align}
	\| T_2 \|_{L^2} 
	&\lesssim \|  w_1 \|_{H^{\frac{d}{2}+\vep}} \| \nabla (w(t) - w(-t))\|_{L^2} \| \nabla w_2 \|_{H^{\frac{d}{2}+\vep}} \notag \\
	&\lesssim \tau^{\gamma} \| w_1 \|_{H^{\frac{d}{2}+\vep}} \| u(t_n) \|_{H^{2\gamma+1}} \| w_2 \|_{H^{\frac{d}{2} + \vep + 1}} \notag \\
    &\leq \tau^{\gamma}C(M_{2\gamma+1}, M_{\frac{d}{2}+\vep + 1}) \leq \tau^\gamma C(M_{2\gamma+1}).  
\end{align}
When $\gamma < d/4$, for $T_1$, by $\|v w\|_{L^2} \lesssim \| v \|_{H^{1+\frac{d}{4}-\gamma}} \| w \|_{H^{2\gamma}} \ (0 \leq \gamma \leq 1)$ and \cref{eq:bi2}, 
\begin{align}
	\| T_1 \|_{L^2} 
	&\lesssim \| w(t) - w(-t)\|_{1+\frac{d}{4} - \gamma} \| \nabla w_1 \cdot \nabla w_2 \|_{H^{2\gamma}} \notag \\
	&\lesssim \tau^{\gamma} \| u(t_n) \|_{H^{1+\frac{d}{4}+\gamma}} \| \nabla w_1 \|_{H^{\frac{d}{4} + \gamma}} \| \nabla w_2 \|_{H^{\frac{d}{4} + \gamma}} 
	\leq \tau^{\gamma}C(M_{1+\frac{d}{4} + \gamma}),  
\end{align}
and for $T_2$, by \cref{eq:bi1}, \cref{eq:exp_est}, and \cref{eq:bi3} with $\gamma=0$, we have
\begin{align*}
	\| T_2 \|_{L^2} 
	&\lesssim \|  w_1 \|_{H^{\frac{d}{2}+\vep}} \| \nabla (w(t) - w(-t)) \|_{H^{\frac{d}{4} - \gamma}} \| \nabla w_2 \|_{H^{\frac{d}{4}+\gamma}}  \\
	&\lesssim \tau^{\gamma} \| w_1 \|_{H^{\frac{d}{2}+\vep}} \| u(t_n) \|_{H^{1+\frac{d}{4}+\gamma}} \| w_2 \|_{H^{1+\frac{d}{4} + \gamma}} 
	\leq \tau^{\gamma}C(M_{1+\frac{d}{4}+\gamma}).
\end{align*}
When $\gamma = d/4$, by \cref{eq:exp_est}, \cref{eq:bi3} with $\gamma=0$, we have
\begin{align*}
&\| T_1 \|_{L^2} \lesssim \| w(t) - w(-t) \|_{L^2} \| \nabla w_1 \|_{H^{\frac{d}{2}+\vep}} \| \nabla w_2 \|_{H^{\frac{d}{2}+\vep}} \leq \tau^\gamma C(M_{1 + \frac{d}{2} + \vep}), \\
&\| T_2 \|_{L^2} \lesssim \|  w_1 \|_{H^{\frac{d}{2}+\vep}} \| \nabla (w(t) - w(-t)) \|_{L^2} \| \nabla w_2 \|_{H^{\frac{d}{2}+\vep}} \leq \tau^\gamma C(M_{1 + \frac{d}{2} + \vep}). 
\end{align*}
\end{proof}

\section*{Acknowledgments}
The authors would like to thank Professor Weizhu Bao (National University of Singapore) for his valuable suggestions and comments and Professor Yvain Bruned (Universit\'e de Lorraine) and Professor Katharina Schratz (Sorbonne Universit\'e) for several helpful discussions explaining the use decorated tree expansions for the construction of low-regularity integrators as well as Professor Tobias Jahnke (Karlsruhe Institute of Technology) for interesting conversations about symmetric integrators. Part of this work was undertaken while the authors were participating in the IMS Young Mathematical Scientist Forum in Applied Mathematics at the Institute for Mathematical Sciences, National University of Singapore, in January 2024.

\bibliographystyle{siamplain}
\bibliography{references}

\begin{thebibliography}{10}

\bibitem{akrivis1993finite}
{\sc G.~D. Akrivis}, {\em {Finite difference discretization of the cubic
  Schr{\"o}dinger equation}}, IMA J. Numer. Anal., 13 (1993), pp.~115--124.

\bibitem{bronsard2023error}
{\sc Y.~Alama~Bronsard}, {\em {Error analysis of a class of semi-discrete
  schemes for solving the Gross--Pitaevskii equation at low regularity}}, J.
  Comput. Appl. Math., 418 (2023), p.~114632.

\bibitem{alama2023symmetric}
{\sc Y.~Alama~Bronsard}, {\em {A symmetric low-regularity integrator for the
  nonlinear Schr{\"o}dinger equation}}, IMA J. Numer. Anal., 44 (2024),
  pp.~3648--3682.

\bibitem{alamabronsard_et_al23}
{\sc Y.~Alama~Bronsard, Y.~Bruned, G.~Maierhofer, and K.~Schratz}, {\em
  Symmetric resonance based integrators and forest formulae}, arXiv preprint
  arXiv:2305.16737,  (2023).

\bibitem{antoine2013computational}
{\sc X.~Antoine, W.~Bao, and C.~Besse}, {\em {Computational methods for the
  dynamics of the nonlinear Schr{\"o}dinger/Gross--Pitaevskii equations}},
  Comput. Phys. Commun., 184 (2013), pp.~2621--2633.

\bibitem{bai2022constructive}
{\sc G.~Bai, B.~Li, and Y.~Wu}, {\em {A constructive low-regularity integrator
  for the one-dimensional cubic nonlinear Schr{\"o}dinger equation under the
  Neumann boundary condition}}, IMA J. Numer. Anal., 43 (2023), pp.~3243--3281.

\bibitem{banicamaierhoferschratz22}
{\sc V.~Banica, G.~Maierhofer, and K.~Schratz}, {\em {Numerical integration of
  Schr{\"o}dinger maps via the Hasimoto transform}}, SIAM J. Numer. Anal., 62
  (2024), pp.~322--352.

\bibitem{bao2012mathematical}
{\sc W.~Bao and Y.~Cai}, {\em {Mathematical theory and numerical methods for
  Bose--Einstein condensation}}, Kinet. Relat. Models, 6 (2013), pp.~1--135.

\bibitem{bao2013optimal}
{\sc W.~Bao and Y.~Cai}, {\em {Optimal error estimates of finite difference
  methods for the Gross--Pitaevskii equation with angular momentum rotation}},
  Math. Comp., 82 (2013), pp.~99--128.

\bibitem{bao2014uniform}
{\sc W.~Bao and Y.~Cai}, {\em {Uniform and optimal error estimates of an
  exponential wave integrator sine pseudospectral method for the nonlinear
  Schrödinger equation with wave operator}}, SIAM J. Numer. Anal., 52 (2014),
  pp.~1103--1127.

\bibitem{bao2023improved}
{\sc W.~Bao, Y.~Cai, and Y.~Feng}, {\em {Improved uniform error bounds of the
  time-splitting methods for the long-time (nonlinear) Schr{\"o}dinger
  equation}}, Math. Comp., 92 (2023), pp.~1109--1139.

\bibitem{bao2003numerical}
{\sc W.~Bao, D.~Jaksch, and P.~A. Markowich}, {\em {Numerical solution of the
  Gross--Pitaevskii equation for Bose--Einstein condensation}}, J. Comput.
  Phys., 187 (2003), pp.~318--342.

\bibitem{bao2024explicit}
{\sc W.~Bao and C.~Wang}, {\em {An explicit and symmetric exponential wave
  integrator for the nonlinear Schr{\"o}dinger equation with low regularity
  potential and nonlinearity}}, SIAM J. Numer. Anal., 62 (2024),
  pp.~1901--1928.

\bibitem{besse2002order}
{\sc C.~Besse, B.~Bid{\'e}garay, and S.~Descombes}, {\em {Order estimates in
  time of splitting methods for the nonlinear Schr{\"o}dinger equation}}, SIAM
  J. Numer. Anal., 40 (2002), pp.~26--40.

\bibitem{bruned_schratz_2022}
{\sc Y.~Bruned and K.~Schratz}, {\em Resonance-based schemes for dispersive
  equations via decorated trees}, Forum Math. Pi, 10 (2022), pp.~e2 1--76.

\bibitem{cao2024new}
{\sc J.~Cao, B.~Li, and Y.~Lin}, {\em {A new second-order low-regularity
  integrator for the cubic nonlinear Schr{\"o}dinger equation}}, IMA J. Numer.
  Anal., 44 (2024), pp.~1313--1345.

\bibitem{celledoni2008symmetric}
{\sc E.~Celledoni, D.~Cohen, and B.~Owren}, {\em {Symmetric exponential
  integrators with an application to the cubic Schr{\"o}dinger equation}},
  Found. Comput. Math., 8 (2008), pp.~303--317.

\bibitem{cohen2008conservation}
{\sc D.~Cohen, E.~Hairer, and C.~Lubich}, {\em Conservation of energy, momentum
  and actions in numerical discretizations of non-linear wave equations},
  Numer. Math., 110 (2008), pp.~113--143.

\bibitem{faou2012geometric}
{\sc E.~Faou}, {\em {Geometric numerical integration and Schr{\"o}dinger
  equations}}, vol.~15, European Mathematical Society, 2012.

\bibitem{feng_maierhofer_schratz_2023}
{\sc Y.~Feng, G.~Maierhofer, and K.~Schratz}, {\em Long-time error bounds of
  low-regularity integrators for nonlinear {S}chr\"odinger equations}, Math.
  Comp., 93 (2023), pp.~1569--1598.

\bibitem{gauckler2010splitting}
{\sc L.~Gauckler and C.~Lubich}, {\em {Splitting integrators for nonlinear
  Schr{\"o}dinger equations over long times}}, Found. Comput. Math., 10 (2010),
  pp.~275--302.

\bibitem{hairer2004symmetric}
{\sc E.~Hairer and C.~Lubich}, {\em Symmetric multistep methods over long
  times}, Numer. Math., 97 (2004), pp.~699--723.

\bibitem{hairer2013geometric}
{\sc E.~Hairer, C.~Lubich, and G.~Wanner}, {\em {Geometric Numerical
  Integration: Structure-Preserving Algorithms for Ordinary Differential
  Equations}}, Springer, 2013.

\bibitem{hochbruck2010exponential}
{\sc M.~Hochbruck and A.~Ostermann}, {\em Exponential integrators}, Acta
  Numer., 19 (2010), pp.~209--286.

\bibitem{jahnke2023numerical}
{\sc T.~Jahnke and M.~Kirn}, {\em {On numerical methods for the
  semi-nonrelativistic limit system of the nonlinear Dirac equation}}, BIT
  Numer. Math., 63 (2023), p.~26.

\bibitem{jahnke2018adiabatic}
{\sc T.~Jahnke and M.~Mikl}, {\em {Adiabatic midpoint rule for the
  dispersion-managed nonlinear Schr{\"o}dinger equation}}, Numer. Math., 138
  (2018), pp.~975--1009.

\bibitem{jahnke2019adiabatic}
{\sc T.~Jahnke and M.~Mikl}, {\em {Adiabatic exponential midpoint rule for the
  dispersion-managed nonlinear Schr{\"o}dinger equation}}, IMA J. Numer. Anal.,
  39 (2019), pp.~1818--1859.

\bibitem{li2021fully}
{\sc B.~Li and Y.~Wu}, {\em {A fully discrete low-regularity integrator for the
  1D periodic cubic nonlinear Schr{\"o}dinger equation}}, Numer. Math., 149
  (2021), pp.~151--183.

\bibitem{liu2023}
{\sc W.~Liu, Y.~Yuan, and X.~Zhao}, {\em Computing the action ground state for
  the rotating nonlinear {S}chr\"odinger equation}, SIAM J. Sci. Comput., 45
  (2023), pp.~A397--A426.

\bibitem{lubich2008splitting}
{\sc C.~Lubich}, {\em {On splitting methods for {Schr{\"o}dinger-Poisson} and
  cubic nonlinear {Schr{\"o}dinger} equations}}, Math. Comp., 77 (2008),
  pp.~2141--2153.

\bibitem{maierhofer_schratz_24}
{\sc G.~Maierhofer and K.~Schratz}, {\em {Bridging the gap: symplecticity and
  low regularity in Runge-Kutta resonance-based schemes}}, arXiv preprint
  arXiv:2205.05024,  (2022).

\bibitem{mclachlan2002splitting}
{\sc R.~I. McLachlan and G.~R.~W. Quispel}, {\em Splitting methods}, Acta
  Numer., 11 (2002), pp.~341--434.

\bibitem{ostermann_schratz2018low}
{\sc A.~Ostermann and K.~Schratz}, {\em {Low regularity exponential-type
  integrators for semilinear Schr{\"o}dinger equations}}, Found. Comput. Math.,
  18 (2018), pp.~731--755.

\bibitem{ostermann2022second}
{\sc A.~Ostermann, Y.~Wu, and F.~Yao}, {\em {A second-order low-regularity
  integrator for the nonlinear Schr{\"o}dinger equation}}, Adv. Contin.
  Discrete Models, 2022 (2022), p.~23.

\bibitem{ostermann2022fully}
{\sc A.~Ostermann and F.~Yao}, {\em {A fully discrete low-regularity integrator
  for the nonlinear Schr{\"o}dinger equation}}, J. Sci. Comput., 91 (2022),
  p.~9.

\bibitem{sulem2007nonlinear}
{\sc C.~Sulem and P.-L. Sulem}, {\em The Nonlinear {S}chr{\"o}dinger Equation:
  Self-Focusing and Wave Collapse}, Springer, 1999.

\bibitem{tai1986observation}
{\sc K.~Tai, A.~Hasegawa, and A.~Tomita}, {\em Observation of modulational
  instability in optical fibers}, Phys. Rev. Lett., 56 (1986), pp.~135--139.

\bibitem{thomas2012nonlinear}
{\sc R.~Thomas, C.~Kharif, and M.~Manna}, {\em A nonlinear {S}chr{\"o}dinger
  equation for water waves on finite depth with constant vorticity}, Phys.
  Fluids, 24 (2012), p.~127102.

\bibitem{wang2022}
{\sc C.~Wang}, {\em Computing the least action ground state of the nonlinear
  {S}chr\"odinger equation by a normalized gradient flow}, J. Comput. Phys.,
  471 (2022), p.~111675.

\end{thebibliography}
\end{document}